\documentclass[11pt]{amsart}
\usepackage[margin=1in]{geometry} 

\usepackage{latexsym,amsmath,amsthm,amsfonts,amscd,amssymb,mathdots}
\usepackage[mathscr]{eucal} 
\usepackage{mathrsfs}
\usepackage[utf8]{inputenc}
\usepackage[all,cmtip]{xy}
\usepackage{hyperref}
\usepackage{graphics}
\usepackage{lscape}
\usepackage{enumerate}
\usepackage{array}
\usepackage{framed}
\usepackage{mathtools}
\usepackage{epsfig,color}
\usepackage{hyperref}
\usepackage{tikz}
\hypersetup{colorlinks}
\usepackage{tikz-cd}

\usepackage{hyperref}\hypersetup{colorlinks}

\usepackage{color} 

\definecolor{darkred}{rgb}{1,0,0} 
\definecolor{darkgreen}{rgb}{0,0.8,0}
\definecolor{darkblue}{rgb}{0,0,1}
\definecolor{customcolor}{rgb}{.5,.1,6}
\hypersetup{colorlinks,
linkcolor=darkblue,
filecolor=darkgreen,
urlcolor=darkred,
citecolor=darkred}
\definecolor{lime}{rgb}{0.1,1,0}

\newcommand{\zbar}{\ensuremath{\overline{z}}}

\newcommand{\cA}{\ensuremath{\mathcal A}}
\newcommand{\cB}{\ensuremath{\mathcal B}}

\newcommand{\cE}{\ensuremath{\mathcal E}}
\newcommand{\cF}{\ensuremath{\mathcal F}}

\newcommand{\cL}{\ensuremath{\mathcal L}}
\newcommand{\cM}{\ensuremath{\mathcal M}}
\newcommand{\cN}{\ensuremath{\mathcal N}}
\newcommand{\cO}{\ensuremath{\mathcal O}}
\newcommand{\cP}{\ensuremath{\mathcal P}}
\newcommand{\cS}{\ensuremath{\mathcal S}}
\newcommand{\cV}{\ensuremath{\mathcal V}}
\newcommand{\cW}{\ensuremath{\mathcal W}}
\newcommand{\sslash}{\mathbin{/\mkern-5.5mu/}}

\newcommand{\de}{\mathrm{d}}

 \numberwithin{equation}{section}

\newtheorem {Theorem}{Theorem}

\numberwithin{Theorem}{section}

\newtheorem {Lemma}[Theorem]    {Lemma}

\newtheorem {Proposition}[Theorem]{Proposition}
\newtheorem {Corollary}[Theorem]{Corollary}
\theoremstyle{definition}
\newtheorem{Definition}[Theorem]{Definition}

\newtheorem{Remark}[Theorem]{Remark}
\newtheorem{Example}[Theorem]{Example}

\newtheoremstyle{MyTheorem}
        {.6em}{.6em}              
        {\itshape}                      
        {}                              
        {\bfseries}                     
        {. }                             
        { }                             
        {\thmname{#1}\thmnumber{\addtocounter{MyTheorem}{-4}#2}\thmnote{ \bfseries #3}}
\theoremstyle{MyTheorem}
\newtheorem{MyTheorem}{Theorem}

\newtheoremstyle{TheoremForIntro} 
        {.6em}{.6em}              
        {\itshape}                      
        {}                              
        {\bfseries}                     
        {.}                             
        { }                             
        {\thmname{#1}\thmnote{ \bfseries #3}}
    \theoremstyle{TheoremForIntro}

    \newtheorem{Assumption}[MyTheorem]{Assumption}

\expandafter\chardef\csname pre amssym.def at\endcsname=\the\catcode`\@
\catcode`\@=11
\def\undefine#1{\let#1\undefined}
\def\newsymbol#1#2#3#4#5{\let\next@\relax
 \ifnum#2=\@ne\let\next@\msafam@\else
 \ifnum#2=\tw@\let\next@\msbfam@\fi\fi
 \mathchardef#1="#3\next@#4#5}
\def\mathhexbox@#1#2#3{\relax
 \ifmmode\mathpalette{}{\m@th\mathchar"#1#2#3}%
 \else\leavevmode\hbox{$\m@th\mathchar"#1#2#3$}\fi}
\def\hexnumber@#1{\ifcase#1 0\or 1\or 2\or 3\or 4\or 5\or 6\or 7\or 8\or
 9\or A\or B\or C\or D\or E\or F\fi}

\font\teneufm=eufm10
\font\seveneufm=eufm7
\font\fiveeufm=eufm5
\newfam\eufmfam
\textfont\eufmfam=\teneufm
\scriptfont\eufmfam=\seveneufm
\scriptscriptfont\eufmfam=\fiveeufm

\catcode`\@=\csname pre amssym.def at\endcsname

\newcommand{\supp}{\operatorname{supp}}

\newcommand{\fl}{{\mathfrak l}}

\newcommand{\sSL}{{\mathsf {SL}}}

\newcommand{\sGL}{{\mathsf {GL}}}

\newcommand{\sG}{{\mathsf G}}

\newcommand{\sP}{{\mathsf P}}
\newcommand{\sL}{{\mathsf L}}

\newcommand{\sU}{{\mathsf U}}

\newcommand{\Hod}{{\mathrm {Hod}}}

\newcommand{\DR}{{\mathsf{dR}}} 
\newcommand{\BB}{{\textsf{BB}}} 

\newcommand{\Hom}{{\mathrm{Hom}}}

\newcommand{\End}{{\mathrm{End}}}

\newcommand{\Id}{{\mathrm{Id}}}

\newcommand{\Tr}{{\mathrm{Tr}}}
\newcommand{\res}{{\mathrm{Res}}}

\newcommand{\Aa}{{\mathcal A}}

\newcommand{\Ll}{{\mathcal L}}

\newcommand{\Ee}{{\mathcal E}}

\newcommand{\Mm}{{\mathcal M}}
\newcommand{\Nn}{{\mathcal N}}

\newcommand{\Ss}{{\mathcal S}}

\def    \C      {{\mathbb C}}
\def    \R      {{\mathbb R}}

\def    \N      {{\mathbb N}}

\def    \P    {{\mathbb P}}
\def    \CP     {{\mathbb C}{\mathbb P}}

\def    \lra     {{\longrightarrow}}

\def    \rk     {\operatorname{rk}}

\def    \tr     {\operatorname{tr}}

\def    \U     {\operatorname{U}}

\def    \CL     {\operatorname{CL}}

\def    \H  {\operatorname{\scriptscriptstyle{H}}}

\newcommand{\An}{\xymatrix{ *{\circ} \ar@{-}[r]|*\dir{ } & *{\circ}\ar@{-}[r]&{\cdots}&*{\circ}\ar@{-}[l]|*\dir{ }\ar@{-}[r]|*\dir{ }&*{\circ} }}
\newcommand{\Anlabel}{\xymatrix@R=.25em{ *{\circ}\ar@<-1ex>@{}[d]^{\alpha_{1}} \ar@{-}[r]|*\dir{ } & *{\circ}\ar@<-1ex>@{}[d]^{\alpha_{2}}\ar@{-}[r]&{\cdots}&*{\circ}\ar@{-}[l]|*\dir{ }\ar@{-}[r]|*\dir{ }\ar@<-2ex>@{}[d]^{\alpha_{n-1}}&*{\circ}\ar@<-1ex>@{}[d]^{\alpha_{n}}\\&&&& }}
\newcommand{\Bn}{\xymatrix{ *{\circ} \ar@{-}[r]|*\dir{ } & *{\circ}\ar@{-}[r]&{\cdots}&*{\circ}\ar@{-}[l]|*\dir{ }\ar@{=}[r]|*\dir{>}&*{\circ} }}
\newcommand{\Bnlabel}{\xymatrix@R=.25em{ *{\circ}\ar@<-1ex>@{}[d]^{\alpha_{1}} \ar@{-}[r]|*\dir{ } & *{\circ}\ar@<-1ex>@{}[d]^{\alpha_{2}}\ar@{-}[r]&{\cdots}&*{\circ}\ar@{-}[l]|*\dir{ }\ar@{=}[r]|*\dir{>}\ar@<-2ex>@{}[d]^{\alpha_{n-1}}&*{\circ}\ar@<-1ex>@{}[d]^{\alpha_{n}}\\&&&& }}
\newcommand{\Cn}{\xymatrix{ *{\circ} \ar@{-}[r]|*\dir{ } & *{\circ}\ar@{-}[r]&{\cdots}&*{\circ}\ar@{-}[l]|*\dir{ }\ar@{=}[r]|*\dir{<}&*{\circ} }}
\newcommand{\Cnlabel}{\xymatrix@R=.25em{ *{\circ}\ar@<-1ex>@{}[d]^{\alpha_{1}} \ar@{-}[r]|*\dir{ } & *{\circ}\ar@<-1ex>@{}[d]^{\alpha_{2}}\ar@{-}[r]&{\cdots}&*{\circ}\ar@{-}[l]|*\dir{ }\ar@{=}[r]|*\dir{<}\ar@<-2ex>@{}[d]^{\alpha_{n-1}}&*{\circ}\ar@<-1ex>@{}[d]^{\alpha_{n}}\\&&&& }}
\newcommand{\Dn}{\xymatrix@R=.25em{&&&&*{\circ} \\ *{\circ} \ar@{-}[r]|*\dir{ } & *{\circ}\ar@{-}[r]&{\cdots}&*{\circ}\ar@{-}[l]|*\dir{ }\ar@{-}[ur]|*\dir{ }\ar@{-}[dr]|*\dir{ }& \\ &&&&*{\circ} }}
\newcommand{\Dnlabel}{\xymatrix@R=.25em{&&&&*{\circ}\ar@<-1ex>@{}[d]^{\alpha_{n-1}} \\ *{\circ}\ar@<-1ex>@{}[d]^{\alpha_{1}} \ar@{-}[r]|*\dir{ } & *{\circ}\ar@<-1ex>@{}[d]^{\alpha_{2}}\ar@{-}[r]&{\cdots}&*{\circ}\ar@{-}[l]|*\dir{ }\ar@{-}[ur]|*\dir{ }\ar@{-}[dr]|*\dir{ }\ar@<-4ex>@{}[d]^{\alpha_{n-2}}& \\ &&&&*{\circ}\ar@<-2ex>@{}[u]^(-.5){\alpha_{n}} }}

\newcommand{\Hbold}{{\bf H}}
\newcommand{\Ibold}{{\bf I}}

\newcommand{\Lbold}{{\bf L}}
\newcommand{\Mbold}{{\bf M}}

\newcommand{\bbold}{{\bf b}}

\newcommand{\kbold}{{\bf k}}

\newcommand{\rbold}{{\bf r}}

\newcommand{\ubold}{{\bf u}}

\newcommand{\CBbb}{\mathbb C}

\newcommand{\NBbb}{\mathbb N}

\newcommand{\PBbb}{\mathbb P}

\newcommand{\RBbb}{\mathbb R}

\newcommand{\ZBbb}{\mathbb Z}

\newcommand{\Acal}{\mathcal A}
\newcommand{\Bcal}{\mathcal B}
\newcommand{\Ccal}{\mathcal C}

\newcommand{\Ecal}{\mathcal E}
\newcommand{\Fcal}{\mathcal F}
\newcommand{\Gcal}{\mathcal G}
\newcommand{\Hcal}{\mathcal H}

\newcommand{\Kcal}{\mathcal K}
\newcommand{\Lcal}{\mathcal L}
\newcommand{\Mcal}{\mathcal M}
\newcommand{\Ncal}{\mathcal N}
\newcommand{\Ocal}{\mathcal O}
\newcommand{\Pcal}{\mathcal P}

\newcommand{\Rcal}{\mathcal R}
\newcommand{\Scal}{\mathcal S}

\newcommand{\Ucal}{\mathcal U}

\newcommand{\Wcal}{\mathcal W}

\newcommand{\hfrak}{\mathfrak h}

\newcommand{\lfrak}{\mathfrak l}

\newcommand{\nfrak}{\mathfrak n}

\newcommand{\pfrak}{\mathfrak p}

\newcommand{\ufrak}{\mathfrak u}

\newcommand{\Dscr}{\mathscr D}

\newcommand{\Hscr}{\mathscr H}

\newcommand{\Nscr}{\mathscr N}

\newcommand{\SL}{\mathsf{SL}}

\newcommand{\SU}{\mathsf{SU}}

\newcommand{\PU}{\mathsf{PU}}

\newcommand{\slfrak}{\mathfrak{sl}}
\newcommand{\glfrak}{\mathfrak{gl}}
\newcommand{\sufrak}{\mathfrak{su}}

\DeclareMathOperator{\imag}{im}

\DeclareMathOperator{\rank}{rank}
\DeclareMathOperator{\coker}{coker}

\DeclareMathOperator{\Lie}{Lie}
\DeclareMathOperator{\diag}{diag}

\DeclareMathOperator{\ind}{index }
\DeclareMathOperator{\charpoly}{\mathsf{char}}

\newcommand{\sbt}{\,\begin{picture}(-1,1)(-1,-1)\circle*{2}\end{picture}\ }
\newcommand*{\dt}[1]{\overset{\sbt}{#1}}

\newcommand{\dbar}{\bar\partial}

\newcommand{\Dol}{{\mathsf{Dol}}}

\newcommand{\onebold}{{\bf 1}}

\newcommand{\isorightarrow}{\xrightarrow{
   \,\smash{\raisebox{-0.5ex}{\ensuremath{\sim}}}\,}}

\usepackage{scalerel,stackengine}
\stackMath
\newcommand\widefrown[1]{%
\savestack{\tmpbox}{\stretchto{%
  \scaleto{%
    \scalerel*[\widthof{\ensuremath{#1}}]{\kern-.6pt\frown\kern-.6pt}%
    {\rule[-\textheight/2]{1ex}{\textheight}}
  }{\textheight}%
}{0.5ex}}%
\stackon[1pt]{#1}{\tmpbox}%
}

\begin{document}
\title[Conformal limits for parabolic $\sSL(n,\C)$-Higgs bundles]{Conformal limits for parabolic $\sSL(n,\C)$-Higgs bundles}

 \author{Brian Collier}
\address{Department of Mathematics,
University of California, Riverside,
   Riverside, CA 92521, USA}
    \email{brianc@ucr.edu}
   
\author{Laura Fredrickson}

\address{Department of Mathematics,
   University of Oregon,
   Eugene, OR 97403, USA}
\email{lfredric@uoregon.edu}
   
\author{Richard   Wentworth}

\address{Department of Mathematics,
   University of Maryland,
   College Park, MD 20742, USA}
\email{raw@umd.edu}

\date{\today}

\begin{abstract}
In this paper we generalize the conformal limit correspondence between Higgs
bundles and holomorphic connections to  the parabolic setting.
Under mild genericity assumptions
on the parabolic weights, we prove  that the conformal limit always exists and
that it defines holomorphic sections of the space of parabolic
$\lambda$-connections which preserve a natural stratification and foliate the
moduli space. Along the way, we give a careful gauge theoretic construction of
the moduli space of parabolic Higgs bundles with full flags which allows the
eigenvalues of the residues of the Higgs field to vary.
A number of new
phenomena arise in the parabolic setting. In particular, in the generality we
consider,  unlike the nonparabolic  case, the
nonabelian Hodge correspondence does not define sections of the space of
logarithmic $\lambda$-connections,  and the
conformal limit does not define a
one-parameter family in any given moduli space.
 \end{abstract}

 \maketitle
\tableofcontents

\section{Introduction}
The moduli spaces of Higgs bundles and holomorphic connections on a compact Riemann surface $X$ are  holomorphic symplectic spaces. 
Both moduli spaces have natural stratifications with strata foliated by holomorphic affine Lagrangians \cite{SimpsonDeRhamStrata}.
The nonabelian Hodge correspondence (\textsf{NAH}) and the conformal limit (\textsf{CL}) give two
very different ways of identifying these moduli spaces.
For example, in rank $2$, \textsf{NAH} identifies the Hitchin section with those
holomorphic connections which arise as holonomies of hyperbolic structures on
the underlying topological surface while \textsf{CL}  identifies the Hitchin
section with those holomorphic connections which arise as  holonomies of
$\C\mathbb P^1$-structures with underlying Riemann surface $X$.

In general, \textsf{NAH} is a real analytic map on the entire moduli space but does not
preserve the strata; it is not holomorphic and the two complex structures
combine to define a hyperk\"ahler structure.
On the other hand, \textsf{CL} is defined on each stratum and holomorphically identifies
the affine Lagrangian fibers of each stratum \cite{CollierWentworth:19}; it does
not however extend to a continuous map on the entire moduli space.
Hitchin sections, which are the closed strata, define geometrically
interesting strata in every rank.
Under CL, Hitchin sections are identified with the space of opers
\cite{GaiottoLimitsOPERS}, whereas applying \textsf{NAH} to a Hitchin
section produces connections with real holonomy known as Hitchin
representations \cite{liegroupsteichmuller}.

It was argued by Gaiotto in \cite{GaiottoConj} that the ``conformal limit'' of
the TBA equations from \cite{GMN} should take the form of generalized
Schr\"odinger operators.
This led him to conjecture a correspondence between the Higgs bundles in the Hitchin section  and the locus of holomorphic connections known as \emph{opers}. 
The context of Gaiotto's conjecture was superconformal field
theories of ``class $\mathcal{S}$'', arising from compactifications on a Riemann surface
$X$ cross a circle of radius $R$ (see \cite{GMN}).
In this situation, it is natural and essential to allow for decorated punctures on $X$. 

Without punctures, Gaiotto's  conjecture was proven in
\cite{GaiottoLimitsOPERS} and generalized in \cite{CollierWentworth:19}.
In order to construct the  analogous moduli spaces of Higgs bundles and holomorphic connections on noncompact Riemann surfaces, one must equip the objects with extra weighted flag data at the punctures called a parabolic  structure. 
There is a parabolic version of the nonabelian Hodge correspondence, and the purpose of this paper is to extend the conformal limit correspondence to the parabolic setting. 

A number of new phenomena arise in the conformal limit of parabolic Higgs bundles. 
For example, \textsf{NAH} changes the parabolic structure
and certain relevant eigenvalues according to Simpson's table \eqref{eq simpson
table intro}. We prove that \textsf{CL} changes the parabolic structure
and eigenvalues  according to a different table \eqref{eq CL table intro}. As a
result, the conformal limit takes place on a larger moduli space than
does nonabelian
Hodge, and the targets are different. Interestingly, this implies that the
conformal limit cannot be defined as a limit of a 1-parameter family in a moduli
space of parabolic logarithmic connections; rather, it must be defined as a
limit in an
infinite dimensional configuration space. Unlike the nonparabolic setting,
nonabelian Hodge does not in general define sections of parabolic logarithmic
$\lambda$-connection moduli space; we show the conformal limit does define such
sections.
Rank two stable parabolic Higgs bundles on the four-punctured sphere are
particularly simple to describe, and in this case we
describe the various strata explicitly.

\subsection{Parabolic conformal limits}
Fix a pair $(X,D)$, where $X$ is a compact Riemann surface with holomorphic
cotangent bundle $K$ and a divisor  $D=p_1+\cdots + p_d$.
Let us briefly define notions of different parabolic objects and
refer to \S 2 for more details.
A parabolic bundle $\cE(\alpha)$ on $(X,D)$ is a holomorphic vector bundle $\cE\to X$ and a choice of weighted flag $\cE_p=\cE_{p,1}\supset \cE_{p,2}\supset\cdots \cE_{p,n_p}\supset \{0\}$ for each $p\in D$, where $\cE_{p,j}$ is given a weight $\alpha_{j}(p)$ satisfying $0\leq \alpha_{1}(p)<\cdots <\alpha_{n_p}(p)<1.$ The case where $\dim(\cE_{p,j})-\dim(\cE_{p,j+1})=1$ for all $j$ and all $p\in D$ is called the \emph{full flags case}.

A Higgs field on $\cE(\alpha)$ is a holomorphic bundle map
$\Phi:\cE\to\cE\otimes K(D)$, while a logarithmic connection on $\cE(\alpha)$ is
a holomorphic differential operator $\nabla:\cE\to\cE\otimes K(D)$ satisfying
the Leibniz rule $\nabla(fs)=s\otimes \partial f +f\nabla s.$ When the
residues of $\Phi$ and $\nabla$ additionally preserve the fixed flag in $\cE_p$
for each $p\in D$, $(\cE(\alpha),\Phi)$ and $(\cE(\alpha),\nabla)$ are called
parabolic Higgs bundles and parabolic logarithmic connections, respectively.
The weights on the flags are necessary to define a notion of stability with
respect to which one can form the coarse moduli spaces of
semistable parabolic Higgs bundles $\cP_0(\alpha)$, semistable strongly
parabolic Higgs bundles $\Scal\Pcal_0(\alpha)$,   and parabolic logarithmic
connections $\cP_1(\alpha)$ with fixed parabolic data
\cite{YokogawaParModuli}.\footnote{Throughout the paper we work over the
complex numbers and will
identify moduli schemes with their analytifications as complex analytic spaces.
The universal properties will play no role except in establishing isomorphisms
with the gauge theoretic constructions of Section \ref{sec:analytic}.}

For a stable parabolic Higgs bundle $(\cE(\alpha),\Phi)$ of parabolic degree 0,
Simpson \cite{SimpsonNoncompactcurves} proved there is a unique hermitian metric
$h$ on $\cE|_{X\setminus D}$ which is adapted to the parabolic structure  and
solves the Hitchin equations (see \S\ref{sec nonabelian hodge} below).
A consequence of solving the Hitchin equations is that 
\begin{equation}
        \label{eq flat connection intro}
D(\cE(\alpha),\Phi)=\bar\partial_E+\Phi^{*_h}+\partial_E^{h}+\Phi
\end{equation}  
is a flat connection on the restriction of the underlying smooth bundle of $\cE$
to $X\setminus D$, where $\bar\partial_E+\partial_E^h$ is the Chern connection
of $h$ and $\Phi^{*_h}$ is the hermitian adjoint of $\Phi$ with respect to $h$.

The $(0,1)$ and $(1,0)$ parts of $D(\cE(\alpha),\Phi))$ define a holomorphic
bundle with connection on $ X\setminus D$. As described in
\cite{SimpsonNoncompactcurves}, the metric $h$ determines an extension of the
bundle $\cV$ to all of $X$ such that the connection extends to a logarithmic
connection $\nabla$ on $\cV$.
By construction, $\cV$ has a natural parabolic structure $\cV(\beta)$ on $(X,D)$ and $(\cV(\beta),\nabla)$ is a stable parabolic logarithmic connection. 
At a point  $p\in D$, the parabolic weights $\alpha$ and $\beta$, and the
eigenvalues of the residues of $\Phi$ and $\nabla$ (referred
to henceforth as \emph{complex masses}) are all related by Simpson's table:
\begin{equation}
        \label{eq simpson table intro}
        \begin{tabular}{|c|c|c|}\hline
                &$(\cE(\alpha),\Phi)$& $(\cV(\beta),\nabla)$\\\hline parabolic weights& $\alpha$& $\beta= \alpha-(\mu+\bar\mu)$ \\\hline complex masses & $\mu$& $\nu=\alpha+\mu-\bar\mu$\\\hline
        \end{tabular}
\end{equation} 
Unlike the nonparabolic setting, the above correspondence does not define a map on moduli spaces since the weights of $(\cV(\beta),\nabla)$ vary with $(\cE(\alpha),\Phi)\in\cP_0(\alpha)$. 
By contrast, we show the conformal limit determines a map
$\CL:\cP_0(\alpha)\to\cP_1(\alpha)$.

We now explain the conformal limit of a stable parabolic Higgs bundle
$(\cE(\alpha),\Phi).$
 The metric $h$ above
 actually defines a $\C^*$-family of flat connections on $E|_{X\setminus D}$ 
\begin{equation}
        \label{eq flat conn lambda
intro}D_\lambda(\cE(\alpha),\Phi)=\bar\partial_E+\lambda\Phi^{*_h}+\partial_E^{h
} +\lambda^{-1}\Phi~, \ \ \ \ \ \ \lambda\in\C^*.
\end{equation}
 For $R>0$, consider $D_\lambda(\cE(\alpha),R\cdot \Phi))$  as a two parameter family of flat connections associated to $(\cE(\alpha),\Phi)$. If $h_R$ is the solution to the Hitchin equations for $(\cE(\alpha),R\Phi)$ and $\hbar=\lambda R^{-1}$, then this two parameter family can written as
\begin{equation}
        \label{eq R hbar fam
intro}D_{R,\hbar}(\cE(\alpha),\Phi)=\bar\partial_E+\hbar
R^2\Phi^{*_{h_R}}+\partial_E^{h_R}+\hbar^{-1} \Phi.
\end{equation}
The $\hbar$-conformal limit $\CL_\hbar(\cE(\alpha),\Phi)$ of  $(\cE(\alpha),\Phi)$ is defined to be 
\[\CL_\hbar(\cE(\alpha),\Phi)=\lim_{R\to 0} D_{R,\hbar}(\cE(\alpha),\Phi)\]
in the case such a limit exists.

The goal, then,  concerns the existence of the conformal limit. Recall
that in the moduli space $\cP_0(\alpha)$ of parabolic Higgs bundles,
$\lim_{\lambda\to 0}(\cE(\alpha),\lambda \Phi)$ exists and is a special type of
parabolic Higgs bundle which we call the \emph{Hodge bundle associated to}
$(\cE(\alpha),\Phi)$. Recall also that $(\cE(\alpha),\Phi)$ is called {\em
strongly parabolic} if the residue of $\Phi$ maps $\cE_{p,j}$ into $\cE_{p,j+1}$
for all $p\in D$ and all $j$.
In many points of the paper, we will require a certain technical hypothesis,
which we state here:
\begin{Assumption}[A] \label{Assumption A}
       The objects being considered are either parabolic Higgs bundle with
{\bf full flags or strongly parabolic} Higgs bundles.
\end{Assumption}
The relevance of Assumption A is discussed below.
First, we state the main
result.

\begin{Theorem} \label{thm:conformal-limit} 
    Let $(\cE(\alpha),\Phi)$ be a stable parabolic Higgs bundle satisfying Assumption A whose associated Hodge bundle is stable. Then for any $\hbar\in\C^*$, the $\hbar$-conformal limit $\CL_\hbar(\cE(\alpha),\Phi)$ exists.  
\end{Theorem}
Let us note that semistability implies stability for an open and dense set of
the weights $\alpha$ for the parabolic structure. In such a setting,
$\cP_0(\alpha)$ is smooth, and Theorem \ref{thm:conformal-limit} applies to
every semistable parabolic Higgs bundle satisfying Assumption A.

As in the nonparabolic case of \cite{CollierWentworth:19}, 
existence is proven by identifying every such $(\cE(\alpha),\Phi)$ with a point in a particular slice through its associated Hodge bundle, and the conformal limit is computed in the slice. 
While the strategy here is the same, the analysis in the parabolic setting is
significantly more subtle. 
In particular, a careful gauge theoretic construction of the moduli space
is necessary. This is carried out in \S 3; specifically, Theorems
\ref{thm:analytic-moduli} and \ref{thm:analytic-DRmoduli}. 
We use $L^2_\delta$ theory of Lockhart-McOwen, adapted to gauge theory on
manifolds with cylindrical ends by Taubes, Mrowka, and others. 
In this context,  many previous results exist in the literature
 (cf.\
\cite{Biquard:91,Konno:93,Nakajima:96,DaskalWentworth:97,Biquard:97,
BiquardBoalch:04} and \cite{BGMparabolic,SimpsonNoncompactcurves}).
To extend to arbitrary residues, we have found it necessary to make the full
flags assumption. This is mostly related to the fact, discovered first by
Simpson,  that in the general case the harmonic metric  needs to be adapted to
the Jordan type of the residue, whereas the $L^2_\delta$ theory essentially
requires semisimple residues.
We have chosen to circumvent this
issue through Assumption A. How to extend the analytic results in \S\ref{sec:analytic} beyond Assumption A is not immediately clear to us.

Let $\Pcal_0^s(\alpha)\subset\Pcal_0(\alpha)$ denote the open subset of stable points, and we
assume this is a nonempty.
The main result of \S\ref{sec:analytic} is the following. 

\begin{Theorem} \label{mainthm:analytic-moduli}
Under Assumption A, for a sufficiently small parameter $\delta>0$, there is a complex manifold
    $\Mbold^{par,s}_{\Dol}(\alpha,\delta)$  constructed as an infinite dimensional
    quotient with the following significance: 
    \begin{enumerate}
        \item there is a biholomorphism
            $\Mbold^{par,s}_{\Dol}(\alpha,\delta)\isorightarrow \Pcal_0^s(\alpha)$; 
        \item $\Mbold^{par,s}_{\Dol}(\alpha,\delta)$ admits a Poisson structure;
        \item the symplectic leaves of this Poisson structure are
            hyperk\"ahler manifolds, and they  correspond to
            fixing the complex masses.
    \end{enumerate}
\end{Theorem}
\begin{Remark}
The Poisson structure on $\cP_0(\alpha)$ was constructed in
\cite{LogaresMartens}. We show that this Poisson structure  arises as a
quotient of the holomorphic symplectic structure induced by an
Atiyah-Bott-Goldman form on the moduli space of \emph{framed} parabolic Higgs
bundles (see \S\ref{sec:hyperkahler}). We also note that the
hyperk\"ahler metric in (3) is the natural $L^2$-metric. However, neither the
Atiyah-Bott-Goldman form nor the $L^2$-metric descend to the full moduli space
$\Mbold^{par,s}_{\Dol}(\alpha,\delta)$.
\end{Remark}

\subsection{Parabolic logarithmic $\lambda$-connections}
A $\lambda$-connection on a holomorphic vector bundle $\cV$ is a holomorphic
differential operator $\nabla^\lambda:\cV\to\cV\otimes K$ which satisfies a
$\lambda$-scaled Leibniz rule $\nabla^\lambda(fs)=\lambda\, s\otimes  \partial f
+f\nabla^\lambda s.$
In particular, when $\lambda=0$, we get a Higgs field, and when $\lambda=1$ we
get a holomorphic connection.
In the nonparabolic setting, it is useful to think of the flat connection
$D_\lambda(\cE,\Phi)$ from \eqref{eq flat conn lambda intro} as a map from $\C$
into the space of $\lambda$-connections defined by $\lambda\mapsto
(\bar\partial_E+\lambda\Phi^{*_h},\lambda\partial_E^h+\Phi)$. Indeed, there is
a moduli space of semistable $\lambda$-connections which naturally fibers over
$\C$, and these maps define sections which foliate the moduli space. Moreover,
these sections are related to the twistor lines of the hyperk\"ahler structure
on the moduli space of Higgs bundles.

There is a straightforward parabolic generalization of $\lambda$-connections which we refer to as parabolic logarithmic $\lambda$-connections. 
Again, there is a moduli space $\cP(\alpha)\to \C$  (resp.\
$\Scal\Pcal(\alpha)\to\CBbb$)
 of semistable parabolic (resp.\ strongly parabolic) logarithmic
$\lambda$-connections with fixed parabolic structure, and the fibers over $0$
and $1$ are the moduli spaces $\cP_0(\alpha)$ and $\cP_1(\alpha)$ of parabolic
Higgs bundles and parabolic logarithmic connections, respectively.
If $(\cV^\lambda(\beta),\nabla^\lambda)$ is the associated parabolic logarithmic $\lambda$-connection associated to the flat connection $D_\lambda(\cE(\alpha),\Phi))$ from \eqref{eq flat conn lambda intro}, the $\lambda$-connection analogue of Simpson's table \eqref{eq simpson table intro} has weights $\beta=\alpha-\lambda\mu-\overline{\lambda\mu}$ and complex masses $\nu=\lambda\alpha+\mu-\lambda\bar\mu$.
In particular, the  parabolic logarithmic $\lambda$-connections $(\cV^\lambda(\beta),\nabla^\lambda)$ does not define a section of $\cP(\alpha)\to\C$ unless the complex masses of the Higgs field all vanish, a condition that does not hold for all points of $\cP_0(\alpha)$. 

The $\hbar$-conformal limit has a natural interpretation as a limit of $\hbar$-connections. Namely, the associated $\hbar$-connection to $D_{R,\hbar}$ from \eqref{eq R hbar fam intro} is $(\bar\partial_E+R^2\hbar\Phi^{*_{h_R}},\hbar\partial_{h_R}+\Phi)$. We note that the both the complex masses and parabolic weights of the family $D_{R,\hbar}$ depend on $R$ and $\hbar$, see Table \ref{eq cf fam simpson table}. However, in the limit $R\to 0,$ i.e., the $\hbar$-conformal limit, we prove the following.

\begin{Theorem}
Let $(\cE(\alpha),\Phi)$ be a stable parabolic Higgs bundles satisfying Assumption A whose associated Hodge bundle is stable. Then the $\hbar$-conformal limit $\CL_\hbar(\cE(\alpha),\Phi)$ naturally extends to a stable parabolic logarithmic $\hbar$-connection on $X$ whose parabolic weights and complex masses are determined by the following table
\begin{equation}
        \label{eq CL table intro}
        \begin{tabular}{|c|c|c|}\hline
                &$(\cE(\alpha),\Phi)$& $\CL_\hbar(\cE(\alpha),\Phi)$\\\hline parabolic weights& $\alpha$& $\beta= \alpha$ \\\hline complex masses & $\mu$& $\nu=\hbar\alpha+\mu$\\\hline
  \end{tabular}
  \end{equation}
In particular, the map $\hbar\mapsto \CL_\hbar(\cE(\alpha),\Phi)$ defines a section of the moduli space of parabolic logarithmic $\lambda$-connection $\cP(\alpha)\to\C$ through $(\cE(\alpha),\Phi)\in\cP_0(\alpha).$ 
\end{Theorem}
\begin{Remark}
The natural extension of $\CL_\hbar(\cE(\alpha),\Phi)$ from a holomorphic $\hbar$-connection on $X\setminus D$ to a parabolic logarithmic $\hbar$-connection on $X$ is determined by the hermitian metric $h_0$ at the Hodge bundle associated to $(\cE(\alpha),\Phi)$. 
\end{Remark}

Scaling the parabolic logarithmic $\lambda$-connection defines a $\C^*$-action
on $\cP(\alpha)\to\C$ which covers the standard action on $\C$. For
$\lambda\neq0$, this action gives a biholomorphism
$\xi\cdot\cP_\lambda(\alpha)\cong\cP_{\xi\lambda}(\alpha)$ for all
$\xi\in\C^*.$
Generalizing Simpson's work \cite{SimpsonDeRhamStrata}, we show the limits $\xi\to0$ always exist and hence are Hodge bundles in $\cP_0(\alpha)$. As in \cite{SimpsonDeRhamStrata}, this is done by showing every semistable parabolic logarithmic $\lambda$-connection admits a Griffiths transverse filtration whose associated graded parabolic Higgs bundle is polystable, see Appendix \ref{appendix Simpson}.
As a result, $\cP(\alpha)$ acquires a Bia{\l}ynicki-Birula stratification
\begin{equation}\label{eq bb strat intro}
        \cP(\alpha)=\coprod_{a\in \pi_0(\cP(\alpha)^{\C^*})} \cW^a,
\end{equation}
where $\cW^a$ is all points whose $\C^*$-limit is in the component
labeled by $a\in \pi_0(\cP_0(\alpha)^{\C^*})$,  and it is a vector bundle over
the smooth locus.
Denote the subset of $\cW^a$ with fixed $\lambda$ by
$\cW^a_\lambda.$ As in the nonparabolic setting, we  prove that the
conformal limit foliates $\cP(\alpha)$ or $\cS\cP(\alpha)$  with strata
preserving sections and biholomorphically identifies the fibers of
$\cW_0^a$ with the fibers of $\cW_\hbar^a.$

\begin{Theorem}\label{thm:CLintro}
Consider the spaces $\cP(\alpha)$ satisfying Assumption A. For each $\hbar\in
\C$ and each connected component $\cF^a\subset\cP(\alpha)^{\C^*}$,
consider the natural projection map $\cW^a_\hbar\to \cF^a.$
For each stable $x\in \cF^a$ denote the fiber over $x$ by
$\cW_\hbar(x)$. Then,  the $\hbar$-conformal limit
\[\mathrm{CL}_\hbar:\cW_0(x)\to \cP_\hbar(\alpha)\] is a biholomorphism
onto $\cW_\hbar(x).$ In particular, each $y\in \cW_0(x)$
defines a section of $\cP(\alpha)\to \C$  defined by $\hbar\to
\mathrm{CL}_\hbar(y)$, and these sections foliate $\cW(x).$
\end{Theorem}
     
     Finally, the fibers  $\Wcal_{\hbar}(x)$ have a  ``brane''
interpretation  which generalizes the nonparabolic case.

\begin{Theorem}  \label{thm:coisotropic}
Let $x\in \cP^s(\alpha)^{\C^*}$. With respect to the Poisson
structure in Theorem \ref{mainthm:analytic-moduli}, the fiber
$\Wcal_0(x)\subset \Pcal_0^s(\alpha)$ is a holomorphically embedded coisotropic
submanifold. Moreover,  the intersections of   $\Wcal_0(x)$ with
    the symplectic leaves of $\Pcal_0^s(\alpha)$ are holomorphic 
    Lagrangian submanifolds.
\end{Theorem}

\subsection{Organization of paper}
This paper is organized as follows. In \S \ref{sec:moduli} we introduce
parabolic Higgs bundles and their generalization to parabolic
$\lambda$-connections. We also briefly discuss the nonabelian Hodge
correspondence and Simpson's table. In \S \ref{sec:analytic} we give a
self-contained exposition of the gauge theoretic construction of  moduli spaces
of parabolic Higgs bundles and the de Rham moduli space of flat connections
using weighted Sobolev spaces. The details are required to give a precise
description of the Bia{\l}ynicki-Birula stratification and the proof of the
existence of a conformal limit in \S \ref{sec:strat}. Finally, in \S
\ref{sec:example} we illustrate the results of the paper in the particular case
of the four-punctured sphere.

 \subsubsection*{Acknowledgements} The authors thank
 David Alfaya, Dylan Allegretti, Sebastian Heller, Rafe Mazzeo, Motohico Mulase,
Carlos Simpson and Hartmut Wei\ss\  for helpful discussions. B.C.'s research is
supported by  NSF grants DMS-2103685 and DMS-2337451. L.F.'s research is
supported by  NSF grant DMS-2005258.
R.W.’s research is supported by  NSF grant DMS-2204346.
He is also grateful to the Max-Planck Institute for Mathematics, where a
portion of this work was completed during Spring 2023. 
This project was initiated while all of us were visiting the Mathematical Sciences Research Institute (MSRI), now becoming the Simons Laufer Mathematical Sciences Institute (SLMath), for the semester-long program \emph{Holomorphic Differentials in Mathematics and Physics} in 2019; the institute is supported by the National Science Foundation (Grant No. DMS-1440140).


\section{Parabolic objects}\label{sec:moduli}
For background we mostly follow \cite{SimpsonNoncompactcurves,LogaresMartens}. 
Fix a closed Riemann surface $X$ of genus $g$ and with structure sheaf
$\Ocal_X$ and  canonical bundle $K_X$. 
Let $\{p_1,\cdots,p_d\}$ be a set of $d$ distinct points in $X$ 
such that $2g-2+d>0$, and let $D=p_1+\cdots +p_d$ be the associated
effective divisor. Let $K_X(D)$ be the line bundle whose sections are
meromorphic $1$-forms on $X$ with at most simple poles at the points of
$D$.  Denote the residue map by
\[\res: K_X(D)\to \bigoplus_{p\in D}\cO_{p},\]
and denote the projection onto $\cO_p$ by $\res_p.$

\subsection{Parabolic bundles} \label{sec:parabolic-bundles}
\begin{Definition}[{\sc parabolic bundle}] \label{def mero filt}
Let $\cE\to X$ be a rank $n$ holomorphic vector bundle. For each point $p\in D$, a parabolic structure at $p$ is a filtration $\{\cE_\alpha\}_{\alpha\in\R}$ of the germs of 
meromorphic sections of $\cE$ at $p$ such that $\cE_0$ consists of germs of holomorphic sections at $p$ and
\begin{enumerate}
    \item $\alpha\leq\beta$ implies $\cE_\alpha\supset \cE_\beta$, 
    \item for each $\alpha\in\R$ there is $\epsilon>0$ so that $\cE_{\alpha-\epsilon}=\cE_\alpha,$ and
    \item $\cE_{\alpha+1}=\cE_\alpha(-p)$ for all $\alpha.$
\end{enumerate}
\end{Definition}
There are natural parabolic structures induced on subbundles, quotients, direct sums, tensor product and exterior products. By property (3), the evaluation map $\cE_0\to \cE_p$ defined by $s\mapsto s(p)$ has kernel $\cE_1$. This defines an isomorphism $\cE_0/\cE_1\cong\cE_p$ and identifies $\cE_{\alpha_i}$ with a linear subspace $\cE_{p,i}\subset\cE_p.$ In particular, a parabolic structure at $p$ is equivalent to a weighted filtration
\begin{equation}
    \label{eq parabolic filt}\cE|_{p}=\cE_{p,1}\supset \cE_{p,2}\supset\cdots\supset \cE_{p,n_p}\supset\cE_{p,n_p+1}=\{0\}, \ \ 0\leq \alpha_{1}(p)<\cdots<\alpha_{n(p)}(p)<1.
\end{equation}
Define the integers 
\[m_j(p)=\dim(\cE_{p,j})-\dim(\cE_{p,j+1}).\]
When $m_j(p)=1$ for all $j$ and $p\in D$, we say the parabolic structure is given by {\em full flags}. 
We will use the notation $\cE(\alpha)$ for a vector bundle $\cE$ equipped with a parabolic structure.

\begin{Definition}[{\sc parabolic map}] 
         \label{def parabolic morph}Given two parabolic vector bundles $\cE(\alpha)$ and $\cF(\beta)$, a holomorphic bundle map $f:\cE\to \cF$ is called {\em parabolic} if, $\alpha_j(p)> \beta_{k}(p)$ implies $f(\cE_{p,j})\subset \cF_{p,k+1}$ for all $p\in D$, and {\em strongly parabolic} if $\alpha_j(p)\geq \beta_k(p)$ implies $f(\cE_{p,j})\subset \cF_{p,k+1}$ for all $p\in D$.
 \end{Definition} 

If $\cF\subset \cE$ is a holomorphic subbundle, then the $j^{th}$ part of the filtration $\cF_{p,j}$ of the induced parabolic structure on $\cF$ is given by 
\[\cF_{p,j}=\cF_p\cap\cE_{p,j}~, \ \ \alpha_j^\cF(p)=\max\limits_k\{\alpha_k(p) ~|~ \cF_p\cap \cE_{p,k}=\cF_{p,j}\}.\]
The induced parabolic structure on the determinant bundle $\Lambda^n\cE$ is just a weight 
\begin{equation}
    \|\alpha(p)\|=\sum_{j=1}^{n_p}m_j\alpha_{j}(p).
\end{equation} 
To normalize the weights, use Property (3) of Definition \ref{def mero filt}. Namely, 
\begin{equation}\label{eq:tensorofpar}\det(\cE(\alpha))=\Big(\Lambda^n\cE\otimes \bigotimes_{p\in D}\cO\left(\left\lfloor \|\alpha(p)\|\right\rfloor\right)\Big)(\beta),\end{equation}
where $\beta(p)=\|\alpha(p)\|-\lfloor\|\alpha(p)\|\rfloor.$ 
A rank $n$ parabolic bundle $\cE(\alpha)$ together with an isomorphism
$\det(\cE(\alpha))\cong\cO(0)$ between $\det(\cE(\alpha))$ and
the trivial bundle with the trivial parabolic structure is an $\sSL(n,\C)$-parabolic bundle. In this case,  $||\alpha(p)||\in\N$ for all $p\in D.$

The {\em parabolic degree} of $\cE(\alpha)$ will be denoted by $\deg(\cE(\alpha))$ and is defined by 
\begin{equation}
    \label{eq: parabolic degree}\deg(\cE(\alpha))=\deg(\cE)+\sum_{p\in D}\|\alpha(p)\|~.
\end{equation}
Define the parabolic slope by $\mu(\cE(\alpha))=\deg(\cE(\alpha))/\rk(\cE)$. A parabolic bundle $\cE(\alpha)$ is {\em semistable} if 
\begin{equation}
        \label{eq parabolic ss}\mu(\cF(\alpha))\leq\mu(\cE(\alpha))~
\end{equation}
for every proper holomorphic subbundles $\cF\subset \cE$. A parabolic
bundle $\cE(\alpha)$ is stable if the above inequity is always strict.  The
parabolic degree of an $\sSL(n,\C)$-parabolic bundle is zero.  

There is a moduli space $\cN(\alpha)=\cN(\alpha,\sSL(n,\C))$  whose closed points
correspond to $\cS$-equivalence classes of semistable parabolic bundles
$\sSL(n,\C)$-bundles \cite{MehtaSeshadri}. 
Let $\Ncal^s(\alpha)\subset\Ncal(\alpha)$ denote the open subset of stable
parabolic bundles.
For generic values of the weights, semistability implies stability, and
$\Ncal^s(\alpha)=\Ncal(\alpha)$. In general,
$\cN(\alpha)$, if nonempty,  is a smooth projective variety of dimension
    \[\dim(\cN(\alpha))=(g-1)(n^2-1)+\frac{1}{2}\sum_{p\in D}\left(n^2-\sum_{j=1}^{n_p}m_j^2\right).\]
  Note that the dimension depends on the parabolic structure but not the weights. 
There is a variant where one replaces
the trivial bundle $\cO(0)$ with any parabolic line bundle $\cL(\beta)$ and considers the moduli space $\cN(\alpha,\cL(\beta))$ of semistable parabolic bundles with fixed parabolic structure together with an isomorphism $\det(\cE(\alpha))\cong \cL(\beta)$ . 
\subsection{Parabolic Higgs bundles}
We now define the notion of a $\sSL(n,\C)$-parabolic Higgs bundle. 
 
\begin{Definition}[{\sc parabolic Higgs bundle}]  \label{def:parabolic-higgs}
    An $\sSL(n,\C)$-parabolic Higgs bundle is a pair $(\cE(\alpha),\Phi)$ on $(X,D)$, where 
    \begin{itemize}
        \item $\cE(\alpha)$ is a parabolic $\sSL(n,\C)$-bundle on $X$, and
        \item $\Phi\in H^0(\End(\cE)\otimes K(D))$ such that $\Tr(\Phi)=0$ and the residue $\res(\Phi)$ preserves the flag in $\cE_p$ for all $p\in D$, i.e., $\res_p(\Phi)(\cE_{p,j})\subset \cE_{p,j}$ for all $p\in D$ and all $j$.
    \end{itemize} 
    A parabolic Higgs bundle $(\cE(\alpha),\Phi)$ is called \emph{strongly parabolic} if $\res_p(\Phi)$ is zero on each graded piece $\cE_{p,j}/\cE_{p,j+1}$, i.e., $\res_p(\Phi)(\cE_{p,j})\subset\Ee_{p,j+1}$ for all $p\in D$ and all $j.$
\end{Definition}

There are natural notions of stability for parabolic Higgs bundles. Namely, $(\cE(\alpha),\Phi)$ is semistable if \eqref{eq parabolic ss} holds for all proper subbundles $\cF\subset\cE$ such that $\Phi(\cF)\subset\cF\otimes K(D)$. 
There is a moduli space $\cP_0(\alpha)=\cP_0(\alpha,\sSL(n,\C))$ whose
closed points correspond to $\cS$-equivalence classes of semistable $\sSL(n,\C)$-parabolic Higgs bundles with fixed parabolic structure \cite{YokogawaParModuli}, the subscript will be explained in the next section.
We again let $\Pcal_0^s(\alpha)$ denote the locus of stable points, and it
is again true that
for generic choices of the weights, semistability implies stability. When
nonempty, the
moduli space $\cP_0^s(\alpha)$ is a smooth quasi-projective variety of dimension
\begin{equation} \label{eqn:dim-par}
\dim(\cP_0^s(\alpha))=(n^2-1)(2g-2+d)\ .
\end{equation}
In particular, the dimension is independent of the fixed parabolic
structure. On the other hand, an open subset of  
semistable strongly parabolic Higgs bundles
$\Ss\cP_0(\alpha)\subset\cP_0(\alpha)$ is identified with the cotangent
bundle of $\Nn^s(\alpha)$. In particular,
$\dim(\cS\cP_0^s(\alpha))=2\dim(\Nn^s(\alpha))$.

We emphasize that the only condition on the residues of the Higgs fields in
$\cP_0(\alpha)$ is that they preserve the flag structure of the parabolic
bundle. For a given parabolic Higgs bundle, the residue map can be
projected onto the diagonal terms of the associated graded
$$\fl_p=\biggl[\, \bigoplus_{j=0}^{n_p-1}\End(\cE_{p,j}/\cE_{p,j+1})\, \biggr]_0~,$$
where $[\,]_0$ indicates the traceless part.
Set $\sL_p=[\, \prod_{j=0}^{n_p-1}\sGL(\cE_{p,j}/\cE_{p,j+1})\, ]_1,$ where
$[\,]_1$ indicates overall determinant $=1$. On the moduli space $\cP_0(\alpha)$ this induces a map
 \begin{equation}
         \label{eq big Res on moduli}\mathrm{Res}:\cP_0(\alpha)\to \bigoplus_{p\in D}\fl_p/\sL_p~.
 \end{equation}
By definition, $\mathrm{Res}^{-1}(0)=\cS\cP_0(\alpha)$ is the strongly parabolic moduli space. There is also map to the GIT quotient $\bigoplus_{p\in D}\fl_p\sslash \sL_p$ which records the ordered eigenvalues of the residue.
Under some assumptions, which are always satisfied in the full flags case, the fibers of $\mathrm{Res}$ were shown to be the symplectic leaves of a natural Poisson structure on $\cP_0(\alpha)$ \cite[\S3.2.4]{LogaresMartens}. 
\begin{Remark}[{\sc residue for full flags}] \label{rem:Resfullflag}
        In the full flags case, the action of $\sL_p$ on $\fl_p$ is trivial, so $\mathrm{Res}$ records the eigenvalues of the residue of the Higgs field at each $p\in D$
        \[\mathrm{Res}:\cP_0(\alpha)\to\bigoplus_{p\in
        D}\fl_p\simeq\bigoplus_{p\in D}\C^{n-1}~.\] 
\end{Remark}
As in the nonparabolic case, choosing a basis of invariant polynomials and evaluating the Higgs field defines the Hitchin map 
\begin{equation}
        \label{eq Hitchin map} \cP_0(\alpha)\to \cB\cong \bigoplus_{j=2}^n H^0(K^j(j D))~.
\end{equation}
Note that the strongly parabolic moduli space $\cS\cP_0(\alpha)$ fibers over $\bigoplus_{j=2}^n H^0(K^j((j-1) D))$. 

There is a natural $\C^*$-action on $\cP_0(\alpha)$, given by $\xi\cdot (\cE(\alpha),\Phi)=(\cE(\alpha),\xi\Phi).$ This action preserves the moduli space $\cS\cP_0(\alpha)$ of strongly parabolic Higgs bundles but does not preserve the other fibers of $\mathrm{Res}$.
In \cite{YokogawaParModuli}, the analogue of the Hitchin map for $\cP_0(\alpha)$ is shown to be proper. 
This implies the $\C^*$-limits $\lim_{\xi\to0}[\cE(\alpha),\xi\Phi]$ always exist in $\cP_0(\alpha)$.
The $\C^*$-fixed points are systems of Hodge bundles \cite{SimpsonNoncompactcurves}. That is, $(\cE(\alpha),\Phi)$ is a $\C^*$-fixed point if and only if the parabolic bundle $\cE(\alpha)$ decomposes as a direct sum of parabolic bundles
\[\cE(\alpha)=\cE_1(\beta_1)\oplus\cdots\oplus\cE_\ell(\beta_\ell)~,\]
 and there are $\phi_j\in H^0(\Hom(\cE_j(\beta_j),\cE_{j+1}(\beta_{j+1}))\otimes K(D))$ such that
\[\Phi=\begin{pmatrix}
    0&\\\phi_1&0\\&\ddots&\ddots\\&&\phi_{\ell-1}&0
\end{pmatrix}.\]

\begin{Remark}
\label{remark full flag implies strongly parabolic fixed points}
Even though the Higgs field at a $\C^*$-fixed point is nilpotent, the $\C^*$-fixed points in $\cP_0(\alpha)$ are not always strongly parabolic, i.e., not necessarily in $\mathrm{Res}^{-1}(0).$ However, crucial for our later analysis, every $\C^*$-fixed point is strongly parabolic in the full flags case. 
\end{Remark}

\begin{Example}[{\sc $\C^*$-fixed points in rank 2}] \label{C* fixed points in rank 2}
For $\rk(\cE)=2,$ the $\C^*$-fixed points are easy to describe. Either
    $\Phi=0$ and $\cE(\alpha)$ is a semistable parabolic $\sSL(2,\C)$-bundle, or $\cE(\alpha)$ is not stable and $\cE(\alpha)$ is a direct sum of two parabolic line bundles $\cE(\alpha)\cong \Ll_1(\beta_1)\oplus \Ll_2(\beta\_2)$. 
The filtration of $\cE_p$ as $\cE_p=\cE_{p,1} \supset \cE_{p,2} \supset \{0\}$ is compatible with these line subbundles.  In particular, for each $p\in D$, $\beta_1\neq\beta_2$ if and only if $\cE_{p}=\cE_{p,1}\supset\cE_{p,2}\supset\{0\}$, and  
$\beta_i(p)>\beta_j(p)$ if and only if $\Ll_i(p)=\cE_{p,2}$, $\beta_{i}(p)=\alpha_2(p)$ and $\beta_j(p)=\alpha_1(p).$ Moreover, with respect to this decomposition
\[\Phi=\begin{pmatrix}
        0&0\\\phi_0&0
\end{pmatrix}:\Ll_1\oplus \Ll_2\to(\Ll_1\otimes K(D))\oplus (\Ll_2\otimes
    K(D))\ ,\]
    where $\phi_0\in H^0(\Hom(\Ll_1,\Ll_2\otimes K(D)))$ is nonzero and satisfies $\res_p(\phi_0)=0$ for each $p\in D$ such that $\beta_1(p)>\beta_2(p).$ Such a Higgs bundle is stable exactly when $\deg(\cL_2(\beta_2))<0$.
\end{Example}

\begin{Remark}
 \label{rem:invariant-notation}
It will be convenient to use invariant notation to denote the structure
    group, parabolics, Levi factors, and unipotent subgroups. 
To this end,  we henceforth employ the following notation.
    $\sG:= \SL(n,\CBbb)$.
    The parabolic subgroup defined by the flag structure at $\Ecal_p$ will
    be denoted by $\sP_p$, with unipotent subgroup $\sU_p$, and Lie algebras
    $\pfrak_p$ and $\ufrak_p$. For convenience, we summarize the relevant dimension
    formulas: 
    \begin{align}
        \begin{split} \label{eqn:dimensions}
            \dim\sG &= \dim\sL_p+2\dim\sU_p= n^2-1 \\
            \dim\sL_p &= -1+\sum_{j=0}^{n_p-1} m_j(p)^2 \\
            \dim\sU_p&= \sum_{0\leq i<j\leq n_p-1} m_i(p)m_j(p) \\
            \dim\sP_p&=-1 +\sum_{0\leq i\leq j\leq n_p-1}
            m_i(p)m_j(p)
        \end{split}
    \end{align}

\end{Remark}

\subsection{Parabolic logarithmic $\lambda$-connections}

We now equip parabolic bundles with holomorphic differential operators known as $\lambda$-connections for $\lambda\in\C$.  

\begin{Definition}[{\sc parabolic logarithmic $\lambda$-connection}] 
    An $\sSL(n,\C)$ parabolic logarithmic $\lambda$-connection is a triple $(\lambda,\cE(\alpha),\nabla)$ on $(X,D)$, where  $\lambda\in\C$, $\cE(\alpha)$ is a parabolic bundle on $(X,D)$, and $\nabla:\cE\to\cE\otimes K(D)$ is a $\C$-linear sheaf map such that 
        \begin{enumerate}
                \item $\nabla(fs)=\lambda\partial f\otimes s+f\nabla s$ for any locally defined holomorphic function $f$ and section $s,$ 
                \item for all $p\in D,$ the residue preserves the flag in $\cE_p$, i.e., $\res_p(\nabla)(\cE_{p,j})\subset\cE_{p,j}\otimes K(D)$ for all $j$,
            \item via the isomorphism $\det(\cE(\alpha))\cong\cO(0),$ the induced operator on $\det(\cE(\alpha))$ is  $\lambda\partial$. 
         \end{enumerate}  
    An $\sSL(n,\C)$ parabolic $\lambda$-connection $(\lambda,\cE(\alpha),\nabla)$ is called \emph{strongly parabolic} if $\res(\nabla)$ acts as multiplication by $\lambda\alpha_j(p)$ on $\cE_{p,j+1}/\cE_{p,j}$. That is, for all $p\in D$ and all $j$
        \[(\res_p(\nabla)-\lambda\alpha_{j}(p)\Id)(\cE_{p,j})\subset\cE_{p,j+1}.\]
\end{Definition}

\begin{Remark}
When $\lambda=1$, we simply refer to these objects as parabolic logarithmic connections since dropping the second and third conditions on $\nabla$ recovers the notion of a logarithmic connection. Note that for $\lambda\neq 0,$ parabolic logarithmic $\lambda$-connections can be identified with parabolic logarithmic connections via the map $(\lambda,\cE(\alpha),\nabla)\mapsto (1,\cE(\alpha),\lambda^{-1}\cdot \nabla)$. One the other hand, a parabolic logarithmic $0$-connection is a parabolic Higgs bundle.
\end{Remark}

The stability conditions for parabolic Higgs bundles generalize immediately
to parabolic $\lambda$-connections, and there is a moduli space
$\cP(\alpha)$ of $\sSL(n,\C)$ parabolic logarithmic $\lambda$-connections
with fixed parabolic structure \cite{AlfayalambdaModuli} whose points correspond
to $\cS$-equivalence classes of semistable parabolic logarithmic
$\lambda$-connections.\footnote{In \cite[Theorem 8.4]{AlfayalambdaModuli} a
moduli space $\cM_{\Hod}(\xi,\alpha,\bar r)$ is constructed which is the
analogue of the strongly parabolic condition on Higgs bundles. This is a closed
subvariety of $\cP(\alpha)$, the moduli space $\cP(\alpha)$ is a
special instance of the $\Lambda$-module moduli space constructed \cite[Theorem
5.8]{AlfayalambdaModuli}.} Note that there is a projection map
\[\Lambda:\cP(\alpha)\to\C\ \ ;\ \ \Lambda([\lambda,\cE(\alpha),\nabla])=\lambda.\]
We will denote the fibers $\Lambda^{-1}(\lambda)$ by $\cP_\lambda(\alpha)$.
In particular, $\Lambda^{-1}(0)=\cP_0(\alpha)$ is the moduli space of
$\sSL(n,\C)$ parabolic Higgs bundles while $\cP_1(\alpha)$ is the moduli
space of $\sSL(n,\C)$ parabolic logarithmic connections. 
There is also a moduli space $\cS\cP(\alpha)$ of strongly parabolic logarithmic connections. 

The moduli space $\cP(\alpha)$ has a natural $\C^*$-action which is given by 
\begin{equation}
        \label{eq C*action}\xi\cdot[\lambda,\cE(\alpha),\nabla]=[\xi\lambda,\cE(\alpha),\xi\nabla].
\end{equation}
This action preserves the Higgs bundle moduli space $\cP_0(\alpha)$ and defines an isomorphism  $\cP_\lambda(\alpha)\cong\cP_{\lambda\cdot \xi}(\alpha)$ for $\lambda\neq0.$ In particular, the $\C^*$-fixed points are exactly systems the Hodge bundles in $\cP_0(\alpha)$ described above. The action also preserves the strongly parabolic locus $\cS\cP(\alpha).$

Generalizing the map $\mathrm{Res}$ from \eqref{eq big Res on moduli}, there is a residue map $\mathrm{Res}:\cP(\alpha)\to \bigoplus_{p\in D}\fl_p/\sL_p,$
where $\mathrm{Res}([\lambda,\cE(\alpha),\nabla])$ is the Levi projection of the residue of $\nabla.$ 
There is also a map $\cP(\alpha)\to\bigoplus_{p\in D}\fl_p\sslash\sL_p$ which records the eigenvalues of the residue of $\nabla.$ 
\begin{Definition}[{\sc complex masses}] \label{def complex masses}
        The eigenvalues of the residue of $\nabla$ will be referred to as the \emph{complex masses} of $(\lambda,\cE(\alpha),\nabla).$
\end{Definition}

In the nonparabolic setting, Simpson showed that the $\C^*$-limits $\lim_{\xi\to 0}[\xi\lambda,\cE(\alpha),\xi\cdot\nabla]$ always exist in $\cP(\alpha)$, and thus correspond to the $\C^*$-fixed points \cite{SimpsonDeRhamStrata}. In Appendix \ref{appendix Simpson}, we show that Simpson's methods can be extended to the parabolic Higgs and logarithmic $\lambda$-connection setting. In particular, we prove the following.
\begin{Proposition}[{\sc $\C^*$-limits exists}] 
        For any $[\lambda,\cE(\alpha),\nabla]\in \cP(\alpha)$, the limit $\lim_{\xi\to 0}[\xi\lambda,\cE(\alpha),\xi\nabla]$
        exists.
\end{Proposition}
As an immediate consequence, the moduli space $\cP(\alpha)$ admits a Bia{\l}ynicki-Birula stratification
\begin{equation}\label{eq bb strat}
        \cP(\alpha)=\coprod_{a\in \pi_0(\cP_0(\alpha)^{\C^*})} \cW^a,
\end{equation}
where $\cW^a$ consists of all points whose $\C^*$-limit is in the
connected component of the $\C^*$-fixed point corresponding to $a\in
\pi_0(\cP_0(\alpha)^{\C^*}).$

\begin{Example}[{\sc stratification for rank 2}] \label{rk 2 example}
        In rank 2, the limit  as $\xi\to0$ and the stratification \eqref{eq bb strat} are easy to describe since they are determined by the Harder--Narasimhan stratification of the underlying parabolic bundle. 
Namely, consider a stable $\lambda$-connection $(\lambda,\cE(\alpha),\nabla)$. 
\begin{itemize}
        \item If $\cE(\alpha)$ is a stable parabolic bundle, then $\lim_{\xi\to0}[\xi\lambda,\cE(\alpha),\nabla]=[0,\cE(\alpha),0]$. 
        \item If $\cE(\alpha)$ is a stable parabolic bundle,  let $\cL(\beta)$ be the maximal destabilize subbundle. Then 
        \[\lim_{\xi\to0}[\xi\lambda,\cE(\alpha),\nabla]=[0,\cL(\beta)\oplus\cE(\alpha)/\Ll(\beta),\begin{pmatrix}
                0&0\\\phi_0&0
        \end{pmatrix}],\]
        where $\phi_0:\cL(\beta)\to\cE(\alpha)/\cL(\beta)\otimes K(D)$ is a nonzero holomorphic section.
\end{itemize}
\end{Example}
\subsection{Nonabelian Hodge, Simpson's table and the Conformal Limit}\label{sec nonabelian hodge}
The nonabelian Hodge correspondence defines a one-to-one correspondence between polystable parabolic Higgs bundles and polystable parabolic logarithmic connections. In both directions, the correspondence is through the existence of a hermitian metric $h$ on the underlying smooth complex vector bundle which is singular at $p\in D$. In this section we recall the main features of the correspondence, all of which were established by Simpson in \cite{SimpsonNoncompactcurves}, and define the conformal limit. More details on the analytical set-up are provided in \S \ref{sec:analytic}.

Let $F_{(\Ecal,h)}$ denote the curvature of the Chern connection $A=(\Ecal,h)$
of $\Ecal$ with respect to $h$.
There is a class of hermitian metrics $h$ on $\cE|_{X\setminus D}$ called
\emph{acceptable} which is a condition involving an upper bound on the norm of
$F_{(\Ecal,h)}$, see \cite[p. 736]{SimpsonNoncompactcurves}.
The key property of acceptable metrics is that the growth rates of local
meromorphic sections induces filtration of  $\cE|_{X\setminus D}$ by coherent
subsheaves, see Definition \ref{def mero filt}, and hence
 a parabolic structure on $\cE|_p$ for each $p\in D.$  
 Let $h$ be an acceptable metric and $z$ be local holomorphic coordinate centered at $p\in D$. A local meromorphic section $s$ of $\cE_{X\setminus D}$ is in $\cE_{\alpha,p}$ if, for all $\epsilon>0$
\[|s(z)|^2_h\in O(|z|^{2\alpha-\epsilon})~.\]

An acceptable metric $h$ is said to be compatible with a parabolic vector bundle $\cE(\alpha)$ if the parabolic structure it induces agrees with $\cE(\alpha).$ 
Suppose the parabolic structure at $p$ is given by 
\[\cE|_{p}=\cE_{p,1}\supset \cE_{p,2}\supset\cdots\supset \cE_{p,n_p}\supset\{0\}, \ \ 0\leq \alpha_{1}(p)<\cdots<\alpha_{n(p)}(p)<1,\]
and $m_j=\dim(\cE_{p,j})-\dim(\cE_{p,j+1})$. Then, in a basis of the associated graded, one particularly nice adapted metric is
\begin{equation}
         \label{eq local model no logs}h = \diag\Big(\underbrace{|z|^{2 \alpha_{1}}, \cdots,|z|^{2 \alpha_{1}} }_{m_1},
\underbrace{|z|^{2 \alpha_{2}}, \cdots,|z|^{2 \alpha_{2}} }_{m_2}, \cdots, 
\underbrace{|z|^{2 \alpha_{n_p}}, \cdots,|z|^{2 \alpha_{n_p}} }_{m_{n_p}}
\Big).
 \end{equation} 
One can also have log terms on the diagonal. For example, if $\rk(\cE)=3$ and the parabolic structure is given by $\cE_p=\cE_{p,1}\supset \cE_{p,2}\supset 0$ with $m_1=2$, then the following metric is also compatible 
\[h=\diag(|z|^{2\alpha_1}(\log|z|)^2,~|z|^{2\alpha_1}(\log|z|)^{-2},~|z|^{2\alpha_2}).\]

\begin{Theorem}[\sc{Nonabelian Hodge Correspondence} {\cite{SimpsonNoncompactcurves}}] \label{thm hitchin eq}
       Let $(\cE(\alpha),\Phi)$ be a parabolic Higgs bundle with
$\deg(\cE(\alpha))=0$. Then $(\cE(\alpha),\Phi)$ is polystable if and only if
there exists a acceptable compatible metric $h$ on $\cE(\alpha)$ with Chern
connection $A$ such that
        \begin{equation}
                \label{eq hitchin eq}F_{A}+[\Phi,\Phi^{*_h}]=0.
        \end{equation}
       Moreover, $D(\cE(\alpha),\Phi)=d_A+\Phi+\Phi^{*_h}$ defines a flat
connection on the underlying smooth bundle restricted to $X\setminus D.$
\end{Theorem}
 
\begin{Remark}
\label{Remark local form of solution}
         Equation \eqref{eq hitchin eq} is called the Hitchin equation. The local form of the metric $h$ solving the Hitchin equation for $(\cE(\alpha),\Phi)$ depends on the Jordan type of the Levi projection of the residue of the Higgs field, i.e., on $\mathrm{Res}(\cE(\alpha),\Phi)$ defined in \eqref{eq big Res on moduli}. In particular, it is of the form \eqref{eq local model no logs} if and only if the Levi projection of the residue is diagonalizable \cite{SimpsonNoncompactcurves}. Thus, in the cases of full flags or strongly parabolic Higgs fields, the solution metric will be of the form \eqref{eq local model no logs}. 
 \end{Remark} 
Let $E$ be the underlying smooth bundle of $\cE(\alpha)$, $h$ be a compatible
metric solving the Hitchin equations and let $d_A=\bar\partial_E+\partial_E^h$
be the $(0,1)$ and $(1,0)$ parts of the Chern connection.
The $(0,1)$ and $(1,0)$ parts of the flat connection $D(\cE(\alpha),\Phi)$ on
$E|_{X\setminus D}$ are given by $\bar\partial_E+\Phi^{*_h}$ and
$\partial_E^h+\Phi$, respectively.
In \cite[Theorem 2]{SimpsonNoncompactcurves} it is proven that the metric $h$ is
acceptable for the holomorphic structure $\bar\partial_E+\Phi^{*_h}$.
Hence, $h$
determines a parabolic structure $\cV(\beta)$ on an extension $\cV\to X$ of the
holomorphic bundle $(E_{X\setminus D},\bar\partial_E+\Phi^{*_{h}})$. Moreover,
it is shown that on which $\nabla=\partial_E^h+\Phi$ extends to a polystable
parabolic logarithmic connection connection on $\cV(\beta).$

The parabolic weights of $\cE(\alpha)$ and $\cV(\beta)$ and complex masses (see
Definition \ref{def complex masses}) of $\Phi$ and $\nabla$ are related by
Simpson's table \eqref{eq simpson table body}
\begin{equation}
        \label{eq simpson table body}
        \begin{tabular}{|c|c|c|}\hline
                &$(\cE(\alpha),\Phi)$& $(\cV(\beta),\nabla)$\\\hline parabolic weights& $\alpha$& $\beta= \alpha-2\mathrm{Re}(\bar\mu)$ \\\hline complex masses & $\mu$& $\nu=\alpha+\mu-\bar\mu$\\\hline
        \end{tabular}
\end{equation} 
\begin{Remark}\label{rem deriving the table}
This table is derived in \cite[\S 5]{SimpsonNoncompactcurves}. The key point is
the parabolic weight changes by subtracting twice the real part of the
eigenvalue of the residue of the term being added to $\bar\partial_E$
($\Phi^{*_h}$ in this case), and the complex mass changes by adding the
eigenvalue of the residue of $\partial_E^h$ ($\alpha$ in this case) and the
eigenvalue of the term being added to $\bar\partial_E$ ($\Phi^{*_h}$ in this
case).
It turns out the Jordan type of $\mathrm{Res}(\cE(\alpha),\Phi)$ and $\mathrm{Res}(\cV(\beta),\nabla)$ are the same. This finer structure is not be present in the strongly parabolic and full flag situations which we restrict to later. 
\end{Remark}

\begin{Remark}
        Simpson also proved the converse of Theorem \ref{thm hitchin eq}.
        Namely, given a polystable parabolic logarithmic connection
$(\cV(\beta),\nabla)$ on $(X,D)$, there is a compatible acceptable metric $h$
on $\cV(\beta)$ such that
$(\bar\partial_\cV,\nabla)=(\bar\partial_E+\Phi^{*_h},\partial_E^h+\Phi)$ for a
polystable parabolic Higgs bundle $(\bar\partial_E,\Phi).$ This direction of the
correspondence will not play a large role in this paper, so we do not say more
about it. As mentioned in the introduction, this correspondence does not define
a map from the moduli space $\cP_0(\alpha)$ to a moduli space $\cP_1(\beta)$ for
any fixed $\alpha,\beta.$
\end{Remark}

There is a $\lambda$-connection version of the above story. Namely, for $\lambda\in\C,$ the metric $h$ solving the Hitchin equations for $(\cE(\alpha),\Phi)$  defines a holomorphic $\lambda$-connection on $E|_{X\setminus D}$ given by
\[(\lambda,\widehat\cV_\lambda,\widehat\nabla_\lambda)=(\lambda,\bar
\partial_E+\lambda \Phi^{*_h},\lambda\partial_E^h+\Phi)\]
The metric $h$ determines a parabolic structure $\cV_\lambda(\beta_\lambda)$ on
the extension of $\widehat\cV_\lambda$, and $\widehat\nabla_\lambda$ extends to
a polystable parabolic logarithmic $\lambda$-connection $\nabla_\lambda$ on
$\cV_\lambda(\beta_\lambda).$ By Remark \ref{rem deriving the table}, the
parabolic weights and complex masses are related by the following
$\lambda$-connection analogue of Table \eqref{eq simpson table body}
\begin{equation}
        \label{eq lambda simpson table}
        \begin{tabular}{|c|c|c|}\hline
                &$(\cE(\alpha),\Phi)$& $(\cV_\lambda(\beta_\lambda),\nabla^\lambda)$\\\hline parabolic weights& $\alpha$& $\beta_\lambda=\alpha-2\mathrm{Re}(\lambda\bar\mu) $ \\\hline complex masses & $\mu$& $\nu_\lambda=\lambda\alpha+\mu-\lambda\bar\mu$\\\hline
        \end{tabular}
\end{equation} 

We are now ready to define the conformal limit of a stable Higgs bundle $(\cE(\alpha),\Phi)$. Since polystability is preserved by scaling the Higgs field, for $R\in\R^{>0}$, there is an compatible acceptable metric $h_R$ solving the Hitchin equations for $(\cE(\alpha),R\Phi)$. Hence, for each $\lambda\in\C$, we have the following family of holomorphic $\lambda$-connections defined on $E|_{X\setminus D}$
\[(\lambda,\bar\partial_E+\lambda R\Phi^{*_{h_R}},\lambda\partial_{h_R}+R\Phi)~.\]
 Scaling this family by $R^{-1}$ defines a family of holomorphic $R^{-1}\lambda$-connections
\[(R^{-1}\lambda,\bar\partial_E+\lambda R\Phi^{*_{h_R}},R^{-1}\lambda\partial_{h_R}+\Phi)~.\]
In terms of $\hbar=R^{-1}\lambda$, the above family is given by 
\begin{equation}
        D_{\hbar,R}(\cE(\alpha),\Phi)=(\hbar, \bar\partial_E+R^2\hbar\Phi^{*_{h_R}}, \hbar\partial_{h_R}+\Phi).
\end{equation}

\begin{Definition}[{\sc $\hbar$-conformal limit}]
        The $\hbar$-conformal limit of the polystable parabolic Higgs bundle $(\cE(\alpha),\Phi)$ is 
        \[\CL_\hbar(\cE(\alpha),\Phi)=\lim_{R\to 0}D_{\hbar,R}(\cE(\alpha),\Phi)~.\]
\end{Definition}

It is important to note that the above limit is taken in the space of holomorhpic $\hbar$-connections on $E|_{X\setminus D}$. Thus, if the limit exists, it is a holomorphic $\hbar$-connection on $E|_{X\setminus D}$. By Remark \ref{rem deriving the table}, the analogue of Simpson's table for $D_{\hbar,R}$ is given by
\begin{equation}
        \label{eq cf fam simpson table}
        \begin{tabular}{|c|c|c|}\hline
                &$(\cE(\alpha),\Phi)$& $D_{\hbar,R}$\\\hline parabolic weights& $\alpha$& $\beta_{\hbar,R}=\alpha-\mathrm{Re}(\hbar R^2\bar\mu) $ \\\hline complex masses & $\mu$& $\nu_{\hbar,R}=\hbar\alpha+\mu-\hbar R^2\bar\mu$\\\hline
        \end{tabular}
\end{equation} 
\begin{Remark}
Note that the extension of $D_{\hbar,R}$ from a holomorphic $\hbar$-connection on $X\setminus D$ to a parabolic logarithmic $\hbar$-connection on $X$ is determined by the  metric $h_R$. The extension of the conformal limit is determined by the metric $h_0$ at the Hodge bundle associated to $(\cE(\alpha),\Phi)$. 
\end{Remark}

\section{Analytic theory} \label{sec:analytic}

The aim of this section is to prove Theorem \ref{thm:analytic-moduli} where we provide a gauge theoretic
construction of the  moduli spaces, as complex manifolds,  of stable strongly parabolic bundles in general,
and  stable parabolic bundles \emph{in the case of full flags}, i.e.\
Assumption A. The Kuranishi slice
method gives local universal families, showing that these moduli spaces are
analytically isomorphic to the moduli spaces   discussed in Section
\ref{sec:moduli}. 
We follow the same method to produce in Theorem \ref{thm:analytic-DRmoduli}
the de Rham space of logarithmic connections. 

Such a construction has previously been carried out 
for parabolic vector bundles, strongly parabolic Higgs bundles, rank $2$
parabolic Higgs bundles,  and wild Higgs bundles under certain assumptions 
(see
\cite{Biquard:91,Konno:93,Nakajima:96,DaskalWentworth:97,Biquard:97,BiquardBoalch:04,Mochizuki:06}).
The constructions generally use weighted Sobolev spaces, whose introduction into
gauge theory goes back to Mrowka and Taubes \cite{MrowkaThesis,Taubes:87,Taubes:93}.  
The main result below is an extension of these constructions
 to the case of parabolic Higgs bundles of arbitrary rank  where the
 complex masses are allowed to vary. We note that the variation of
 parabolic weights has been treated in \cite{KimWilkin:18}.

An intermediate construction is that of the \emph{framed} moduli space of
parabolic bundles. We show that  this admits a holomorphic symplectic
form that is invariant under the change of framing. The moduli space of parabolic bundles 
therefore inherits a Poisson structure as a quotient, and this recovers earlier results of 
Logares--Martens \cite{LogaresMartens} (see also
\cite{Bottacin:95,Markman:94}).
We prove that the symplectic leaves are exactly the moduli spaces with
fixed complex masses. Moreover, we show that the holomorphic symplectic structure is
actually hyperk\"ahler, thus generalizing the result of Konno and
Nakajima in the cases cited above.

\subsection{Gauge theory with weighted Sobolev spaces}

\subsubsection{Connections and gauge groups} \label{sec:gauge-groups}
In this section we define the space of connections and  the group of gauge
transformations.
The material here is largely based on \cite{Matic:88},
\cite{BiquardThese, Biquard:91},  and 
\cite[Sec.\ 3]{DaskalWentworth:97}, but specific details are required for the
application to conformal limits.
 For the sake of completeness,  we therefore give a  precise and
 self-contained presentation. 

We continue with the notation of \S\ref{sec:moduli}. 
Set $X^\times:=X\setminus D$, and 
let $E\to X^\times$ be a trivial
rank $n$ bundle with hermitian metric $h_0$.  Denote the endomorphism bundle
of $E$ by $\glfrak_E$, and the bundle of skew-hermitian endomorphisms
by $\ufrak_E$. Their traceless versions are denoted by
$\slfrak_E$ and $\sufrak_E$, respectively.
For each $p\in D$,
  choose  local conformal coordinates $z$
on open  disks $\Delta(p)$ centered at $p$. We assume $\overline\Delta(p)\cap
\overline\Delta(p')=\emptyset $ for $p\neq p'$, 
so that $X_0:=X\setminus \cup_{p\in D}\Delta(p)$ is a Riemann surface with
boundary. 
Let $\Delta^\times(p):=\Delta(p)\setminus\{p\}$. 
Write $z=re^{i\theta}$, and identify $\Delta^\times(p)$ with a
semi-infinite cylinder $C(p)$ via: $(r,\theta)\mapsto
(\tau,\theta)$, where $\tau=-\log r$ (see  Figure
\ref{fig:cylindricalend}).
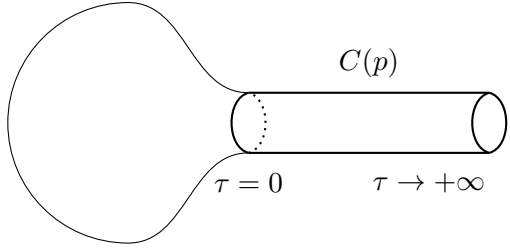
\begin{figure}
    \begin{tikzpicture}[scale=0.8]
\draw  
(0,.5) to [out=180, in=0] (-2,2) to [out=180, in=90] (-4,0) to
    [out=-90, in=180]  (-2,-2)  to [out=0, in=180] (0,-.5);
\draw [thick] (0,0.5) -- (4,0.5);
\draw [thick] (0,-0.5) -- (4,-0.5);
\draw [thick, dotted] (0,-0.5) to [out=15, in=-15] (0,0.5);
\draw [thick] (0,-0.5) to [out=165, in=-165] (0,0.5);
\draw [thick] (4,-0.5) to [out=15, in=-15] (4,0.5);
\draw [thick] (4,-0.5) to [out=165, in=-165] (4,0.5);
\node at (0, -1) {$\tau=0$};
\node at (3, -1) {$\tau\to+\infty$};
    \node at (2, 1) {$C(p)$};
\end{tikzpicture}
\caption{\label{fig:cylindricalend} Cylindrical end}
\end{figure}
We suppose that $X^\times $ is endowed with a fixed conformal metric 
that is euclidean $d\tau^2+d\theta^2$ on every $C(p)$. Let $d\mu$
denote the associated area form. The Lefschetz operator
$\Lambda:\Omega^2(X^\times)\to \Omega^0(X^\times)$ is the complex linear
extension obtained from setting $\Lambda(d\mu)=1$, and by convention $\Lambda$   
vanishes on functions and $1$-forms.
We also extend  the collection of coordinates $\tau$ on the union of
cylinders $C(p)$ to a smooth function, also denoted $\tau$,  on  all of $X^\times$.
For a function $f$ on $X^\times$, by 
$\displaystyle
\lim_{\tau\to +\infty} f(\tau,\theta)
$
we mean the collective limits on each $C(p)$, $p\in D$, when they exist.

We also choose a fixed framing of $E$ on each $C(p)$; that is, a unitary
frame $\{e_i(p)\}_{i=1}^n$ for each $p\in D$ (we will typically denote this
simply $\{e_i\}_{i=1}^n$ when the point is implicit). 
Fix the data $\alpha(p)$ for each $p\in D$ as in \S\ref{sec:parabolic-bundles}. 
Let $A_0$ denote a fixed unitary  connection  which,
in each local  frame $e_1,\ldots, e_n$  on $C(p)$,
has the form $d_{A_0}=d+\sqrt{-1}\hat\alpha(p) d\theta$, where 
$$
\hat\alpha(p)=\left(\begin{matrix}\alpha_1(p)&& \\
&\ddots&\\&&\alpha_{n_p}(p)\end{matrix}\right)
$$
As before, $0\leq \alpha_1(p)< \cdots< \alpha_{n_p}(p)<1$, and in the
matrix above each $\alpha_j(p)$ is repeated $m_j(p)$ times. 
The frame $\{e_i(p)\}$ gives an identification of the restriction of the
bundle $\glfrak_E$ to $C(p)$ with $\glfrak_n$. The 
collections $\{\alpha(p), m(p)\}$ define  parabolic,  Levi, and unipotent
subalgebras $\pfrak_p,\lfrak_p,\ufrak_p \subset\glfrak_n$, 
respectively.

We denote by $\nabla_0$ the covariant derivative on $E$-valued
tensors obtained from the Levi-Civita connection on $X^\times$ and $A_0$. 
For $\delta\in \RBbb$, define the weighted Sobolev spaces
$L^p_{k,\delta}$ of sections of $E$ (and associated bundles)
by completing the space $C^\infty_0(E)$ of smooth compactly  supported
sections on $X^\times$ in the norm
$$
\Vert
\sigma\Vert_{L^p_{k,\delta}}=\left\{\int_{X^\times}d\mu\,
e^{\tau\delta}\left(
|\nabla_0^{(k)}\sigma|^p+\cdots+|\nabla_0\sigma|^p+|\sigma|^p\right)
\right\}^{1/p}
$$
where norms $|\cdot|$ will be understood  to be taken with respect to the
background hermitian structure on $E$ and the conformal metric on $X$. 
Weighted Sobolev spaces have the usual embedding and
multiplication properties (cf.\ \cite[Lemma 7.2]{LockhartMcOwen:85}). 

Define the space of unitary connections:
\begin{equation} \label{eqn:connection-space}
\Acal_{\delta}=d_{A_0}+ L^2_{1,\delta}\left(\sufrak_E\otimes T^\ast X^\times\right)
\end{equation}
The complexification of this bundle splits
$$(\sufrak_E\otimes T^\ast X)\otimes \CBbb\simeq (\slfrak_E\otimes K_X)\oplus
(\slfrak_E\otimes\overline K_X)$$
and for $A\in \Acal_\delta$, 
the decomposition into $(1,0)$ and $(0,1)$ type 
will be denoted: $d_A=\partial_A+\dbar_A$. 
Formal $L^2$-adjoints are denoted $d_A^\ast$, $\partial_A^\ast$, and $\dbar_A^\ast$, 
and we have the K\"ahler identities:
\begin{equation} \label{eqn:kahler}
    \dbar_A^\ast=-i[\Lambda,\partial_A]\quad ,\quad
    \partial_A^\ast=+i[\Lambda,\dbar_A] \ .
\end{equation}
On $\Delta^\times(p)$, a holomorphic frame for $\dbar_{A_0}$ is given
by $s_i=|z|^{\alpha_k(p)}e_i$, $k=1,\ldots, n_p$,  for 
$$
\sum_{j\leq k-1} m_j(p) < i\leq\sum_{j\leq k} m_j(p) 
$$
(we have set $m_0(p)=0$).

Consider the following spaces.
\begin{align*}
    \Rcal_{\delta}&=\left\{ \eta\in L^2_{2,loc.}(\glfrak_E)
    \mid \Vert d_{A_0}\eta\Vert_{L^2_{1,\delta}}<+\infty\right\}\\
    \Hcal_{\delta}&=\left\{ \eta\in \Rcal_{\delta} \mid
d_{A_0}^\ast(e^{\tau\delta}d_{A_0}\eta)=0\right\}
\end{align*}
We refer to $\Hcal_{\delta}$ as infinitesimal \emph{harmonic} gauge
transformations.
Let $\displaystyle\nabla_{\partial_p}=\frac{d}{d\theta}+i\hat\alpha(p)$ denote the
boundary operator (i.e.\ the restriction of $d_{A_0}$) acting on sections of $E$ (and the associated bundle $\glfrak_E$)
over the component of $\partial X_0$ intersecting $C(p)$. 
\begin{Lemma} \label{lem:boundary-kernel}
The kernel $\ker\nabla_{\partial_p}$ consists of the constant sections in
    $\lfrak_p\subset\glfrak_E$.
\end{Lemma}

\begin{proof}
The Lie algebra $\lfrak_p$ may be identified with the
centralizer in $\glfrak_n$ of $\hat\alpha(p)$.  If 
$\psi\in \ker\nabla_{\partial_p}$, then 
 write $\psi(\theta)=\sum_{m\in \ZBbb}\psi_m e^{im\theta}$, 
 where $\psi_m\in \glfrak_n$. It follows that
    $m\psi_m+[\widehat\alpha(p), \psi_m]=0$,
 or in terms of the unitary frame $\{e_i\}$, 
    $(m+\alpha_j(p)-\alpha_k(p))(\psi_m)_{jk}=0$. 
 This implies $\psi_m=0$ if $m\neq 0$, i.e.\ $\psi(\theta)$ is constant, and it 
 commutes with $\widehat\alpha(p)$. 
\end{proof}

\begin{Remark} \label{rem:spectrum}
 More generally,  the spectrum  of $\nabla_{\partial_p}$ acting on sections
    of the bundle $\glfrak_E$ restricted to the boundary consists of 
    $$\{ \sqrt{-1}\lambda^m_{ij}\mid
    \lambda^m_{ij}=m+\alpha_i(p)-\alpha_j(p) \ ,\  m\in \ZBbb\}\ .$$ 
In particular, the spectrum is symmetric about the origin.
\end{Remark}

 Let $\lambda(p)$ denote the smallest (positive) nonzero eigenvalue
 of $\nabla_{\partial_p}$.
 Explicitly, 
\begin{align*} 
    \lambda(p) &=\min \left\{ |\lambda^m_{ij}|\mid m,i,j,\
    \lambda^m_{ij}\neq 0\right\} \\
    &=\min\left\{ \{|\alpha_i(p)-\alpha_j(p)|\mid  \alpha_i(p)\neq\alpha_j(p) \}
, \{ 1-|\alpha_i(p)-\alpha_j(p)|\mid i,j=1,\ldots,n \}\right\}
\end{align*}
 We shall  choose
$\delta$ to satisfy
 \begin{equation} \label{eqn:delta-assumption}
     0<\delta< \min_{p\in D} \lambda(p)
 \end{equation}
 This is the  key assumption made  throughout this section.
The definition of gauge groups relies on the following result.

\begin{Proposition} \label{prop:boundary-map}
There is a direct sum decomposition 
$
    \Rcal_{\delta}=L^2_{2,\delta}(\glfrak_E)\oplus
    \Hcal_{\delta}
$.
Moreover, there is a continuous map 
    $$\bbold:\Rcal_{\delta}\lra
\ker\nabla_\partial : \eta\mapsto\lim_{\tau\to
+\infty}\eta(\tau,\theta)$$ 
The map $\bbold$ satisfies
$\bbold^{-1}(0)=L^2_{2,\delta}(\glfrak_E)$, and
    the restriction  $\bbold:\Hcal_{\delta}\to
\ker\nabla_\partial$ is an isomorphism. 
\end{Proposition}

\begin{proof}
 We first define  
 the isomorphism 
 \begin{equation}\label{eqn:b}
     \bbold: \Hcal_{\delta}\isorightarrow \bigoplus_{p\in D} \lfrak_p
 \end{equation}
    Let $\eta\in \Hcal_{\delta}$ and consider its restriction to $C(p)$.
 Write $\eta$ locally as $\sum_{\lambda}
    f_\lambda(\tau)\psi_\lambda(\theta)$, where $\psi_\lambda$ is an
    eigenfunction of the boundary operator with eigenvalue
    $\sqrt{-1}\lambda$, $\lambda\in \RBbb$ (we suppress the notation for
    multiplicities in the eigenvalues). Then harmonicity of $\eta$ implies
 $$
(e^{\tau\delta}f'_\lambda)'=\lambda^2 e^{\tau\delta}f_\lambda\ .
 $$
    The general solution is:
$$f_\lambda =
    c_\lambda^\pm\exp\left(\frac{\tau}{2}(-\delta\pm\sqrt{4\lambda^2+\delta^2})\right)\
    ,\ c^\pm_\lambda \text{ constant.}
    $$
 Now since $d_{A_0}\eta\in L^2_\delta$, we have $f'_\lambda\in
L^2_\delta$. 
This means that $c^+_\lambda=0$, unless $\lambda=0$. Hence, we may define 
$$
\bbold(\eta):=c_0^+=
\lim_{\tau\to +\infty}\eta(\tau,\theta)
$$
which exists on
each $C(p)$ and lies in $\lfrak_p$. 
 Notice that $e^{\tau\delta}d_{A_0}\eta$ is also  bounded. If $\eta\in
    \ker \bbold$ we can integrate by parts
 $$
    0=\langle \eta, d^\ast_{A_0}(e^{\tau\delta}d_{A_0}\eta)\rangle_{L^2}
=\Vert d_{A_0}\eta\Vert_{L^2_\delta}^2
 $$
    where $d_{A_0}^\ast$ denotes the formal $L^2$ adjoint of $d_{A_0}$, 
 so $\eta$ is covariantly constant. But since it also vanishes on the boundary
    it must vanish identically. Therefore, $\bbold$ is injective. 

    Now suppose $\eta\in \Rcal_{\delta}$. 
The operator $\dbar_{A_0}$ 
    has closed range (cf.\ \cite[Thm.\ 1.3]{LockhartMcOwen:85}, and note the restriction
    \eqref{eqn:delta-assumption}), and so
we may apply the Hodge theorem to write:
$$
\dbar_{A_0}\eta=\beta+ \dbar_{A_0}u 
$$
where $\beta$ is harmonic with respect to the weighted  adjoint:
$\dbar_{A_0}^\ast(e^{\tau\delta}\beta)=0$,
and $u\in L^2_{1,\delta}$. 
Then 
$\dbar_{A_0}^\ast(e^{\tau\delta}\dbar_{A_0}(\eta-u))
=0$.
Since the curvature $F_{A_0}=0$ on the cylinders, 
    a comparison of the Laplacians for $\dbar_{A_0}$ and $d_{A_0}$
    on the cylinder involves only the boundary
operator. 
An argument like the one above then  says that $\eta-u$ has a well-defined
limit as $\tau\to +\infty$.  Hence, $\bbold(\eta)$ is well defined on all of
    $\Rcal_{\delta}$ by setting 
$$
\bbold(\eta):=
\lim_{\tau\to +\infty}(\eta-u)(\tau,\theta)
$$
This agrees with the previous definition in the case $\eta\in
\Hcal_{\delta}$, since then $\dbar_{A_0}u=0$, and  it is
straightforward to show that  $\lim_{\tau\to
+\infty}u(\tau,\theta)=0$.
If $c\in \lfrak_p$, then let $\eta$ be a smooth section of $\glfrak_E$ that
is constant $=c$ on $C(p)$, and vanishes on $C(p')$ for all $p'\neq p$.
Since $d_{A_0}\eta$ is then compactly supported, it follows that $\eta\in
\Rcal_{\delta}$. Since $u\in L^2_{1,\delta}$, it is easy to show that
$\bbold(\eta)=c$,  and
this proves surjectivity of $\bbold$. 
Now if $\eta\in \Rcal_{\delta}\cap\ker\bbold$, then $\eta\in L^2_{2,\delta}$.
It is also easy to show that  elements $\eta\in L^2_{2,\delta}$ are
continuous and $\displaystyle\lim_{\tau\to+\infty}e^{\tau\delta}\eta(\tau,\theta)=0$ (cf.\ \cite[p.\
222, \emph{Remarque}]{Biquard:91}).  
Hence, the kernel of $\bbold$ is exactly $L^2_{2,\delta}$. 

Finally, we claim that for any $\eta\in \Rcal_{\delta}$  there is
$u\in L^2_{2,\delta}$ such that $\eta+u\in \Hcal_{\delta}$. To
show this, let us first note the Poincar\'e inequality: there is
a constant $C>0$ so that
\begin{equation} \label{eqn:poincare}
\Vert u\Vert_{L^2_\delta}\leq C\Vert d_{A_0} u\Vert_{L^2_\delta}
\end{equation}
for all $u\in L^2_{1,\delta}$. For if such an inequality did not
hold, then 
 we could find a sequence $u_j\in L^2_{1,\delta}$
with $\Vert u_j\Vert_{L^2_\delta}=1$, and $\Vert d_{A_0}u_j\Vert_{L^2_\delta}\to 0$. 
Notice that this implies that  $\Vert u_j\Vert_{L^2_{1,\delta}}$ is
uniformly bounded. We may therefore assume 
that $u_j\to u$ weakly in $L^2_{1, \delta}$. 

We shall use the following several times:
Fix a smooth function $\phi$ on $\RBbb$ such that
$$
\phi(x)=\begin{cases} 0 & x\leq 0 \\ 1 & x\geq 1
\end{cases}
$$
    For $R\geq 0$,  define the cut-off function $\phi_R\in C^\infty(X^\times)$ by
    \begin{equation} \label{eqn:phiR}
    \phi_R(z):=\phi(\tau(z)-R)
    \end{equation}

    For $\varepsilon>0$,
we may choose $R$ sufficiently large  so that $\Vert
(d\phi_R)u\Vert_{L^2_\delta}\leq \varepsilon$.
By the compact embedding $L^2_1\hookrightarrow L^4$ on the
complement of the cylinders, we may assume 
$(1-\phi_R)u_j\to (1-\phi_R)u$ strongly in $L^2_\delta$.
We also have
$$
\Vert d_{A_0}(\phi_R u_j)\Vert_{L^2_{\delta}}\lra \Vert
(d\phi_R)u\Vert_{L^2_\delta}\ .
$$
Now the inequality \eqref{eqn:poincare} holds on
each $C(p)$ by an integration by parts argument (cf.\ \cite[Th\'eor\`eme
1.2]{Biquard:91}). 
Applied to $\phi_R u_j$,  we see that if $\varepsilon$ is chosen sufficiently
small, and $R$ accordingly, then $\Vert\phi_R u_j\Vert_{L^2_\delta}\leq 1/2$
for large $j$, and
therefore $\Vert (1-\phi_R)u\Vert\geq 1/2$.
This now is a contradiction, because we must have
$d_{A_0}u=0$, so $|u|$ is constant in $L^2_\delta$, and therefore zero. 

With this understood, choose $u_j\in L^2_{2,\delta}$
so that
$$
\Vert d_{A_0}(\eta+u_j)\Vert_{L^2_{1,\delta}}\lra L:= \inf_{v\in L^2_{2,\delta}}
\Vert d_{A_0}(\eta+v)\Vert_{L^2_{1,\delta}}
$$
Using \eqref{eqn:poincare}, we have a uniform bound on $\Vert
u_j\Vert_{L^2_{2,\delta}}$.  We can therefore extract a 
subsequence (denoted the same) so that $u_j\to u$ weakly in
$L^2_{2,\delta}$. 
Let $\widetilde \eta=\eta+u$.  
Then $\Vert d_{A_0}(\widetilde\eta)\Vert_{L^2_{1,\delta}}=L$. 
In particular, $\langle d_{A_0}\widetilde \eta,
d_{A_0}v\rangle_{L^2_{1,\delta}}=0$ for all smooth, compactly
supported $v$.
It follows that $\widetilde\eta\in \Hcal_{\delta}$. 
This completes the proof of  Proposition \ref{prop:boundary-map}.
\end{proof}

Later on we will need  the following remark. Clearly, the definition of
the boundary map $\bbold$ is valid for any $\eta\in L^2_{2,loc.}$,
$d_{A_0}\eta\in L^2_{1,\delta}$, that is harmonic on each portion of
$C(p)$ where $\tau>R$, for some $R\geq 0$. 
Then elements  $\ell\in \lfrak_p$ may be extended to smooth sections of
$\glfrak_E$ as $\phi_R\cdot\ell$. 
Note that $\bigoplus_{p\in D}\lfrak_p$ has an
invariant metric. The next result is
straightforward from the eigensection expansion used above. We omit the proof.

\begin{Lemma} \label{lem:delta-minus-delta}
    Suppose
    $\eta\in L^2_{2,loc.}(\glfrak_E)$,
$d_{A_0}\eta\in L^2_{1,\delta}(\glfrak_E)$, is harmonic on each portion of
    $C(p)$ where $\tau>R$, for some $R\geq 0$. Then there is a constant
    $C(R)$ depending upon $R$ but not $\eta$, so that:
\begin{enumerate}
    \item $|\bbold(\eta)|\leq C(R)\Vert\eta\Vert_{L^2_{-\delta}}$;
    \item $\Vert \eta-\phi_R\cdot \bbold(\eta)\Vert_{L^2_\delta}\leq
        C(R)\Vert \eta\Vert_{L^2_{-\delta}}$.
\end{enumerate}
\end{Lemma}

The following is  also clear.
\begin{Lemma}
The norm
$$
\Vert
    \eta\Vert^2_{\Rcal_{\delta}}=\Vert d_{A_0}\eta\Vert^2_{L^2_{1,\delta}}+
 |\bbold(\eta)|^2
$$
    gives $\Rcal_{\delta}$ a Banach space structure for which the
    projections onto $L^2_{2,\delta}(\glfrak_E)$ and $\Hcal_{\delta}$ 
are continuous. Moreover, pointwise multiplication
    $\Rcal_{\delta}\times\Rcal_{\delta}\to\Rcal_{\delta}$ is
    well-defined and continuous, and $\bbold:\Rcal_{\delta}\to
\ker\nabla_{\partial}$ is a continuous linear map. 
\end{Lemma}

We are now in a position to define the gauge groups.
Let
$$
    \Gcal_{\delta}:=\left\{ \eta\in \Rcal_{\delta} \mid
\det\eta=1\right\}\quad ,\quad
    \Gcal_{\delta, \ast}:=\left\{ \eta\in \Gcal_{\delta} \mid
\bbold(\eta)=I\right\}
$$
Then $\Gcal_{\delta}$ and  $\Gcal_{\delta,\ast}$ are complex Banach Lie
groups with Lie algebras 
\begin{equation}\label{eq:Rdelta0}
\Lie \Gcal_{\delta}=\Rcal_\delta^0:=\{\eta\in  \Rcal_{\delta}\mid \tr \eta=0\}\quad ,\quad
    \Lie \Gcal_{\delta,\ast}=L^2_{2,\delta}(\slfrak_E)
 \end{equation}
It is clear that with these definitions there is a smooth action of
$\Gcal_{\delta}$ on $\Acal_{\delta}$ defined by pullback:
$\dbar_{g(A)}:=g\circ\dbar_A\circ g^{-1}$.
Note also that  $\Gcal_{\delta,\ast}\subset \Gcal_{\delta}$ is a normal
subgroup,
and the center $Z\subset\SL(2,\CBbb)$ embeds into $\Gcal_\delta$. 
We denote the 
quotient
\begin{equation} \label{eqn:quotient-group-center}
    \overline\Lbold:=\Gcal_{\delta}/Z\times \Gcal_{\delta,\ast}
\end{equation}

Notice that the defining conditions for $\Rcal_\delta$ are closed with
respect to hermitian conjugation.
For future reference, we  therefore define the closed Banach subgroups of
$\Gcal_{\delta}$ (resp.\ $\Gcal_{\delta,\ast}$):
$$
\Kcal_\delta=\{ g\in \Gcal_\delta \mid gg^\ast=I\} \quad ,\quad
\Kcal_{\delta,\ast}=\{ g\in \Gcal_{\delta,\ast} \mid gg^\ast=I\}
$$
These have Lie algebras
$$
\Lie \Kcal_{\delta}=\Ucal_{\delta}:=\{\eta\in  \Rcal^0_{\delta}\mid
\eta=-\eta^\ast\}\quad ,\quad
    \Lie \Kcal_{\delta,\ast}=L^2_{2,\delta}(\sufrak_E)
    $$

We end this section by discussing holomorphic structures. 
In the following, $K_X$ (resp.\ $\overline K_X$) denotes the complex line 
bundle of $(1,0)$-forms (resp.\ $(0,1)$-forms) on
$X$. The conformal metric on $X^\times$ gives a hermitian structure on
$K_X$, $\overline K_X$, on $X^\times$. 
The  Levi-Civita connection
and $A_0$ give a connection, still denoted $A_0$, on tensors on
associated bundles to $E$ tensored by $K_X$ or $\overline K_X$. With this
understood, we have

\begin{Lemma} \label{lem:smooth}
    Given $A\in \Acal_{\delta}$, there is $g\in
\Gcal_{\delta,\ast}$ such that $g(A)$ is smooth, and $\dbar_{g(A)}=\dbar_{A_0}$
on each $C(p)$. 
\end{Lemma}

\begin{proof}
    This is essentially \cite[Prop.\ II.7 and Lemme III.5]{BiquardThese},
    which itself is modeled on \cite[Lemma 14.8]{AtiyahBott:82}. 
The key point is to show that 
$$
\dbar_A: 
L^2_{2,\delta}(\slfrak_E)\lra L^2_{1,\delta}(\slfrak_E\otimes \overline K_X)
$$
is Fredholm.
    By \cite[Thm.\ 1.1]{LockhartMcOwen:85}, the corresponding operator
    (which is translation invariant on each $C(p)$)
$$
\dbar_{A_0}: 
L^2_{2,\delta}(\slfrak_E)\lra L^2_{1,\delta}(\slfrak_E\otimes \overline K_X)
$$
is Fredholm. 
There is a constant $\varepsilon>0$ so that $\dbar_{A_0}+B$
is Fredholm (of the same index), for any bounded map $B$ with $\Vert B\Vert <\varepsilon$ (\cite[Thm.\ 2.9]{Schechter:67}). 
Write $\beta=\dbar_A-\dbar_{A_0}\in L^2_{1,\delta}(\slfrak_E\otimes\overline K_X)$.
    Then we have the continuous inclusion $L^2_{1,\delta}\hookrightarrow L^4_{\delta}$ and multiplication maps $L^4_\delta\times L^4_\delta\to L^2_\delta$. 
    Recall the cut-off function $\phi_R$. Since there is a constant $C$
    (independent of $\eta$) satisfying
    \begin{equation} \label{eqn:sup_eta} 
    \sup |\eta|\leq C\Vert \eta\Vert_{L^2_{2,\delta}} \ ,
    \end{equation} 
    we can choose $R$ sufficiently large (depending upon $\beta$) so that
    the map 
$$
L^2_{2,\delta}(\slfrak_E)\lra L^2_{1,\delta}(\slfrak_E\otimes\overline K_X)
: \eta\mapsto \phi_R[\beta,\eta]
$$
     is bounded of 
    norm less than $\varepsilon$. Now the map 
$\eta\mapsto (1-\phi_R)[\beta,\eta]$
    defines a compact operator between the same spaces. This follows
 from the compact inclusion $L^2_1\hookrightarrow L^4$ on the support of
    $1-\phi_R$. 
 Since adding a compact operator to a Fredholm operator is still Fredholm with the same index
 (\cite[Thm.\ 2.10]{Schechter:67}), the result follows by writing
 $\dbar_A=\dbar_{A_0}+(1-\phi_R)\cdot\beta+\phi_R\cdot\beta$.
\end{proof}

\subsubsection{Higgs fields}
We begin with the key definition.
\begin{Definition}[{\sc admissible Higgs fields}]
The set of admissible  Higgs fields is defined to be
$$
\Dscr_\delta(\slfrak_E\otimes K_X)=\left\{ \Phi\in
L^2_{-\delta}(\slfrak_E\otimes K_X) 
\mid \dbar_{A_0}\Phi\in L^2_{\delta}(\slfrak_E\otimes K_X\otimes \overline K_X)
 \right\}
$$
\end{Definition}
Since $\dbar_{A_0}$ is a  closed operator, 
the  domain 
$\Dscr_\delta\subset L^2_{-\delta}$ becomes a Banach space with respect to
the graph norm, which we denote by 
\begin{equation} \label{eqn:graph_norm}
    \Vert\Phi\Vert^2_{\Dscr_\delta}:=
    \Vert\Phi\Vert^2_{L^2_{-\delta}}+\Vert\dbar_{A_0}\Phi\Vert^2_{L^2_{\delta}}
    \ .
    \end{equation}
From the inclusion $L^2_\delta\hookrightarrow L^2_{-\delta}$, there is a continuous embedding 
\begin{equation}\label{eqn:stupid_embedding}
L^2_{1,\delta}(\slfrak_E\otimes K_X)\hookrightarrow
    \Dscr_\delta(\slfrak_E\otimes K_X)\ . 
\end{equation}

We will need the following useful decomposition.
\begin{Proposition}[{\sc Decomposition of Higgs fields}]
\label{prop:higgs-decomposition}
    Given $\Phi\in \Dscr_{\delta}(\slfrak_E\otimes K_X)$, there is an expression $\Phi=\Phi_0+\Phi_1$,
    where $\Phi_1\in L^2_{1,\delta}$, and $\Phi_0$ is continuous and
    bounded.  There a well-defined limit
    $$
    \lim_{\tau\to+\infty}\Phi_0=\ell\in \bigoplus_{p\in D}\lfrak_p
    $$
    where by the limit  we mean with respect to the  trivialization $dz/z$
    of $K_X$ and the fixed unitary frame $\{e_i\}$ on $C(p)$.
    Moreover, there is a constant $C$, depending on $A_0$ and $\delta$ but
    independent of $\Phi$, such that
    \begin{align}
        |\ell|\leq \sup|\Phi_0|&\leq
    C\Vert\Phi\Vert_{\Dscr_{\delta}}
        \ ; \label{eqn:phi0_graph} \\
        \Vert\Phi_1\Vert_{L^2_{1,\delta}}&\leq 
        C\Vert\Phi\Vert_{\Dscr_{\delta}} \ . \label{eqn:phi1_graph}
    \end{align} 
   In particular,   this defines a continuous map
$$
\rbold :
\Dscr_{\delta}(\slfrak_E\otimes K_X)\lra  \bigoplus_{p\in D}\lfrak_p
: \Phi\mapsto\rbold(\Phi):= \ell
$$
   with $\ker \rbold = L^2_{1,\delta}(\slfrak_E\otimes K_X)$.
\end{Proposition}

\begin{Remark} We emphasize that while the decomposition $\Phi=\Phi_0+\Phi_1$
in Proposition \ref{prop:higgs-decomposition} is not unique, the limit $\ell$
is.
\end{Remark}

\begin{proof}[Proof of Proposition \ref{prop:higgs-decomposition}]
    By the definition of $\Dscr_\delta$ we
    have $\dbar_{A_0}\Phi\in L^2_{\delta}$. We may therefore use the
    Hodge decomposition as in the previous section to write:
    $$
    \dbar_{A_0}\Phi=\Omega+\dbar_{A_0}\psi
    $$
    with $\Omega$ harmonic and $\psi\in L^2_{1,\delta}$.
    Then 
    $$
    \dbar_{A_0}^\ast(e^{\tau\delta}\dbar_{A_0}(\Phi-\psi))=0
    $$
    Trivializing by $dz/z$
    and using 
   the same analysis as in the previous section, we see that
    $\Phi-\psi$ has a well defined limit $\ell$ as $\tau\to +\infty$, where $\ell$ lies in the Levi factors.  
    Moreover, $\ell=0$ implies $\Phi\in L^2_{1,\delta}$. 
    The first assertion of the Proposition  then follows by setting
    $\Phi_0=\Phi-\psi$ and $\Phi_1=\psi$. 
   By Lemma \ref{lem:delta-minus-delta}, we have
    \begin{equation} \label{eqn:l-estimate}
        |\ell|\leq C \Vert \Phi_0\Vert_{L^2_{-\delta}}
    \end{equation}
   Next, since $\Omega$ is the harmonic projection, 
   \begin{equation} \label{eqn:harmonic-projection}
       \Vert \Omega\Vert_{L^2_\delta}\leq \Vert
       \dbar_{A_0}\Phi\Vert_{L^2_\delta}
   \end{equation}
Consider the Laplacian
$$
\dbar_{A_0} \dbar_{A_0}^{\ast_\delta}
    : L^2_{2,\delta}(\slfrak_E\otimes K_X\otimes\overline K_X)\lra  L^2_{\delta}(\slfrak_E\otimes K_X\otimes\overline K_X)
    \ ,
$$
    where 
    $\dbar_{A_0}^{\ast_\delta}:= e^{-\tau\delta}\dbar_{A_0}^\ast e^{\tau\delta}$.
    By 
    \cite[Thm.\ 1.1]{LockhartMcOwen:85}, this is  a Fredholm operator, and in
    fact it has index $0$.  Let $G^{(1)}_\delta$ denote its Green's operator.
    Then by definition, $\Phi_1=\dbar_{A_0}^{\ast_\delta}
    G^{(1)}_\delta(\dbar_{A_0}\Phi)$.  It follows from the elliptic estimate that
   \begin{equation} \label{eqn:phi1}
       \Vert \Phi_1\Vert_{L^2_{1,\delta}}\leq C\Vert
       \dbar_{A_0}\Phi\Vert_{L^2_\delta}
   \end{equation}
    which gives \eqref{eqn:phi1_graph}. 
   This  also implies trivially that
   \begin{equation} \label{eqn:phi0}
       \Vert \Phi_0\Vert_{L^2_{-\delta}} \leq 
       \Vert \Phi\Vert_{L^2_{-\delta}}+\Vert \Phi_1\Vert_{L^2_{-\delta}}
       \leq 
       \Vert \Phi\Vert_{L^2_{-\delta}}+C\Vert \dbar_{A_0}\Phi\Vert_{L^2_\delta}    
       \leq C\Vert \Phi\Vert_{\Dscr_{\delta}} \ ,
   \end{equation}
    and the estimate on $\ell$ in \eqref{eqn:phi0_graph} follows from this and \eqref{eqn:l-estimate}.

    To prove the estimate on $\Phi_0$ in \eqref{eqn:phi0_graph}, let
    $\widetilde\ell=\phi_R\cdot\ell$ be the smooth extension
    of the constant sections $\ell_p\otimes (dz/z)$  on each $C(p)$, obtained
    by multiplying by a cut-off function as above. 
    Set $\widetilde\Phi_0=\Phi_0-\widetilde\ell$. 
By \eqref{eqn:l-estimate} and \eqref{eqn:phi0}, it suffices to prove a
bound on $\widetilde\Phi_0$. 
    We first claim that 
    \begin{equation} \label{eqn:tilde_estimate} 
        \Vert\widetilde\Phi_0\Vert_{L^2_{1,\delta}}\leq C \Vert \Phi\Vert_{\Dscr_{\delta}}
    \end{equation} 
       for a constant $C$ which depends on the choice of cut-off function
       but is otherwise independent of $\Phi$.
    Indeed, since  $\widetilde\Phi_0$
    vanishes as $\tau\to \infty$, we can integrate by parts to obtain
    $$
        \Vert\partial_{A_0}\widetilde\Phi_0\Vert^2_{L^2_\delta}=\Vert\dbar_{A_0}\widetilde\Phi_0\Vert^2_{L^2_\delta}
        +\langle\widetilde\Phi_0, i\Lambda F_{A_0}\widetilde\Phi_0\rangle_{L^2_\delta} 
        +\langle\widetilde\Phi_0,
        i\delta\Lambda(d\tau\wedge d_{A_0}\widetilde\Phi_0)\rangle_{L^2_\delta} 
    $$
   It follows that 
    \begin{equation} \label{eqn:l21}
        \Vert\nabla_0 \widetilde\Phi_0\Vert^2_{L^2_\delta}\leq
        C(\Vert\dbar_{A_0}\widetilde \Phi_0\Vert^2_{L^2_\delta}+
        \Vert\widetilde\Phi_0\Vert^2_{L^2_\delta})
        \leq C( \Vert\Omega\Vert^2_{L^2_\delta}+
        \Vert\Phi_0\Vert^2_{L^2_{-\delta}})
        \leq C\Vert\Phi\Vert^2_{\Dscr_\delta}
    \end{equation} 
    where we have used \eqref{eqn:harmonic-projection}, \eqref{eqn:phi0}, and 
Lemma \ref{lem:delta-minus-delta} (2). 
The claim \eqref{eqn:tilde_estimate} now follows
from \eqref{eqn:l21} and \eqref{eqn:poincare}. 
Next, a bound on $\sup_{X_0}|\widetilde\Phi_0|$  follows from elliptic
regularity and the fact that
$\dbar_{A_0}^{\ast_\delta}\dbar_{A_0}\Phi_0=0$.
Finally, on each $C(p)$,  
$\dbar_{A_0}^{\ast_\delta}\dbar_{A_0}\widetilde\Phi_0=0$,
and so we have
$$
\Vert \dbar^\ast_{A_0}\dbar_{A_0}\widetilde\Phi_0\Vert_{L^2_\delta}\leq
C\Vert \widetilde\Phi_0\Vert_{L^2_{1,\delta}}
$$
Since the operator $\dbar^\ast_{A_0}\dbar_{A_0}$ is translation invariant
on $C(p)$,
the usual a priori estimate (cf.\ \cite[eq.\ (2.4)]{LockhartMcOwen:85})
along with \eqref{eqn:tilde_estimate} imply an estimate
$$
\Vert\widetilde\Phi_0\Vert_{L^2_{2,\delta}}\leq C 
\Vert\Phi\Vert^2_{\Dscr_\delta}\ ,
$$
and the $L^\infty$ bound follows as in \eqref{eqn:sup_eta}.  This completes the
proof of the existence of a decomposition.

If we have two such expressions:
$$
\Phi=\Phi_0+\Phi_1=\Phi_0'+\Phi_1'
$$
then $\Phi_0-\Phi_0'$ is bounded, continuous, and in $L^2_{\delta}$, and
therefore
$$
\lim_{\tau\to +\infty}
\Phi_0-\Phi_0' =0\ .
$$
Hence, the limit $\ell$ is independent of the decomposition.
\end{proof}

\begin{Remark} \label{rem:general-A}
    Notice that for any $A\in \Acal_{\delta}$, $\Phi\in \Dscr_{\delta}$, 
we have
$\dbar_{A}\Phi\in L^2_{\delta}$.
Indeed, $\dbar_A-\dbar_{A_0}=\beta\in L^2_{1,\delta}$,
and with respect to the decomposition in  Proposition \ref{prop:higgs-decomposition}, 
$$
[\beta,\Phi]=[\beta,\Phi_0]+[\beta,\Phi_1]
$$
Since $\Phi_0$ is bounded, the first term on the right hand side above is in $L^2_\delta$. Since
$\Phi_1\in L^2_{1,\delta}$,  the second term is in $L^2_\delta$ via the
inclusion $L^2_{1,\delta}\hookrightarrow L^4_\delta$. 
    Hence, $\dbar_A\Phi\in L^2_{\delta}\iff\dbar_{A_0}\Phi\in
    L^2_{\delta}$.

\end{Remark}

\begin{Definition}[{\sc Higgs pairs}] \label{def:higgs-bundle}
    A parabolic Higgs pair is a couple
    $(A,\Phi)\in \Acal_{\delta}\times \Dscr_{\delta}(\slfrak_E\otimes K_X)$ satisfying: $\dbar_A\Phi=0$.
We say that $(A,\Phi)$ is strongly parabolic if $\Phi\in L^2_{1,\delta}$.
    We let 
    $$\Bcal^{par}_{\delta}\subset \Acal_{\delta}\times
    \Dscr_{\delta}(\slfrak_E\otimes K_X)\quad ,\quad
    \Bcal^{spar}_{\delta}\subset \Acal_{\delta}\times
    L^2_{1,\delta}(\slfrak_E\otimes K_X)$$
    denote the spaces of parabolic and strongly
    parabolic Higgs pairs, respectively. 
\end{Definition}

The following key example gives the link Higgs pairs and parabolic Higgs
bundles.
\begin{Example} \label{ex:phi} 
    Let $(A,\Phi)\in \Bcal^{par}_{\delta}$, and 
 assume that 
$A=A_0$ on $C(p)$, $p\in D$. 
    In  the local holomorphic frame $\{s_i\}$ from
   \S\ref{sec:gauge-groups}, write 
    $$\Phi s_i=\sum_{j=1}^n\Phi_{ij} s_j =\sum_{j=1}^nP_{ij}(z) s_j\otimes \frac{dz}{z}$$
     Then the  $P_{ij}(z)$ are
    holomorphic functions on $\Delta^\times(p)$. In terms of the unitary frame,
$$
\Phi e_i= \sum_{j=1}^n\widehat \Phi_{ij}e_j=\sum_{j=1}^n
\Phi_{ij}|z|^{\alpha_j(p)-\alpha_i(p)}e_j
    =\sum_{j=1}^nP_{ij}(z)|z|^{\alpha_j(p)-\alpha_i(p)}e_j\otimes \frac{dz}{z}
$$  
Now $\widehat \Phi_{ij}\in L^2_{-\delta}$. Recalling that
$e^{-\tau\delta}=|z|^{\delta}$, and that $dz/z$ has norm $1$, we have 
$$
    \int_{\Delta^\times (p)}\frac{|dz|^2}{|z|^2}\, |P_{ij}(z)|^2
|z|^{2(\alpha_j(p)-\alpha_i(p))+\delta}<+\infty 
$$
By the condition \eqref{eqn:delta-assumption},
    the above bound means that the $P_{ij}(z)$ are regular on $\Delta(p)$.
    Moreover,
     $P_{ij}(0)= 0$ unless $2(\alpha_j(p)-\alpha_i(p))+\delta>0$, which
     occurs only if $\alpha_j(p)\geq \alpha_i(p)$, again by
     \eqref{eqn:delta-assumption}.
In other words, $\Phi$ can have only simple poles, and $\res_p\Phi$ is
    upper triangular with respect to the canonical frame. 
\end{Example}

\begin{Proposition} \label{prop:strong} 
Assume full flags. Then the following  maps are continuous: 
    \begin{align}
        \Dscr_\delta(\slfrak_E\otimes K_X)
        \times\Dscr_\delta(\slfrak_E\otimes K_X)
        &\lra L^2_{\delta}(\slfrak_E\otimes K_X\otimes \overline K_X)
        : (\varphi,\psi)\mapsto [\varphi^\ast,\psi]
        \label{eqn:basic_continuity}
        \\
        L^2_{2,\delta}(\slfrak_E)
        \times\Dscr_\delta(\slfrak_E\otimes K_X)
        &\lra L^2_{1,\delta}(\glfrak_E\otimes K_X)
        : (\eta,\varphi)\mapsto \eta\varphi
        \label{eqn:eta_continuity} \\
        L^2_{2,\delta}(\slfrak_E)
        \times\Dscr_\delta(\slfrak_E\otimes K_X)
        &\lra L^2_{\delta}(\slfrak_E\otimes K_X\otimes \overline K_X)
        : (\eta,\varphi)\mapsto [\varphi, [\varphi^\ast,\eta]]
        \label{eqn:phi_continuity}
    \end{align}
\end{Proposition}

\begin{proof}
We shall prove that under the assumptions there is a constant $C$,
    independent of $\eta,\varphi,\psi$, such that
    \begin{align}
        \Vert [\varphi^\ast,\psi]\Vert_{L^2_\delta}&\leq C
        \Vert\varphi\Vert_{\Dscr_\delta}
        \Vert\psi\Vert_{\Dscr_\delta} \ ; \label{eqn:bracket_estimate} \\
        \Vert\eta \varphi\Vert_{L^2_{1,\delta}}&\leq C
        \Vert \eta\Vert_{L^2_{2,\delta}}
        \Vert\varphi\Vert_{\Dscr_\delta}
         \ . \label{eqn:eta_estimate} \\ 
        \Vert [\varphi, [\varphi^\ast,\eta]]  \Vert_{L^2_{\delta}}&\leq C
        \Vert \eta\Vert_{L^2_{2,\delta}}
        \Vert\varphi\Vert_{\Dscr_\delta}^2
         \ . \label{eqn:product_estimate} 
    \end{align}
    For \eqref{eqn:bracket_estimate}, 
     use the decomposition from Proposition \ref{prop:higgs-decomposition} to
     write: $\varphi=\varphi_0+\varphi_1$, $\psi=\psi_0+\psi_1$.
     Consider
     the terms on the right hand side in the 
     expansion
     $$
    [\varphi^\ast, \psi]-[\varphi_0^\ast, \psi_0]=
    [\varphi^\ast_0,\psi_1]+[\varphi_1^\ast,\psi_0]+[\varphi_1^\ast,\psi_1]
    \ .
     $$
Because $\varphi_0$ and $\psi_0$ are bounded, 
     and $\varphi_1$ and $\psi_1$ are in $L^2_{1,\delta}$, 
     the first two terms 
     are clearly in $L^2_\delta$, 
     and the last term is in 
$L^2_\delta$
      because of
     the embedding $L^2_{1,\delta}\hookrightarrow L^4_{\delta}$.
     The  term $[\varphi_0^\ast, \psi_0]$ is not, in general, in $L^2_\delta$.
     Let $m_0$ (resp.\ $\ell_0$) denote the limits of $\varphi_0$ (resp.\
     $\psi_0$) from Proposition \ref{prop:higgs-decomposition}, and 
    let $\widetilde m_0:=\phi_1\cdot m_0$ (resp.\
    $\widetilde\ell:=\phi_1\cdot \ell_0$) be smooth extensions 
     to $X^\times$ (see \eqref{eqn:phiR}). Then  in particular, $\widetilde
m_0$, $\widetilde\ell$
     and their derivatives have pointwise norms bounded by the norms of $m_0$ and $\ell_0$. 
     We further expand:
    \begin{equation} \label{eqn:expansion} 
    [\varphi_0^\ast,\psi_0]=[\varphi_0^\ast-\widetilde m_0^\ast,
    \psi_0-\widetilde\ell_0]+[\varphi_0^\ast-\widetilde
    m_0^\ast,\widetilde\ell_0]+[\widetilde m_0^\ast,\psi_0-\widetilde\ell_0]
    +[\widetilde m_0^\ast, \widetilde\ell_0]
    \end{equation} 
      Since we assume full flags, $\lfrak_p$ is abelian for every
     $p\in D$, 
     so  the last term is compactly supported. Moreover, by
     our choice of extension $\widetilde\ell_0$
     and the second part of Proposition \ref{prop:higgs-decomposition}, it is bounded by
     the graph norm of $\psi$. 
    The other terms on the right hand side of \eqref{eqn:expansion} are
    bounded by the same graph norm by using Lemma
    \ref{lem:delta-minus-delta}: the first term by \eqref{eqn:phi1}
    and the multiplication theorem; and the second  and third terms by our
    choice of extensions and \eqref{eqn:phi0}.
    This proves the estimate \eqref{eqn:bracket_estimate}.

    For \eqref{eqn:eta_estimate}, again using the embedding
$L^2_{1,\delta}\hookrightarrow L^4_\delta$, along with  
Proposition \ref{prop:higgs-decomposition}, we have
$$
\Vert \eta\varphi\Vert_{L^2_\delta}\leq \Vert
\eta\Vert_{L^2_\delta}\sup|\varphi_0|+
C\Vert \eta\Vert_{L^2_{1,\delta}}
\Vert \varphi_1\Vert_{L^2_{1,\delta}}
\leq C\Vert \eta\Vert_{L^2_{2,\delta}}
        \Vert\varphi\Vert_{\Dscr_\delta}
        \ .
$$
    On the
    other hand, 
   \begin{align*} 
       \Vert \dbar_{A_0}(\eta\varphi)\Vert_{L^2_\delta}&\leq
\Vert \dbar_{A_0}(\eta)\varphi_0\Vert_{L^2_\delta}
+ \Vert \dbar_{A_0}(\eta)\varphi_1\Vert_{L^2_\delta}
+ \Vert \eta\dbar_{A_0}\varphi\Vert_{L^2_\delta} \\
       &\leq 
       \Vert\eta\Vert_{L^2_{1,\delta}}\sup|\varphi_0|
       +C\Vert\eta\Vert_{L^2_{2,\delta}}\Vert\varphi_1\Vert_{L^2_{1,\delta}}
       +\sup|\eta|\Vert\varphi\Vert_{\Dscr_\delta} \\
       &\leq C\Vert\eta\Vert_{L^2_{2,\delta}}
        \Vert\varphi\Vert_{\Dscr_\delta}
   \end{align*} 
The argument for \eqref{eqn:product_estimate} is similar.
This completes the proof of the Proposition.
\end{proof}

\begin{Corollary} \label{cor:gauge_group_action}
Defining
$$
 \Gcal_{\delta}\times   \Acal_{\delta}\times\Dscr_{\delta} \lra \Acal_{\delta}\times\Dscr_{\delta}
    : (g,(A,\Phi))\mapsto (g(A), g\Phi g^{-1})
$$
    gives a smooth (left) action of the gauge group $\Gcal_{\delta}$ on
    $\Acal_{\delta}\times\Dscr_{\delta}$ which preserves the space
    $\Bcal^{par}_{\delta}$. 
\end{Corollary}

By Lemma \ref{lem:smooth}, after a gauge transformation in
$\Gcal_{\delta,\ast}$,
we may always
assume a Higgs pair $(A,\Phi)$ is smooth and of the form $(A_0,\Phi)$ on each
$C(p)$.
Therefore, Example \ref{ex:phi} shows  there is a well defined map:
\begin{equation} \label{eqn:residue}
    \mathrm{Res}: \Bcal^{par}_{\delta}\lra \bigoplus_{p\in D} \lfrak_p
\end{equation}
taking $\Phi$ to the collection $\{P_{ij}(0)\}_{\alpha_i(p)=\alpha_j(p)}$.
Note that  this does not depend on the representative in the
$\Gcal_{\delta,\ast}$-orbit of $(A,\Phi)$.
It is also clear that
$\Bcal^{spar}_{\delta}=\mathrm{Res}^{-1}(0)$.

Let us also note the following. Recall the definition of $\rbold$ from
Proposition \ref{prop:higgs-decomposition}.

\begin{Proposition}
For $(A,\Phi)\in \Bcal^{par}_{\delta}$,
$\mathrm{Res}(A,\Phi)=\rbold(\Phi)$.
\end{Proposition}

\begin{proof}
 First, note that $\rbold(g\Phi g^{-1})=\rbold(\Phi)$ for $g\in
\Gcal_{\delta,\ast}$.  Now if $A=A_0$ on each $C(p)$, then we may take
$$
\Phi_0=\sum_{j=1}^n P_{ij}(0)|z|^{\alpha_j(p)-\alpha_i(p)}
e_j\otimes\frac{dz}{z}
$$
near each $p\in D$. Then $\Phi_0$ has the same limit as the one defined in
\ref{ex:phi}, and by using the expression for $\Phi$ in  the example one checks
that $\Phi-\Phi_0\in L^2_{1,\delta}$. Conversely, since both $\rbold$ and
$\res_{\lfrak}$ are gauge invariant, we may always reduce to this case.
\end{proof}

\begin{Remark} \label{rem:tensor}
    If $A_0^{(1)}$ and $A_0^{(2)}$ are model connections on bundles
    $E^{(1)}$ and $E^{(2)}$ as in \S\ref{sec:gauge-groups}, then
    clearly $E^{(1)}\oplus E^{(2)}$ inherits a model connection. In
    particular, the results of this section apply, \emph{mutatis mutandi}, to Higgs fields
    $\Phi\in L^2_{1,\delta}(\Hom(E^{(1)}, E^{(2)})\otimes K_X)$. 
\end{Remark}

\subsubsection{Relation with parabolic Higgs bundles and nonabelian Hodge}
We now make the connection between  Higgs pairs and Higgs bundles.
Recall the configuration spaces  $\Bcal_{\delta}^{par}$ and
$\Bcal_{\delta}^{spar}$ defined in  
Definition \ref{def:higgs-bundle}. 

\begin{Proposition}[{\sc Higgs pairs and Higgs bundles}]
\label{prop:analytic-higgs}
    Associated to each $\Gcal_{\delta,\ast}$ orbit in
    $\Bcal_{\delta}^{par}$ (resp.\ 
    $\Bcal_{\delta}^{spar}$) there is a unique isomorphism class of  framed
    parabolic (resp.\ strongly parabolic) Higgs bundles. Moreover, the
    residue maps defined in eqs.\ \eqref{eq big Res on moduli} and
    \eqref{eqn:residue} agree.
\end{Proposition}

\begin{proof}
    By Lemma \ref{lem:smooth}, a $\Gcal_{\delta,\ast}$-orbit in
    $\Bcal_{\delta}^{par}$ 
    contains a  pair $(A,\Phi)$, where $\dbar_A$ is a smooth  $\dbar$-operator $\dbar_A$ that is equal to
    $\dbar_{A_0}$ on each $C(p)$. This defines a holomorphic bundle
    $\Ecal^0$ on $X^\times$ with a preferred holomorphic frame $\{s_i\}$ on
    each $C(p)$. Gluing to the trivial bundle on $\Delta(p)$ uniquely
    determines an
    extension of $\Ecal^0$ to a holomorphic bundle $\Ecal$ on $X$.  
    For each $0\leq \alpha<1$, define $\Ecal_\alpha$ to be the sheaf of
    germs generated by $\{s_i\}$ for $\alpha\leq \alpha_i(p)$. This defines
    a parabolic structure on $\Ecal$ with jumps $\alpha(p)$ at $p$. 
    According to the discussion in Example \ref{ex:phi}, 
     $\Phi\in H^0(\End(\Ecal(\alpha))\otimes
    K(D))$, with given residues. Hence, $(\Ecal,\Phi)$ is a parabolic
    Higgs bundle in the sense of Definition \ref{def:parabolic-higgs}. Given two such pairs
    $(A_1,\Phi_1)$  and $(A_2,\Phi_2)$  
    in the same $\Gcal_{\delta,\ast}$ orbit, then the element $g$ such that
$g(A_1,\Phi_1)=(A_2,\Phi_2)$ 
    extends as a holomorphic isomorphism of parabolic Higgs bundles
    $(\Ecal_1(\alpha),\Phi_1)\simeq (\Ecal_2(\alpha),\Phi_2)$. 
\end{proof}

In view of Proposition \ref{prop:analytic-higgs},
we have the following.
\begin{Definition}
    We call a  Higgs pair  $(A,\Phi)\in\Bcal_{\delta}^{par}$ stable (resp.\
    semistable)
    according to the stability (resp.\ semistability) of any of  the isomorphic  parabolic Higgs
bundles in its $\Gcal_{\delta,\ast}$-orbit.
Let $\Bcal_{\delta}^{par,s}$ (resp.\ $\Bcal_{\delta}^{spar,s}$) denote the
open subset of stable parabolic (resp.\ stable strongly parabolic) Higgs
bundles. 
\end{Definition}

At this point, we have made no assumption on the weights $\alpha$, other
than the full flags assumption for the general parabolic case. Consider
now the parabolic  bundle $\Ecal(\alpha)$  determined by extending
$\dbar_{A_0}$ to a bundle $\Ecal$ with parabolic structure,  as in
Proposition \ref{prop:analytic-higgs} (the Higgs field is not relevant here).  We now suppose 
the weights $\alpha$ and the background
connection $A_0$ have been chosen so that $\Ecal(\alpha)$ has parabolic
degree zero, and that  in fact $\Lambda^n(\Ecal(\alpha))$ is
holomorphically trivial.  
Hence,  in particular, $\Vert\alpha(p)\Vert\in \NBbb$ for all $p\in D$.

Let $\onebold_{D,\Vert\alpha\Vert}$ denote a section 
 of 
$$ \Lcal_{D,\Vert\alpha\Vert}:=
\bigotimes_{p\in D}
\Ocal(\Vert\alpha(p)\Vert p)$$
defined by the divisor. This gives a trivialization, unique up to a nonzero
multiplicative constant, of $\Lcal_{D,\Vert\alpha\Vert}$ on $X^\times $.  
Use this to define a singular hermitian metric on $\Lcal_{D,\Vert\alpha\Vert}$
by declaring $\Vert\onebold_{D,\Vert\alpha\Vert}\Vert=1$.
Then the product of this with  the background metric on $E\to X^\times$
combine to give a smooth metric on $\Lambda^n\Ecal(\alpha)$.
Since we assume this is trivial, there is a  modification of the product metric,  unique up to a nonzero
multiplicative constant,  which makes $\Lambda^n\Ecal(\alpha)$ flat. 
We can then choose a trivialization of $\Lambda^n\Ecal(\alpha)$ by taking a
global nonzero flat section.

Since all the bundles in 
$\Bcal_{\delta}^{par}$ (resp.\ 
    $\Bcal_{\delta}^{spar}$) 
    induce the same structure on $\Lambda^n\Ecal(\alpha)$, the
    trivialization may be fixed once and for all. We record this discussion in the
    following.

\begin{Proposition} \label{prop:analytic-sl-higgs}
Assume the weights and  background
connection $A_0$ have been chosen so that 
 $\Lambda^n(\Ecal(\alpha))$ is trivial. 
 Then 
    associated to each $\Gcal_{\delta,\ast}$ orbit in
    $\Bcal_{\delta}^{par}$ (resp.\ 
    $\Bcal_{\delta}^{spar}$) there is a unique isomorphism class of  framed
    parabolic (resp.\ strongly parabolic) $\SL(n,\CBbb)$-Higgs bundles. 
\end{Proposition}

Finally, we end this subsection by  quoting  the important result on the existence of
Hermitian-Einstein metrics on parabolic Higgs pairs, proved by Simpson,
Biquard, and Mochizuki. 
For $A\in \Acal_{\delta}$, write $A=A_0+ a$, with $a\in
L^2_{1,\delta}(\sufrak_E\otimes T^\ast X^\times)$. Let
$$F_A=F_{A_0}+d_{A_0}a+a\wedge a$$ denote the curvature. Then since
$F_{A_0}$ vanishes in a neighborhood of $D$, it follows that $F_A\in
L^2_\delta$. 
We denote by $$F_A^\perp := F_A-\frac{1}{n}\tr(F_A)\cdot \Ibold$$
 the traceless part of $F_A$.
Given $\Phi\in \Dscr_\delta(\slfrak_E\otimes K_X)$,  and assuming full flags,  we also have 
from \eqref{eqn:bracket_estimate} that $[\Phi,\Phi^\ast]\in L^2_\delta$.
The \emph{Hitchin equations} for a parabolic Higgs pair $(A,\Phi)\in \Bcal_{\delta}^{par}$
are:
\begin{equation} \label{eqn:hitchin}
    F_A^\perp + [\Phi,\Phi^\ast]=0\ .
\end{equation}

We have the following parabolic version of the Hitchin-Simpson theorem.
\begin{Theorem}[{\sc Nonabelian Hodge Theorem}] \label{thm:hk}
    Let $(A,\Phi)\in \Bcal_{\delta}^{par,s}$. Then there is $g\in
    \Gcal_{\delta}$ such that the pair $g(A,\Phi)$ satisfies
    \eqref{eqn:hitchin}. Moreover, $g$ is unique up to the action of
    $\Kcal_{\delta}$.
\end{Theorem}

\begin{proof}[Sketch of proof]
    In the context of weighted Sobolev spaces, this version of the theorem
    is due to Biquard \cite[Thm.\ 8.1]{Biquard:97}. 
Let us note the following. 
    First,  by Proposition  \ref{prop:boundary-map} we have that 
    $\Rcal_{\delta}\simeq \widehat L^2_{2,\delta}$, where the latter
    space
    is defined in \cite[p.\ 53]{Biquard:97}. Choose any $\widetilde
    \delta>\delta$  still satisfying the assumption \eqref{eqn:delta-assumption}. On
    \cite[p.\ 77]{Biquard:97} it is shown that there is a solution to
    \eqref{eqn:hitchin} after acting by a gauge transformation in $\widehat
    L^p_{2,\widetilde\delta}$ for any $p>2$. Now by the embedding 
    $ L^p_{2,\widetilde\delta}\subset L^2_{2,\delta}$ for $p$ sufficiently
    close to $2$
    (see \cite[Lemme 4.5]{Biquard:97}) we conclude that there is in fact a
    solution in the orbit of $\Gcal_{\delta}$.
    The uniqueness statement is standard.
\end{proof}

As in the case of gauge theory on closed manifolds,
we may alternatively regard \eqref{eqn:hitchin} as an
equation for a metric. Recall that $h_0$ denotes the hermitian metric on $E\to
X^\times$. Then for each $g\in \Gcal_{\delta}$, we define a new metric
$g(h_0)$ by the rule: $\langle u,v\rangle_{g(h_0)}:=\langle gu,
gv\rangle_{h_0}$.  
We define the space of admissible metrics:
\begin{equation} \label{eqn:metrics}
    \Mcal_\delta :=\left\{ g(h_0) \mid g\in \Gcal_{\delta} \right\}
    \simeq \Gcal_\delta/\Kcal_\delta
\end{equation}
Then we have the following reformulation of Theorem \ref{thm:hk}.

\begin{Theorem}[{\sc Harmonic metric}] \label{thm:hk-metric}
    Let $(A,\Phi)\in \Bcal_{\delta}^{par,s}$
    and $(\Ecal(\alpha),\Phi)$ a stable parabolic Higgs bundle  associated  to
    $(A,\Phi)$ from Proposition \ref{prop:analytic-higgs}.  
    Then there exists a unique metric $h\in \Mcal_\delta$
    satisfying
$$
    F_{(\Ecal,h)}^\perp + [\Phi,\Phi^{\ast_h}]=0\ ,
    $$
    where $(\Ecal,h)$ denotes the Chern connection for $\Ecal$ with the
    metric $h$. 
\end{Theorem}
The metric in the theorem above is called the \emph{harmonic metric} for $(\Ecal(\alpha),\Phi)$.

\subsection{Analytic Dolbeault moduli spaces}  \label{sec:dolbeault}
Fix $\alpha$, and suppose $\delta$ satisfies
    \eqref{eqn:delta-assumption}. 
Set
$$\Mbold^{par,s}_{\Dol}(\alpha,\delta)=\Gcal_{\delta}\backslash\Bcal_{\delta}^{par,s}\ ,\ \Mbold^{spar,s}_{\Dol}(\alpha,\delta)=
 \Gcal_{\delta}\backslash\Bcal_{\delta}^{spar,s}$$
The main goal of this section is to prove the following:

\begin{Theorem}[{\sc Moduli spaces of Higgs bundles}]
\label{thm:analytic-moduli}
    \begin{enumerate}
        \item
            If nonempty, 
            the moduli space $\Mbold^{spar,s}_{\Dol}(\alpha,\delta)$ is a smooth complex manifold of
            complex dimension 
            $$\dim\Mbold^{spar,s}_{\Dol}(\alpha,\delta)=(2g-2)\dim\sG+2\sum_{p\in
            D}\dim(\sG/\sP_p) \ .$$
\item 
    Assume full flags. If nonempty,  
            the moduli space $\Mbold^{par,s}_{\Dol}(\alpha,\delta)$ is a smooth complex manifold of
complex dimension 
            $$\dim\Mbold^{par,s}_{\Dol}(\alpha,\delta)=(2g-2+d)\dim\sG=
            \dim\Mbold^{spar,s}_{\Dol}(\alpha,\delta)+\sum_{p\in D}\dim(\sL_p)\ .$$
            \item    
   The assignment of a parabolic Higgs bundle to $(A,\Phi)$
(Proposition \ref{prop:analytic-higgs})    induces biholomorphisms
   $$ 
        \Mbold^{par,s}_{\Dol}(\alpha,\delta) \isorightarrow \Pcal_0^s(\alpha) \quad ,\quad
        \Mbold^{spar,s}_{\Dol}(\alpha,\delta) \isorightarrow
        \Scal\Pcal_0^s(\alpha) \ .
   $$ 
    \end{enumerate}
\end{Theorem}

We note that previous constructions exist in the strongly parabolic case
(e.g.\ \cite{Konno:93}) and for rank $2$ parabolic bundles \cite{Nakajima:96}.
Below we consider the general case of parabolic bundles under Assumption A of
the introduction.
We will actually first construct the \emph{framed} moduli spaces 

$$\Mbold^{par,s}_{\Dol,\ast}(\alpha,\delta)=\Gcal_{\delta,\ast}\backslash\Bcal_{
\ delta}^{par,s}\quad
 ,\quad
\Mbold^{spar,s}_{\Dol,\ast}(\alpha,\delta)=\Gcal_{\delta,\ast}\backslash\Bcal_{\
delta}^{spar,s}$$
as
finite dimensional complex manifolds. It can be shown (cf.\
\cite[Prop.\ 7.1.14]{DiffGeomCompVectBun}) that the  residual gauge
group $\overline\Lbold$  acts freely and  properly  on 
$\Mbold^{par,s}_{\Dol,\ast}(\alpha,\delta)$ and
$\Mbold^{spar,s}_{\Dol,\ast}(\alpha,\delta)$. 
The 
moduli spaces are then the quotients
 $$\Mbold^{par,s}_{\Dol}(\alpha,\delta)=\overline\Lbold\backslash\Mbold^{par,s}_{\Dol,\ast}(\alpha,\delta)\quad
 ,\quad
\Mbold^{spar,s}_{\Dol}(\alpha,\delta)=\overline\Lbold\backslash\Mbold^{spar,s}_{\Dol,\ast}(\alpha,\delta)$$
and inherit a complex manifold structure \cite{Cartan:54}.

\subsubsection{Index computations}
The natural operators  in the deformation complexes of Higgs bundles are
$D''=\dbar_A+\Phi$ and $D'=\partial_A+\Phi^\ast$. By the K\"ahler
identities, $D'$ is the formal $L^2$-adjoint of $D''$ (see \cite[\S 1]{Simpson:92}).
In the context of strongly parabolic Higgs bundles in the  weighted Sobolev
spaces we are using, it is more natural to consider
the $L^2_\delta$-adjoint $D'_\delta:=e^{-\tau\delta} D' e^{\tau\delta}$. 
The adjoint for parabolic Higgs bundles (not necessarily strongly parabolic)
would be more complicated. However, in light of Proposition \ref{prop:strong}, we
may instead use the same operators as in the strongly parabolic case.
The goal of this section is to prove  index formulas for the Dirac type
operators $D''+D'_\delta$. 

\begin{Proposition} \label{prop:index}
    Let $(A,\Phi)\in \Bcal_{\delta}^{spar}$ be a smooth parabolic Higgs bundle.

    \noindent
        (1)
    Let
    $$T^{spar}_{(A,\Phi)}: 
L^2_{1,\delta}(\slfrak_E\otimes\overline K_X)\oplus
L^2_{1,\delta}(\slfrak_E\otimes K_X)
\lra L^2_\delta(\slfrak_E)\oplus
L^2_{\delta}(\slfrak_E\otimes K_X\otimes \overline K_X) 
$$
be the  operator defined by 
    $$T^{spar}_{(A,\Phi)}(\beta,\varphi)
    =(e^{-\tau\delta}\dbar_A^\ast(e^{\tau\delta}\beta)
-i\Lambda[\Phi^\ast, \varphi] , \dbar_A\varphi+[\Phi,\beta])
$$
    Then $T^{spar}_{(A,\Phi)}$ is Fredholm of
index
$$
    \ind(T^{spar}_{(A,\Phi)})=(2g-2+d)\dim\sG
$$

\noindent
    (2)  Consider the same operator  as above, but with different
    domain:  
    $$T^{par}_{(A,\Phi)}: 
L^2_{1,\delta}(\slfrak_E\otimes\overline K_X)\oplus
\Dscr_\delta(\slfrak_E\otimes K_X)
\lra L^2_\delta(\slfrak_E)\oplus
L^2_{\delta}(\slfrak_E\otimes K_X\otimes \overline K_X) 
$$
    Then $T^{par}_{(A,\Phi)}$ is Fredholm of index
$$
    \ind(T^{par}_{(A,\Phi)})=(2g-2+d)\dim\sG+\sum_{p\in D} \dim \sL_p
$$
\end{Proposition}

\begin{proof}
We first note that  the operators $T^{spar}_{(A,\Phi)}$ and
    $T^{spar}_{(A,\Phi)}$ are well-defined by Proposition
    \ref{prop:strong}.
    Also, as in the proof of Lemma \ref{lem:smooth}, 
    to prove Fredholmness and compute the index, we may drop the terms
    involving 
      $\Phi$ and $\Phi^\ast$
and  replace $A$ with $A_0$. 
So it suffices to consider the decoupled operators
\begin{align*}
    \dbar_{A_0} : L^2_{1,\delta}(\slfrak_E)&\lra L^2_{\delta}(\slfrak_E\overline K_X) : \eta\mapsto \dbar_{A_0}\eta\notag \\
    T_\beta: L^2_{1,\delta}(\slfrak_E\otimes\overline K_X)&\lra
L^2_\delta(\slfrak_E) :
 \beta
    \mapsto e^{-\tau\delta}\dbar_{A_0}^\ast(e^{\tau\delta}\beta)\notag\\
T_\varphi : 
    L^2_{1,\delta}(\slfrak_E\otimes K_X)&\lra 
L^2_{\delta}(\slfrak_E\otimes K_X\otimes \overline K_X) :
\varphi\mapsto \dbar_{A_0}\varphi \notag\\
\end{align*}
We have:
    \begin{equation}\label{eqn:T-index}
    \ind(T^{spar}_{(A,\Phi)})= \ind(T_\beta)+\ind(T_\varphi)
            \end{equation}
    The operators $\dbar_{A_0}$ and $T_\varphi$ are Fredholm (\cite[Thm.\ 1.1]{LockhartMcOwen:85}). 
The systems are equivalent to the $L^2$-extended operator of
    Atiyah-Patodi-Singer
    (see \cite[p.\ 56]{Taubes:93} and \cite[Prop.\
    3.7]{DaskalWentworth:97}).
    It then  follows from \cite[Thm.\ 3.10]{APS:75} that the indices are 
    \begin{align}  
        \ind(\dbar_{A_0})&= -(g-1+d/2)\dim\sG-\frac{1}{2}\sum_{p\in D} \dim
        \sL_p \notag \\
        \ind(T_\varphi)&= (g-1+d/2)\dim\sG-\frac{1}{2}\sum_{p\in D} \dim \sL_p   \label{eqn:index-phi}
\end{align}
        Here, we have used the fact from Remark \ref{rem:spectrum} that the boundary operator has a symmetric spectrum, and
    therefore the $\eta$-function vanishes identically (see \cite[eq.\
    (1.7)]{APS:75}).  The adjoint has index
\begin{equation} \label{eqn:index-beta}
    \ind(T_\beta)=-\ind(\dbar_{A_0})=
    (g-1+d/2)\dim\sG+\frac{1}{2}\sum_{p\in D} \dim \sL_p
\end{equation}
    The proof of part (1) then follows from 
    eqs.\ \eqref{eqn:T-index}, \eqref{eqn:index-phi}, and \eqref{eqn:index-beta}.
    For part (2), let $\widehat T_\varphi$ denote the same operator $T_\varphi$ but
    with domain  $\Dscr_\delta$.
The proof of part (2) then follows immediately from 
    the following claim:
    \medskip

\par\noindent{\bf Claim.} 
    $\ind(\widehat T_\varphi)=\ind(T_\varphi)+\sum_{p\in D} \dim \sL_p$.
    \medskip

\noindent The Claim follows from \cite[eq.\ (3.25)]{APS:75}.
    Alternatively, consider the same  operator with different domains and
    ranges:
    $$
\widetilde T_\varphi : 
    L^2_{1,-\delta}(\slfrak_E\otimes K_X)\lra 
L^2_{-\delta}(\slfrak_E\otimes K_X\otimes \overline K_X) :
    $$
Clearly, $\ker \widetilde T_\varphi=\ker \widehat T_\varphi$. 
For the cokernel, 
let $\beta\in L^2_{\delta}(\slfrak_E\otimes K_X\otimes \overline K_X)$
    satisfy $\dbar_{A_0}^\ast(e^{\tau\delta}\beta)=0$. Then there is a
    well-defined limit 
    $$
    \ell=\lim_{\tau\to\infty} e^{\tau\delta}\beta\in \bigoplus_{p\in
    D}\lfrak_p
    $$
    Now $\widetilde\ell:=\phi_R\cdot\ell\in \Dscr_{\delta}(\slfrak_E\otimes
    K_X)$, and 
    $
    \langle \dbar_{A_0}\ell, \beta\rangle_{L^2_\delta}=2\pi |\ell|^2
    $.
    Therefore,  for $\beta\in \coker \widehat T_\varphi$, $\ell=0$, and so
    $e^{\tau\delta}\beta\in L^2_\delta$.
    Let $\Omega=e^{2\tau\delta}\beta$. Then
    $\dbar_{A_0}^\ast(e^{-\tau\delta}\Omega)=0$, $\Omega\in L^2_{-\delta}$,
    and indeed $\Omega\in \coker \widetilde T_\varphi$. Conversely, given
    $\Omega\in \coker \widetilde T_\varphi$, we have 
    $$\lim_{\tau\to\infty} e^{-\tau\delta}\Omega=0$$
    and so 
    we may define $\beta=e^{-2\tau\delta}\Omega\in \coker\widehat
    T_\varphi$. Hence, 
    $\coker \widetilde T_\varphi\simeq \coker \widehat T_\varphi$, and 
    $\ind(\widetilde T_\varphi)= \ind(\widehat T_\varphi)$. Finally, by
    \cite[Thm.\ 1.2]{LockhartMcOwen:85},
    $$\ind(\widetilde T_\varphi)=\ind(T_\varphi)+\sum_{p\in D} \dim \sL_p$$
    This proves the Claim and completes the proof of the Proposition.
\end{proof}

We shall also need the following.
\begin{Proposition} \label{prop:nude-index}  Recall the
definition of $\mathcal{R}_\delta^0$ from \eqref{eq:Rdelta0}.
Then the operator
$$
    (D'')^\ast D'' : \Rcal^0_\delta\lra L^2_\delta(\slfrak_E)
$$
has index zero.
Moreover, $\ker((D'')^\ast D'')=\ker D''$.
\end{Proposition}

\begin{proof}
    As in the proof of Proposition \ref{prop:index}, we have
    $$
    \ind((D'')^\ast D'') = \ind(T_{A_0})
    $$
    where $T_{A_0}=\dbar_{A_0}^\ast\dbar_{A_0}$.
    Now consider the same operator with different domain:
    $$
    \widetilde T_{A_0} : L^2_{2,\delta}(\slfrak_E)\lra L^2_\delta(\slfrak_E)
    $$
    As in the proof of \cite[Lemma 7.3]{LockhartMcOwen:85},
    $$
    \ind(\widetilde T_{A_0})=\dim\ker_{L^2_{2,\delta}}(\dbar_{A_0}^\ast\dbar_{A_0})-
    \dim\ker_{L^2_{2,-\delta}}(\dbar_{A_0}^\ast\dbar_{A_0})
    $$
    Expanding in terms of Fourier modes as in the proof of Proposition
    \ref{prop:boundary-map}, we see that
    \begin{equation} \label{eqn:nude-index}
        \ind(\widetilde T_{A_0})=-\dim\bigoplus_{p\in
        D}\lfrak_p
    \end{equation}
Moreover, if $\eta\in \ker_{L^2_{2,-\delta}}(\dbar_{A_0}^\ast\dbar_{A_0})$,
    then $\eta\in \Rcal^0_\delta$,  and conversely by Proposition
    \ref{prop:boundary-map}, $\Rcal^0_\delta\subset L^2_{2,-\delta}$. The
    first statement then follows.

    For the second statement, suppose $\eta\in \Rcal^0_\delta$ satisfies
$(D'')^\ast D''\eta=0$. Then it suffices to show $D''\eta\in L^2_\delta$. For
if this is the case, then for the cut-off function $\phi_R$ as in \eqref{eqn:phiR},  and using the
fact that $\eta$ is bounded (cf.\ Proposition \ref{prop:boundary-map}),
we have
$$
\lim_{R\to +\infty}\left|\langle D''\eta, (d\phi_R)\eta\rangle_{L^2}\right|\leq
\lim_{R\to +\infty}\Vert D''\eta\Vert_{L^2_\delta}\Vert
(d\phi_R)\eta\Vert_{L^2_{-\delta}}=0
$$
It follows that
\begin{align*}
0&=\lim_{R\to +\infty} \langle(D'')^\ast D''\eta, (1-\phi_R)\eta\rangle_{L^2}\\
&=\lim_{R\to +\infty} \left\{\langle D''\eta, (1-\phi_R)D''\eta\rangle_{L^2}
-\langle D''\eta, (d\phi_R)\eta\rangle_{L^2}
\right\} \\
&=\Vert D''\eta\Vert^2_{L^2}
\end{align*}
and so $D''\eta=0$.

Now to show $D''\eta\in L^2_\delta$, first note that  by the definition of
$\Rcal_\delta$, $\dbar_{A_0}\eta\in L^2_{1,\delta}$. We have
$\dbar_A=\dbar_{A_0}+\beta$, with $\beta\in  L^2_{1,\delta}$, and so
$\dbar_A\eta\in L^2_{\delta}$ because $\eta$ is also bounded.
Write $\Phi=\Phi_0+\Phi_1$ according to Proposition
\ref{prop:higgs-decomposition}. Then since $b$ is bounded, $[\Phi_1,b]\in
L^2_\delta$. Because of Assumption A, as in the proof of
\eqref{eqn:bracket_estimate},   we also have
$[\Phi_0,b]\in L^2_\delta$.
Hence, $D''\eta=(\dbar_A+\Phi)\eta\in L^2_\delta$, the proof is complete.
\end{proof}

\subsubsection{The Hodge slice}
Let $(A,\Phi)$ be a \emph{smooth} parabolic Higgs pair. 
Consider the deformation complex
\begin{align} 
    \begin{split}\label{eqn:def-complex}
        \Ccal_\delta^{par}(A,\Phi) : & \\
        L^2_{2,\delta}(\slfrak_E)&\xrightarrow{\hspace{.1in}d_1\hspace{.05in}}
L^2_{1,\delta}(\slfrak_E\otimes\overline K_X )\oplus
\Dscr_\delta(\slfrak_E\otimes K_X)
\xrightarrow{\hspace{.1in}d_2\hspace{.05in}}
L^2_{\delta}(\slfrak_E\otimes K_X\otimes \overline K_X) 
    \end{split}
\end{align}
where 
\begin{align}
    \begin{split} \label{eqn:d1d2}
d_1(\eta)&= (\dbar_A\eta, [\Phi,\eta]) \\
d_2(\beta,\varphi) &= \dbar_A\varphi+[\Phi,\beta]
    \end{split}
\end{align}
\begin{Lemma}[{\sc Stable implies simple}] \label{lem:simple}
    Consider the extension of $d_1$ to the larger domain  $\Rcal^0_\delta$. 
    If $(A,\Phi)$ is stable then $\ker d_1=\{0\}$.
\end{Lemma}

\begin{proof}
    Let $\eta\in \ker d_1$. 
  Let  $(\Ecal,\Phi)$ be the parabolic bundle   corresponding to
    $(A,\Phi)$ as in Proposition \ref{prop:analytic-higgs}.  
    Without loss of generality, we may assume $A=A_0$ on each cylinder
    $C(p)$. We now proceed as in Example \ref{ex:phi}. Namely, in the local
    unitary (resp.\ holomorphic) frame $\{e_i\}$ (resp.\ $\{s_i\}$) we
    write  
    $$
    \eta e_i =\sum_{j=1}^n\widehat\eta_{ij}e_j =\sum_{j=1}^n \eta_{ij}(z)
    |z|^{\alpha_j(p)-\alpha_i(p)}e_j
    $$
    where the $\eta_{ij}(z)$ are holomorphic functions on
    $\Delta^\times(p)$.
    Since $\widehat \eta_{ij}$ is bounded (cf.\ Proposition
    \ref{prop:boundary-map}), $\eta_{ij}(z)$ is regular on
    $\Delta(p)$.
    Moreover,
     $\eta_{ij}(0)= 0$ if $\alpha_i(p)> \alpha_j(p)$.
Hence, $\eta$ defines a parabolic endomorphism. 
Since  stability implies simplicity, this forces $\eta\equiv 0$. 
\end{proof}

Define the cohomology groups of the complex:
$$
    H^i(\Ccal_\delta^{par}(A,\Phi)) =\begin{cases} \ker d_1 &i=0 \\
        \ker d_2/\imag d_1& i=1 \\
    \coker d_2 & i=2\end{cases}
$$
We have the formal $L^2$ adjoints:
\begin{align*}
d_1^\ast(\beta,\varphi)&= \dbar_A^\ast\beta- i\Lambda[\Phi^\ast,\varphi] \\
    d_2^\ast B &= ([\Phi^\ast,i\Lambda B], \dbar_A^\ast B)
\end{align*}
 For the weighted Sobolev spaces, the natural adjoints are
$d_1^{\ast_\delta}:= e^{-\tau\delta} d_1^\ast  e^{\tau\delta}$
and $d_2^{\ast_\delta}:= e^{-\tau\delta} d_2^\ast  e^{\tau\delta}$.
We have the following immediate but important consequence of Lemma
\ref{lem:simple}.
\begin{Proposition} \label{prop:smooth-B}
    The space $\Bcal_\delta^{par,s}$ (resp.\ $\Bcal_\delta^{spar,s}$)
    is a  smooth Banach submanifold of \break
    $
L^2_{1,\delta}(\slfrak_E\otimes\overline K_X )\oplus
\Dscr_\delta(\slfrak_E\otimes K_X)
    \ (\text{resp.}\
L^2_{1,\delta}(\slfrak_E\otimes\overline K_X )\oplus
    L^2_{1,\delta}(\slfrak_E\otimes K_X))
    $.
\end{Proposition}

\begin{proof}
By the implicit function theorem, it suffices to prove that $d_2$ is
    surjective if $(A,\Phi)$ is stable. 
After a complex gauge transformation we may assume that $A=A_0$ on each
    $C_i$. Then as in the proof of Proposition \ref{prop:index}, we see that if $B\in
    (\imag d_2)^\perp$, then $e^{\tau\delta}B\in \Rcal^0_{\delta}$ and indeed,
    $e^{\tau\delta}i\Lambda B^\ast\in \ker d_1$. 
    Hence, $\coker d_2=\{0\}$ by Lemma \ref{lem:simple}.
\end{proof}

\begin{Definition} \label{eqn:harmonic-h1}
    Given $(A,\Phi)\in \Bcal_{\delta}^{par}$ smooth, let
    $$
    \Hbold^{par}_\delta(A,\Phi)=\{ (\beta,\varphi)\in 
L^2_{1,\delta}(\slfrak_E\otimes\overline K_X )\oplus
\Dscr_\delta(\slfrak_E\otimes K_X)
    \mid d_2(\beta,\varphi)=0\ ,\ d_1^{\ast_\delta}(\beta,\varphi)=0\}
    $$
\end{Definition}

\begin{Lemma} \label{lem:cohomology}
    Given $(A,\Phi)\in \Bcal_{\delta}^{par,s}$ smooth,
we have: 
$$
H^1(\Ccal_\delta^{par}(A,\Phi))  \simeq\Hbold^{par}_\delta(A,\Phi)
$$
\end{Lemma}

\begin{proof}
    Let $(\beta,\varphi)\in \ker d_2$. 
    Because of Proposition \ref{prop:strong},  $d_1^\ast(\beta,\varphi)\in
L^2_\delta$.  By Lemma \ref{lem:simple}, 
    $$d_1^{\ast_\delta}d_1 : L^2_{2,\delta}(\slfrak_E)\lra L^2_{\delta}(\slfrak_E)
$$
is surjective, so we can find $\eta$ so that
    $d_1^{\ast_\delta}d_1\eta=- d_1^{\ast_\delta}(\beta,\varphi)$.
   Then 
    $
    (\beta,\varphi)+d_1\eta\in \Hbold^{par}_\delta(A,\Phi)
    $.
\end{proof}

By the same argument as in the proof of Proposition \ref{prop:smooth-B} 
(in this case $e^{\tau\delta}i\Lambda B^\ast\in L^2_{2,\delta}$), we have
\begin{Lemma}[{\sc Serre duality}] \label{lem:serre}
    For
    $(A,\Phi)\in \Bcal_{\delta}^{par}$,   we have:
    $$H^2(\Ccal_\delta^{par}(A,\Phi))\simeq
    H^0(\Ccal_\delta^{par}(A,\Phi))^\ast$$
\end{Lemma}

By Lemma \ref{lem:serre} and Proposition \ref{prop:index}, it follows that
\begin{equation} \label{eqn:dim-par-analytic}
\dim \Hbold^{par}_\delta(A,\Phi)=(2g-2+d)\dim\sG+\sum_{p\in D} \dim \sL_p
    \end{equation}
for all $(A,\Phi)\in \Bcal_{\delta}^{par,s}$.

Let us also recall the strongly parabolic deformation complex. 
Let $(A,\Phi)\in \Bcal^{par}_{\delta}$ be smooth and define:

\begin{align}\begin{split} \label{eqn:def-complex-strong}
    \Ccal_\delta^{spar}(A,\Phi) : &\\
    L^2_{2,\delta}(\slfrak_E)&\xrightarrow{\hspace{.1in}d_1\hspace{.05in}}
L^2_{1,\delta}(\slfrak_E\otimes\overline K_X )\oplus
L^2_{1,\delta}(\slfrak_E\otimes K_X)
\xrightarrow{\hspace{.1in}d_2\hspace{.05in}}
L^2_{\delta}(\slfrak_E\otimes K_X\otimes \overline K_X) 
\end{split}
\end{align}
Note that by \eqref{eqn:eta_continuity} the complex is well-defined, i.e.\  we
\emph{do not} need to assume that $(A,\Phi)$ is strongly
parabolic. 
Define the harmonics $$\Hbold^{spar}_{\delta}(A,\Phi)\simeq
H^1(\Ccal_\delta^{spar}(A,\Phi)):=\ker d_2/\imag d_1$$ in the same way
as \eqref{eqn:harmonic-h1}. Then as in the proof of  
Proposition  \ref{prop:smooth-B}, for stable Higgs bundles we have vanishing of $H^0$ and $H^2$ of the
complex \eqref{eqn:def-complex-strong}, and so 
\begin{align} 
    \begin{split}\label{eqn:dim-spar}
    \dim \Hbold^{spar}_{\delta}(A,\Phi)&=(2g-2)\dim\sG+2\sum_{p\in D}
        \dim(\sG/\sP_p)
+\sum_{p\in D}\dim \sL_p\\
    &=\dim \Hbold^{par}_{\delta}(A,\Phi)- \sum_{p\in D}\dim \sL_p
    \end{split}
\end{align}
for all $(A,\Phi)\in \Bcal_{\delta}^{par,s}$ (cf.\
\eqref{eqn:dim-par-analytic}).

We have the following useful consequence.
From Proposition \ref{prop:higgs-decomposition}, there is a well-defined
residue map
\begin{equation} \label{eqn:harmonic-residue}
    \Hbold_\delta^{par}(A,\Phi)\lra \bigoplus_{p\in D}\lfrak_p :
[(\beta,\varphi)]\mapsto \lim_{\tau\to +\infty} \varphi_0
\end{equation}
The kernel is exactly $\Hbold_\delta^{spar}(A,\Phi)$.
The dimension count above \eqref{eqn:dim-spar} then implies the next
result.

\begin{Corollary} \label{cor:residue}
    The map \eqref{eqn:harmonic-residue}
is surjective.
\end{Corollary}

The space $\Hbold_\delta^{par}(A,\Phi)$ (resp.\
$\Hbold_\delta^{spar}(A,\Phi)$)
is identified with the tangent space to $\Mbold^{par,s}_{\Dol,\ast}$
(resp.\ $\Mbold^{spar,s}_{\Dol,\ast}$) for $(A,\Phi)\in \Bcal_\delta^{par,s}$
(resp.\ $(A,\Phi)\in \Bcal_\delta^{spar,s}$) at the point $[(A,\Phi)]$.
We now proceed to define coordinate neighborhoods via the Kuranishi method. 
For
$(A,\Phi)\in \Bcal_{\delta}^{par,s}$,
following \cite[Def.\ 3.1]{CollierWentworth:19}, we define the \emph{Hodge
slice} by
\begin{align}\begin{split} \label{eqn:hodge-slice}
    \Scal^{par}_\delta(A,\Phi)&=\\
    \bigl\{
        (\beta,\varphi)&\in L^2_{1,\delta}(\slfrak_E\otimes\overline
        K_X)\oplus \Dscr_\delta \mid d_2(\beta,\varphi)+[\beta,\varphi]=0\
        ,\ d_1^{\ast_\delta}(\beta,\varphi)=0
        \bigr\}
\end{split}
\end{align}

Consider the operator 
$$d_2 d_2^{\ast_\delta}: L^2_{2,\delta}(\slfrak_E\otimes K_X\otimes\overline K_X)
\lra L^2_\delta(\slfrak_E\otimes K_X\otimes\overline K_X)
$$ 
Again using
\cite[Thm.\ 1.1]{LockhartMcOwen:85},
along with \eqref{eqn:phi_continuity} and the argument in the proof of Lemma
\ref{lem:smooth}, 
this is  a Fredholm operator of index $0$, and in
fact by Lemma \ref{lem:serre} it is invertible.  Let $G^{(2)}_\delta$ denote its Green's operator. 
We then define the \emph{Kuranishi map}
\begin{equation} \label{eqn:kuranishi}
    \kbold:\ker d_1^{\ast_\delta}\lra \ker d_1^{\ast_\delta} : (\beta,\varphi)\mapsto
    (\beta,\varphi)+d_2^{\ast_\delta} G^{(2)}_\delta([\beta,\varphi])
\end{equation}
An application of the inverse function theorem shows that $\kbold$ defines
a homeomorphism of neighborhoods $U,V$ of the origin from 
$$\kbold:
\Scal^{par}_\delta(A,\Phi)\supset U\isorightarrow
V\subset
\Hbold^{par}_{\delta}(A,\Phi)$$ 

Similarly,  if
$(A,\Phi)\in \Bcal_{\delta}^{spar,s}$ we define
\begin{align}\begin{split} \label{eqn:strong-hodge-slice}
    \Scal^{spar}_\delta(A,\Phi)=&\\
    \bigl\{
        (\beta,\varphi)\in L^2_{1,\delta}(\slfrak_E\otimes\overline
        K_X)&\oplus L^2_{1,\delta}(\slfrak_E\otimes
        K_X)\mid d_2(\beta,\varphi)+[\beta,\varphi]=0\
        ,\ d_1^{\ast_\delta}(\beta,\varphi)=0
        \bigr\}
\end{split}
\end{align}
In this case, the Kuranishi map $\kbold$ defines
a homeomorphism of neighborhoods of the origin from $\Scal^{spar}_\delta(A,\Phi)$ to
$\Hbold^{spar}_{\delta}(A,\Phi)$.

\subsubsection{Proof of Theorem \ref{thm:analytic-moduli}}
The existence of a complex manifold structure on the moduli space is now very
standard (cf.\  \cite{Kim:87} and \cite[Ch.\ VII]{DiffGeomCompVectBun}).

\begin{Proposition}[{\sc Local slice}] \label{prop:kim}
Let $(A,\Phi)\in \Bcal_{\delta}^{par,s}$ be a smooth, stable, parabolic
Higgs pair. 
Then the  map
$$
    p_{\H} : 
\Scal_{\delta}^{par}(A,\Phi)\lra \Mbold_{\Dol,\ast}^{par,s}(\alpha,\delta)
: (\beta,\varphi)\mapsto [(\dbar_{A}+\beta,\Phi+\varphi)]
$$
is a local homeomorphism from an open neighborhood of the origin in 
$\Scal_{\delta}^{par}(A,\Phi)$ to an open neighborhood of $[(A,\Phi)]$. This gives a local coordinate chart on
 $\Mbold_{\Dol,\ast}^{par,s}(\alpha,\delta)$.
The transition functions for the local coordinate charts constructed in
    this way  are holomorphic.
\end{Proposition}

We first need a generalization of the Poincar\'e inequality.

\begin{Lemma} \label{lem:mpim}
    Fix $(A,\Phi)$ be as in Proposition \ref{prop:kim}, and let $d_1$ denote the operator defined
    in \eqref{eqn:d1d2}. Then there is a constant $C>1$, depending on $(A,\Phi)$,
    such that for any $\eta\in L^2_{2,\delta}(\slfrak_E)$,
    $$
    \Vert \eta\Vert_{L^2_{1,\delta}}\leq C\Vert
    d_1\eta\Vert_{\Dscr_\delta}\ . 
    $$
\end{Lemma}

\begin{proof} To the contrary, suppose the existence of a sequence $\eta_j$ with
    $\Vert\eta_j\Vert_{L^2_{1,\delta}}=1$ and
    $\Vert d_1\eta_j\Vert_{\Dscr_\delta}\to 0$.
    Notice that from the elliptic estimate this implies that $\Vert\eta_j\Vert_{L^2_{2,\delta}}$ is uniformly bounded. 
    We may therefore assume $\eta_j\to \eta$
    weakly in  $L^2_{2,\delta}$. 
    Choose a cut-off function $\psi$ vanishing
    on $X_0$ and identically $=1$ for $\tau$ large.  
    Since $A_0$ is flat on the cylinders $C(p)$, \eqref{eqn:poincare}
    implies
    \begin{equation} \label{eqn:preliminary}
    \Vert \psi\eta_j\Vert_{L^2_{1,\delta}}\leq
    C\Vert d_{A_0}(\psi\eta_j)\Vert_{L^2_\delta}
    \leq
    2C\Vert\dbar_{A_0}(\psi\eta_j)\Vert_{L^2_\delta}
    \end{equation}
    Now $A=A_0+\beta$ with $\beta \in L^2_{1,\delta}$. By the embedding
    $L^2_{1,\delta}\hookrightarrow L^4_{\delta}$, we have
    $$
    \Vert [\beta, \psi\eta_j]\Vert_{L^2_{\delta}}\leq C \Vert
    \beta\Vert_{L^2_{1,\delta}(\supp(\psi))}\Vert  \psi\eta_j
    \Vert_{L^2_{1,\delta}}
    $$
    Hence, we may choose $\psi$ so that the inequality
    \eqref{eqn:preliminary} holds with $A_0$ replaced by $A$ (and a
    different constant). By readjusting $\psi$, we conclude as 
 in  the proof of \eqref{eqn:poincare}, that $\Vert\eta\Vert_{L^2_{1,\delta}}\geq 1/2$ and
    $\dbar_{A}\eta=0$.
Since $\Vert
    [\Phi,\eta_j]\Vert_{L^2_{-\delta}}\to 0$
    and $\eta_j\to \eta$ strongly on compact sets, we have
     $[\Phi,\eta]=0$.  This  then contradicts the
    assumption that $\ker d_1=\{0\}$. 
\end{proof}

\begin{proof}[Proof of Proposition \ref{prop:kim}]
    The first assertion now follows exactly as in  \cite[Lemma 1.7]{Kim:87}.
    For the sake of completeness, we repeat the argument in our case.
    Fix a smooth $(A,\Phi)$. If
    $(\beta,\varphi)\in L^2_{1,\delta}(\slfrak_E\otimes \overline
    K_X)\times\Dscr_\delta$, and $g\in \Gcal_{\delta,\ast}$,
    denote by  $g(\beta,\varphi):=(\widetilde\beta,\widetilde\varphi)$,
    where
    $$
    g(\dbar_A+\beta,\Phi+\varphi)=(\dbar_A+\widetilde\beta,\Phi+\widetilde\varphi)\
    .
    $$
With this understood,
   define the map 
        $$
        \Psi : L^2_{2,\delta}(\slfrak_E)\times L^2_{1,\delta}(\slfrak_E\otimes \overline
    K_X)\times
        \Dscr_\delta\lra L^2_{\delta}(\slfrak_E) \ :\ 
        (\eta,\beta,\varphi)\mapsto\
        d_1^{\ast_\delta}[e^\eta(\beta,\varphi)]
    $$
    Then the derivative  $D_1\Psi(0,0,0)$ with respect to  $\eta$,
    evaluated at the  origin
    $(\eta,\beta,\varphi)=(0,0,0)$, is $d_1^{\ast_\delta}d_1$. 
    Since $(A,\Phi)$ is stable, by Lemma \ref{lem:simple},  $\ker d_1=\{0\}$, and $D_1\Psi(0,0,0)$ is
    therefore an isomorphism. By the
    implicit function theorem, for $(\beta,\varphi)$ in a neighborhood
    $U$ of the origin there is a unique $g\in
    \Gcal_{\delta,\ast}$ in a neighborhood $V$ of the identity  so
    that $g(\beta,\varphi)\in \ker d_1^{\ast_\delta}$.

    Suppose now that 
    $ g(\beta,\varphi)=(\widetilde\beta,\widetilde\varphi) $
    for some $g\in \Gcal_{\delta,\ast}$, where $(\beta,\varphi),
    (\widetilde\beta,\widetilde\varphi)\in \ker d_1^{\ast_\delta}$. 
    If we define $\eta$ by $g=1+\eta$, then by Proposition
    \ref{prop:boundary-map}, $\eta\in
    L^2_{2,\delta}(\glfrak_E)$. We compute: 
    $$
    d_1\eta=(\beta-\widetilde\beta+\eta\beta-\widetilde\beta\eta,
    \varphi-\widetilde\varphi+\eta\varphi-\widetilde\varphi\eta)\ .
    $$
    Using Proposition \ref{prop:strong} (which holds equally well for
    $\glfrak_E$-valued sections),  we obtain the estimate 
    \begin{equation} \label{eqn:d1_estimate} 
    \Vert d_1\eta\Vert_{\Dscr_\delta}\leq C'
    (1+\Vert\eta\Vert_{L^2_{2,\delta}})(\Vert(\beta,\varphi)\Vert_{\Dscr_\delta}+\Vert(\widetilde\beta,\widetilde\varphi)\Vert_{\Dscr_\delta})
    \end{equation} 
    for a constant $C'$. 
The usual elliptic estimate gives
$$
    \Vert\eta\Vert_{L^2_{2,\delta}}\leq C(
    \Vert\eta\Vert_{L^2_{1,\delta}}+\Vert\dbar_{A_0}\eta\Vert_{L^2_{1,\delta}})\ .
$$
By patching the elliptic estimate on a compact set to the argument using
    multiplication properties as above, we can replace $A_0$ with $A$ in
    the above equation to obtain 
$$
    \Vert\eta\Vert_{L^2_{2,\delta}}\leq C(
    \Vert\eta\Vert_{L^2_{1,\delta}}+\Vert d_1\eta\Vert_{\Dscr_\delta})\ .
$$
    Finally, using this, along with \eqref{eqn:d1_estimate} and  Lemma
    \ref{lem:mpim}, we have
    $$
    \Vert\eta\Vert_{L^2_{2,\delta}}\leq 
\frac{C'(\Vert(\beta,\varphi)\Vert_{\Dscr_\delta}+\Vert(\widetilde\beta,\widetilde\varphi)\Vert_{\Dscr_\delta})}
{1-CC' (\Vert(\beta,\varphi)\Vert_{\Dscr_\delta}+\Vert(\widetilde\beta,\widetilde\varphi)\Vert_{\Dscr_\delta})}
    $$
    If $(\beta,\varphi)$ and $(\widetilde\beta,\widetilde\varphi)$ are
    sufficiently small, $g\in V$ and is therefore the identity. This proves
    that the slice is homeomorphic onto its image. 
    The second assertion on the holomorphicity follows similarly, as
in \cite[Lemma 2.3]{Kim:87}. 
\end{proof}

\begin{proof}[Completion of the Proof of Theorem \ref{thm:analytic-moduli}]
By the discussion at the end of the previous section, 
$\Scal_{\delta}^{par}(A_0,\Phi_0)$
is locally homeomorphic to a domain in complex euclidean space of the
correct dimension. This gives a local coordinate chart on 
$\Mbold_{\Dol,\ast}^{par,s}(\alpha,\delta)$.  As 
we also have that the transition functions are holomorphic, this therefore gives
$\Mbold_{\Dol,\ast}^{par,s}(\alpha,\delta)$ the structure of a complex manifold.   
The same arguments apply to the strongly parabolic case.
This proves 
statements (1) and (2) of the theorem.

For (3), we claim that the based moduli spaces 
$\Mbold_{\Dol,\ast}^{par,s}(\alpha,\delta)$ and
$\Mbold_{\Dol,\ast}^{spar,s}(\alpha,\delta)$ represent the moduli
functors for families of \emph{framed} parabolic and strongly parabolic 
Higgs bundles introduced in \cite{LogaresMartens}.  
Below we shall only give a few details in the parabolic case, for example. 

An important first step is the existence
of a universal family on $X\times \Scal_{\delta}^{par,s}(A_0,\Phi_0)$. 
Let $\widehat E\to X^\times\times \Scal_{\delta}^{par,s}(A_0,\Phi_0)$
be the pullback of the bundle $E\to X^\times$ above. 
Consider the $\dbar$-operator given  at the point $(\beta,\varphi)$ 
$$
\dbar_{\widehat E}=\dbar_{\Scal}+\dbar_{A_0}+\beta
$$
Here, $\dbar_{\Scal}$ is short-hand for the Cauchy-Riemann operator on
$\Scal_{\delta}^{par,s}(A_0,\Phi_0)$, which acts on sections of $\widehat E$
because it is a pullback from $X^\times$. 
Notice that $\beta$, regarded as a form on
$\Scal_{\delta}^{par,s}(A_0,\Phi_0)$, varies holomorphically. Hence,
$\dbar_{\widehat E}^2=0$, and this therefore gives a holomorphic structure
on $\widehat \Ecal$ on  $\widehat E$. As in 
Proposition \ref{prop:analytic-higgs},
there is a natural extension, also denoted $\widehat \Ecal$,  of
$\widehat E$ as a holomorphic bundle on  $ X\times
\Scal_{\delta}^{par,s}(A_0,\Phi_0)$. 
This extension depends on a choice of trivial holomorphic bundle on $ D\times
\Scal_{\delta}^{par,s}(A_0,\Phi_0)$. 
The parabolic structure is given by the growth of holomorphic sections with
respect to the background metric on $\widehat E$.  The associated graded
bundle on $D\times \Scal_\delta^{par,s}$ is naturally trivialized by
the germs of the $s_i$.
Finally, the relative logarithmic  Higgs field 
is defined by  $\widehat\Phi_0+\varphi$, where $\widehat\Phi_0$ is the
pullback of $\Phi_0$ and $\varphi$ is the tautological form defined as with
$\beta$. 

Using the existence of the universal family one can show as in
\cite[Sec.\ 6]{Norton:79},  or the detailed exposition in 
\cite[Sec.\ 4.2]{FriedmanMorgan:94},
that 
$\Mbold_{\Dol,\ast}^{par,s}(\alpha,\delta)$ locally represents the analytic functor of framed
parabolic bundles. The construction of an algebraic space, $\Pcal_0(\alpha)$, 
representing the algebraic functor is
outlined in \cite{LogaresMartens}. The isomorphism of
$\Mbold_{\Dol,\ast}^{par,s}(\alpha,\delta)$
with the analytification of $\Pcal_0(\alpha)$ then follows as in
\cite{Miyajima:89}. Finally, it is clear that the identification is
equivariant with respect to the  action of $\overline\Lbold$, and this in turn
proves assertion (3) of Theorem \ref{thm:analytic-moduli}.
\end{proof}

\subsubsection{Variation of the harmonic metric} \label{sec:harmonic-metric}
Having defined the manifold structure on the framed moduli space 
$\Mbold^{par,s}_{\Dol,\ast}(\alpha,\delta)$
and its quotient $\Mbold^{par,s}_{\Dol}(\alpha,\delta)$, we can now prove a statement
about the variation of the harmonic metric.
Notice that the space \eqref{eqn:metrics} of admissible metrics has the
structure of a Banach manifold with tangent space
$$T_{h_0}\Mcal_\delta\simeq \sqrt{-1}\,\Ucal_\delta\ .$$

\begin{Theorem}[{\sc Variation of harmonic metric}] \label{thm:first-variation}
Let $(A,\Phi)\in \Bcal^{par,s}_\delta$ be a solution to \eqref{eqn:hitchin} for
the metric $h_0$,
and let $\Hscr:\Scal^{par}_\delta(A,\Phi)\to \Mcal_\delta$ denote the
family of
harmonic metrics from Theorem \ref{thm:hk-metric}.
Then $\Hscr$ is a $C^\infty$ map. Moreover, the derivative $\Hscr$ at the origin
vanishes.
\end{Theorem}

\begin{proof}
Define the map
\begin{align*}
\Nscr&: \Scal^{par}_\delta(A,\Phi)\times \Mcal_\delta \lra
L^2_{\delta}(i\sufrak_E) \\
&(\beta,\varphi; h)\mapsto i\Lambda\left(F_{(\dbar_E+\beta,h)}
+[\Phi+\varphi, (\Phi+\varphi)^{\ast_h}] \right)
\end{align*}
Then the derivative of $\Ncal$ with respect to $\Mcal_\delta$ at the origin is
$$
d_2\Ncal_{(0,0;h_0)}: \sqrt{-1}\, \Ucal_\delta\lra
 L^2_{\delta}(i\sufrak_E)
 :
 \dt h\mapsto i\Lambda\left(\dbar_E\partial_E^{h_0}(\dt
h)+[\Phi,[\Phi^\ast_{h_0},\dt h]]\right)
$$
A calculation shows that
$d_2\Ncal_{(0,0;h_0)}(\dt h)=(D'')^\ast D''(\dt h)$.
By  Proposition \ref{prop:nude-index}
(note that $(D'')^\ast D''$ preserves the
hermitian subspace)  and Lemma \ref{lem:simple}, we conclude
that $d_2\Ncal_{(0,0;h_0)}$ is an isomorphism. The smoothness of $\Hscr$ then
follows from the implicit function theorem.
The second statement follows exactly as in \cite[Prop.\
3.12]{CollierWentworth:19}.
\end{proof}

\subsubsection{The hyperk\"ahler metric} \label{sec:hyperkahler}
Recall the residue map  (see \eqref{eqn:residue}
$$\mathrm{Res}: \Mbold^{par,s}_{\Dol}(\alpha,\delta)\lra \bigoplus_{p\in D}\lfrak_p$$
(as well as Assumption A).
The following is an immediate consequence of Corollary \ref{cor:residue}.
\begin{Proposition} \label{prop:res-submersion}
The map $\res_{\lfrak}$ is a holomorphic submersion.
\end{Proposition}

The goal of this section is to prove the following.
\begin{Theorem}[{\sc Hyperk\"ahler structure}]  \label{thm:hyperkahler}
For any $\ell\in \bigoplus_{p\in D}\lfrak_p$, the fiber
    $\res^{-1}_{\lfrak}(\ell)$ is a hyperk\"ahler manifold.
\end{Theorem}
This generalizes the results of Konno \cite{Konno:93} and Nakajima
\cite{Nakajima:96}.
The first step is the definition of a holomorphic 2-form.
On $\Bcal^{par,s}_\delta$, we set:
\begin{equation} \label{eqn:symplectic-form}
    \Omega((\beta_1,\varphi_1), (\beta_2,\varphi_2)):=
    i\int_{X^\times}\tr(\varphi_2\wedge\beta_1-\varphi_1\wedge\beta_2)
\end{equation}
This is well-defined, since we have a natural duality
$
L^2_{-\delta}\times L^2_\delta\hookrightarrow L^1
$.
Clearly, $\Omega$ is holomorphic and closed. We have the following:

\begin{Proposition} \label{prop:based-symplectic}
The form $\Omega$ descends to a holomorphic symplectic form
    $\Omega_\ast$  on
    $\Mbold^{par,s}_{\Dol, \ast}$.
\end{Proposition}

We first prove the following.
\begin{Lemma} \label{lem:nude-gauge}
    Let $(A,\Phi)\in \Bcal_\delta^{par,s}$ be a smooth Higgs pair, 
    and let $(\beta,\varphi)\in L^2_{1,\delta}$ be in the kernel of $d_2$.  
Then there is a unique $\eta\in \Rcal^0_\delta$ such that
    $(\beta,\varphi)+d_1\eta\in \ker d_1^\ast$. 
\end{Lemma}

\begin{proof}
It suffices to solve
$
    d_1^\ast d_1\eta=-d_1^\ast(\beta,\varphi)
$.
    By Proposition \ref{prop:nude-index}, this is possible if
    $\ker(d_1^\ast d_1)=\{0\}$ on $\Rcal^0_\delta$. As in the proof of  Proposition
    \ref{prop:boundary-map}, it follows easily that $\ker(d_1^\ast
    d_1)=\ker d_1$, and hence the desired result is a consequence of Lemma
    \ref{lem:simple}.
\end{proof}

\begin{proof}[Proof of Proposition \ref{prop:based-symplectic}]
We have
    \begin{equation} \label{eqn:omega-integration}
    \Omega((\beta,\varphi), d_1\eta):=
    -i\int_{X^\times}\tr(d_2(\beta,\varphi)\eta)
    +i\int_{X^\times}d\tr(\varphi\eta)
    \end{equation} 
    The second term on the right hand side vanishes for $\eta\in
    L^2_{2,\delta}$ (cf.\ \eqref{eqn:eta_continuity}). Hence,
    $\Omega$ descends, and we denote by $\Omega_\ast$ the resulting
    holomorphic form on 
    $\Mbold^{par,s}_{\Dol, \ast}$.
    Now suppose $(\beta,\varphi)\in \ker d_2$ is such that 
    $
    \Omega_\ast((\beta,\varphi), \cdot)\equiv 0
    $. Recall from Proposition \ref{prop:higgs-decomposition}  
    that there is a decomposition $\varphi=\varphi_0+\varphi_1$, where
    $\varphi_1\in L^2_{1,\delta}$, and 
    a well-defined limit $\varphi_\infty:=\lim_{\tau\to +\infty}\varphi_0$. 
    If $\ell\in \bigoplus_{p\in D} \lfrak_p$ is arbitrary, and $\eta\in
    \Rcal^0_\delta$ such that $\lim_{\tau\to +\infty}\eta=\ell$, then the
    calculation in \eqref{eqn:omega-integration}  shows that  
    $$
    0= \Omega((\beta,\varphi), d_1\eta):= \tr(\varphi_\infty\ell)
    $$
    Since $\ell$ is arbitrary, we see that $\varphi_\infty=0$, and hence $\varphi\in L^2_{1,\delta}$. 
    By Lemma \ref{lem:nude-gauge}, we can then find a different $\eta\in
    \Rcal^0_\delta$ such that
    $$
    (\widetilde\beta,\widetilde\varphi)=(\beta,\varphi)+d_1\eta\in \ker
    d_1^\ast
    $$
    Notice that $(-\widetilde\varphi^\ast, \widetilde\beta^\ast)\in \ker
    d_2$. The condition  
    $$\Omega((\widetilde\beta,\widetilde \varphi),(\widetilde\varphi^\ast, -\widetilde\beta^\ast
    ))=0
    $$
    implies that $(\widetilde \beta,\widetilde\varphi)=0$, or that
    $(\beta,\varphi)=-d_1\eta$. Recall that $\eta$ has a well-defined limit
$\eta_\infty:=\lim_{\tau\to +\infty}\eta
    $.
    By Corollary \ref{cor:residue}, for arbitrary $\ell\in
    \bigoplus_{p\in D}\lfrak_p$ we can find $(\beta_1,\varphi_1)\in \ker
    d_2$ such that $(\varphi_1)_\infty=\ell$. 
    Again using \eqref{eqn:omega-integration}, we have
    $$
    0=\Omega_\ast((\beta,\varphi),(\beta_1,\varphi_1))=\tr(\eta_\infty\ell)
    $$
Since $\ell$ is arbitrary, we conclude that $\eta_\infty=0$, and so by
    Proposition \ref{prop:boundary-map}, $\eta\in L^2_{2,\delta}$. But then
    the class of $(\beta,\eta)$ is zero. Hence, $\Omega_\ast$ is
    nondegenerate.
\end{proof}

\begin{Proposition} \label{prop:poisson}
The symplectic form $\Omega$ induces a Poisson structure on
    $\Mbold^{par,s}_{\Dol}$ via the fibration 
    $$p_\ast: \Mbold^{par,s}_{\Dol,\ast}\lra\Mbold^{par,s}_{\Dol}$$
    The symplectic leaves of $\Mbold^{par,s}_{\Dol}$ are the fibers
    $\res^{-1}_\lfrak(\ell)$, for $\ell\in \bigoplus_{p\in D}\lfrak_p$. 
\end{Proposition}

\begin{proof}
Since the residual gauge group $\overline\Lbold$ preserves the symplectic form
    $\Omega_\ast$, it is standard that the Poisson structure on
    $\Mbold^{par,s}_{\Dol,\ast}$ 
    induces one on the quotient $\Mbold^{par,s}_{\Dol}$. 
    It is easy to see using Proposition \ref{prop:higgs-decomposition}
    and \eqref{eqn:omega-integration} that the fibers of $p_\ast$ are
    isotropic. It follows that the rank of the Poisson structure on $\Mbold^{par,s}_{\Dol}$
    coincides with the dimension of the fibers $\res^{-1}_\lfrak(\ell)$. Hence, to
    show these are the symplectic leaves it suffices to show that
    $\Omega_\ast$ descends to a symplectic form $\Omega_\ell$ on each
    $\res^{-1}_\lfrak(\ell)$. The
    tangent space to a fiber of
    $$
    \Mbold^{par,s}_{\Dol,\ast}\lra  \bigoplus_{p\in D}\lfrak_p
    $$
    consists of $(\beta,\varphi)\in \ker d_2$ with $\varphi\in
    L^2_{2,\delta}$. The 
    fact that the restriction of $\Omega_\ast$ to  this fiber descends
    follows again from \eqref{eqn:omega-integration}. 
    To prove nondegeneracy, note that by  Lemma \ref{lem:nude-gauge} we may
    assume $(\beta,\varphi)\in \ker d_1^\ast$. Then $(-\varphi^\ast,
    \beta^\ast)\in \ker d_2$, and 
    $$
    \Omega_\ast((\beta,\varphi),(-\varphi^\ast,\beta^\ast))=\Vert
    \beta\Vert^2_{L^2}+\Vert\varphi\Vert^2_{L^2}
    $$
    Nondegeneracy of $\Omega_\ell$ follows, and this completes the proof.
\end{proof}

\begin{proof}[Proof of Theorem \ref{thm:hyperkahler}]
This now follows as in the case of Higgs bundles on closed Riemann
    surfaces. Namely, on $\res^{-1}(\ell)$ we choose representatives for
    tangent vectors  $(\beta,\varphi)\in \ker d_2\cap \ker d_1^\ast$.  That we may do this
    follows from Lemma \ref{lem:nude-gauge}.  
    We then define almost complex structures
    \begin{align*}
        I(\beta,\varphi)&=(i\beta,i\varphi)\\
        J(\beta,\varphi)&=(i\varphi^\ast,-i\beta)\\
        K(\beta,\varphi)&=(-\varphi^\ast,\beta^\ast)
    \end{align*}
    Then with respect the $L^2$ metric the associated fundamental forms
    $\Omega_I=\Omega_\ell$, $\omega_J$, and $\omega_K$ are closed. The integrabilty of
    $J$ and $K$ then follows from \cite[Lemma 6.8]{selfduality}. The complex
    structure we have defined on $\res^{-1}_\lfrak(x)$ coincides with $I$. 
\end{proof}

\subsection{deRham moduli spaces}

\subsubsection{Construction of moduli of flat bundles} \label{sec:flat}
The method used in \S\ref{sec:dolbeault}
to construct the Dolbeault moduli spaces of parabolic Higgs bundles may be
used to form the deRham moduli spaces. As in the case of Higgs bundles,  we only describe the set-up for
the strongly parabolic case, or the case of full flags.
Since the details should be clear from the previous discussion, we omit proofs. 

In parallel with Definition \ref{def:higgs-bundle}, consider the following
\begin{Definition}  \label{def:analytic-flat-connection}
  Let
    \begin{align*} 
        \Fcal_{\delta}^{par} &=\left\{(A,\Phi)\in \Acal_{\delta}\times
    \Dscr_\delta(\slfrak_E\otimes K_X) \mid
    \nabla=d_A+\Phi\text{ is projectively flat}\right\} \\
    \Fcal_{\delta}^{spar}&=\left\{(A,\Phi)\in \Acal_{\delta}\times
        L^2_{1,\delta}(\slfrak_E\otimes K_X) \mid
    \nabla=d_A+\Phi\text{ is projectively flat}\right\}
    \end{align*} 
\end{Definition}
The flatness condition is 
\begin{equation} \label{eqn:phi-flat}
    F^\perp_\nabla=
F^\perp_A+\dbar_A\Phi=0
\end{equation}
which makes sense as an equation in $L^2_\delta$ because of the defining property of
$\Dscr_\delta$ (see Remark \ref{rem:general-A}). 
The group $\Gcal_{\delta}$ acts on $\Fcal_{\delta}^{par}$ by
pulling back the complex connection  $\nabla$  and rewriting in terms of
a unitary and $(1,0)$ part. 
Explicitly,  for $g\in \Gcal_{\delta}$,
$$
g(\nabla):=g\circ\nabla\circ g^{-1}=d_{g(A)} + g^{-1}\Phi g -
\partial_A(g)g^{-1}-(g^\ast)^{-1}\partial_A(g^\ast)\ .
$$

By Lemma \ref{lem:smooth}, we can find a
gauge transformation such that $\dbar_\nabla=\dbar_A$ is a
smooth $\dbar$ operator. Hence, $F_A$ is smooth, and 
it then 
follows from elliptic regularity and \eqref{eqn:phi-flat} that  $\Phi$ is smooth. 
So $\nabla$ defines a holomorphic connection on $X^\times$. 
Using Lemma \ref{lem:smooth} once again, we may assume $A=A_0$ on the
cylinders. Since $A_0$ is flat on the cylinders, \eqref{eqn:phi-flat}
implies $\dbar_{A_0}\Phi=0$ on each $C(p)$. From Example
\ref{ex:phi} one sees that $\nabla$ defines a logarithmic connection on
$X$  on a
holomorphic bundle with a parabolic structure, and
$\res_p(\nabla)=\alpha(p)+\res_p(\Phi)$ for each $p\in D$.

Set
$$
\Mbold^{spar,s}_{\DR}(\alpha,\delta):=\Gcal_{\delta}\backslash\Fcal_{\delta}^{
spar , s }  \
,\quad
\Mbold^{par,s}_{\DR}(\alpha,\delta):=\Gcal_{\delta}\backslash\Fcal_{\delta}^{par
, s }  \
.$$
Then we have the following:

\begin{Theorem}[{\sc Moduli spaces of flat connections}]
\label{thm:analytic-DRmoduli}
    Fix $\alpha$, and suppose $\delta$ satisfies
    \eqref{eqn:delta-assumption}. Then  if nonempty
            the moduli spaces $\Mbold^{spar,s}_{\DR}(\alpha,\delta)$
            and $\Mbold^{par,s}_{\DR}(\alpha,\delta)$ are  smooth complex
            manifolds of
            complex dimensions 
            \begin{align*}
                \dim\Mbold^{spar,s}_{\DR}(\alpha,\delta)&=(2g-2)\dim\sG+2\sum_{p\in
                D}\dim(\sG/\sP_p) \\
                \dim\Mbold^{par,s}_{\DR}(\alpha,\delta)&=(2g-2+d)\dim\sG \
                .
            \end{align*}
    Moreover, the assignment of a logarithmic connection to $(A,\Phi)$
    described above
    induces biholomorphism 
    $$
        \Mbold^{spar,s}_{\DR}(\alpha,\delta) \isorightarrow
        \Scal\Pcal_1(\alpha) \quad ,\quad
        \Mbold^{par,s}_{\DR}(\alpha,\delta) \isorightarrow \Pcal_1(\alpha)
        \ .
        $$
\end{Theorem}

\subsubsection{Joint parametrization by the Hodge slice}
Let $(A,\Phi)\in \Bcal^{par,s}_{\delta}$ satisfy the Hitchin equations, so
that $D=D''+D'$, where $D''=\dbar_A+\Phi$, $D'=\partial_A+\Phi^\ast$ is a flat
connection: $D\in \Fcal^{par,s}_{\delta}$. Set:
$$
D'_\delta:= e^{-\tau\delta}D'e^{\tau\delta}\quad ,\quad
D_\delta=D''+D'_\delta
$$
Then: $F_{D_\delta}=\delta\dbar\partial\tau\cdot id$. Fix a $\U(1)$
connection $\nabla_\delta^0$  on the trivial line bundle $\underline\CBbb$ with curvature
$-\delta\dbar\partial\tau$, i.e.\ the Chern connection for the metric
$e^{-\tau\delta}$. Then clearly, under the identification
$\underline\CBbb\otimes E\simeq E$ given by the trivialization of
$\underline \CBbb$, $D=\nabla_\delta^0\otimes D_\delta$. 
We now define the map:
\begin{equation} \label{eqn:p-dR}
    p_{\DR} : \Scal^{par}_\delta(A,\Phi)\to \Mbold^{par,s}_{\DR,\ast}(\alpha,\delta)
    : (\beta,\varphi)\mapsto [\nabla^0_\delta\otimes
    (D_\delta+\beta+\varphi)]
\end{equation}
Then $p_{\DR}$ is well-defined on a neighborhood of the origin in the
slice. We have the following
\begin{Proposition} 
    The map $p_{\DR}$ gives a biholomorphism of a neighborhood of the
    origin in $\Scal^{par}_\delta(A,\Phi)$ to a neighborhood of $[D]\in
    \Mbold^{par,s}_{\DR,\ast}(\alpha,\delta)$.
\end{Proposition}

By the discussion in \S \ref{sec:flat}, there is a well-defined
residue map
$$
\mathrm{Res} : \Mbold^{par,s}_{\DR}(\alpha,\delta)\lra \bigoplus_{p\in D}\lfrak_p : \nabla\mapsto \bigoplus_{p\in D}\res_p(\nabla)
$$

We state the following without proof (cf.\ proof of Theorem
\ref{thm:hyperkahler}).

\begin{Theorem}
    For any $\ell\in \bigoplus_{p\in D}\lfrak_p$,
    the spaces $\mathrm{Res}^{-1}(\ell)\subset
    \Mbold^{par,s}_{\DR}(\alpha,\delta)$ are holomorphic symplectic
    manifolds. 
\end{Theorem}


\section{Stratifications and conformal limits} \label{sec:strat}

\subsection{Hodge pairs}
The goal of this section is to define a generalization of the
Bia{\l}ynicki-Birula decomposition of the Dolbeault moduli space in the
parabolic setting. This is based on Simpson's notion of a Hodge bundle
\cite{SimpsonThesis},
which in the moduli space is a fixed point for the $\CBbb^\ast$-action. In
our analytic setting, there is a subtlety regarding the relationship
between the Hodge and parabolic filtrations. We address this in \S\ref{sec:filtrations}.
In \S\ref{sec:BB} we define the Bia{\l}ynicki-Birula (or \BB) slices,
that were first introduced in \cite{CollierWentworth:19} in the nonparabolic
setting.

\subsubsection{The parabolic and Hodge filtrations} \label{sec:filtrations}

There is a natural action of  $\CBbb^\ast$ 
on $\Bcal^{par}_{\delta}$ given by $\lambda\cdot(A,\Phi)=(A,\lambda\Phi)$.
Since this commutes with the action of $\Gcal_{\delta}$, we have a
$\CBbb^\ast$-action on
$\Mbold_\Dol^{par,s}(\alpha,\delta)$. 
Suppose $[(A,\Phi)]\in \Mbold_\Dol^{par,s}(\alpha,\delta)$ is a fixed point of
$\CBbb^\ast$.  Then 
as in the closed surface case (cf.\ \cite[eq.\ (7.2)]{selfduality}), 
there is a representative $(A,\Phi)$ of $[(A,\Phi)]$ so that 
for every $\vartheta$ there is $g_\vartheta\in
\Kcal_{\delta}$, with the property  that $d_A=g_\vartheta(d_A)$ and
$e^{i\vartheta}\Phi=g_\vartheta\Phi g_\vartheta^{-1}$. 
Note that this implies that $g_\vartheta$ has constant eigenvalues. 
It follows that  $\Ecal=(E,\dbar_A)$  splits
holomorphically and isometrically  as 
\begin{align}
    \begin{split}\label{eqn:splitting}
        \Ecal&=\Ecal_1\oplus\cdots\oplus \Ecal_\ell \\
        \Phi : \Ecal_a &\to \Ecal_{a+1}\otimes K_X \ , \ a=1, \ldots, \ell-1
    \end{split}
\end{align}
and $\Phi$ annihilates $\Ecal_\ell$. 
The important point is that for generic $\vartheta$ the $E_i$ are
precisely the eigenbundles of $g_\vartheta$. 
Following Simpson, we call such a Higgs pair a \emph{Hodge pair}.

The following compares the quasi-parabolic structure with the
Hodge splitting.
\begin{Proposition} \label{prop:parabolic-hodge}
Let $(A,\Phi)$ be a Hodge pair. Then after possibly acting by
an element of $\Kcal_{\delta,\ast}$, the $C^\infty$ splitting
\eqref{eqn:splitting} is preserved by $d_{A_0}$ and is
    translation invariant  on each $C(p)$, $p\in D$. 
\end{Proposition}

\begin{proof}
Let
$$
    \Kcal^c_{\delta}=\left\{ g\in \Kcal_{\delta}\mid
    \bbold(g)=\bbold(g_\vartheta)\
,\ \charpoly(g)\text{ is constant }\right\}
$$
    where $\charpoly$ denotes the (pointwise) characteristic polynomial.
    Then $\Kcal_{\delta}^c\subset\Kcal_{\delta}$ is a Banach submanifold
with tangent space
$$
    T_g\,\Kcal_{\delta}^c=\left\{ [g,H] \mid H\in
L^2_{2,\delta}(\sufrak_E) \right\}
$$
    Choose $g_0\in \Kcal_{\delta}^c$  to be an element that is constant on
    each $C(p)$ (with respect to the fixed unitary frame $\{e_i\}$). Define: 
    $$F:\Kcal_{\delta,\ast}\to \Kcal_{\delta}^c  : h\mapsto h g_0 h^{-1}$$
Then 
        $$DF(1): T_{id}\,\Kcal_{\delta,\ast}\simeq
    L^2_{2,\delta}(\sufrak_E)\lra T_{g_0}\Kcal_{\delta}^c : H\mapsto [H,g_0]$$
        is surjective, and so by the implicit
function theorem $F$ is locally an open mapping. It follows that if $g$ is
sufficiently close to $g_0$, then $g=hg_0 h^{-1}$ for some $h$. 

Now returning to the case at hand, we can find an $h_0$ which is
the identity outside a large compact set, such that $h_0g_\vartheta
h_0^{-1}$ is sufficiently close to $g_0$. By the previous
    paragraph, we can then find $h\in \Kcal_{\delta}$ so that $hg_\vartheta
h^{-1}= g_0$. Hence, after a gauge transformation, we may assume
    $g_\vartheta=\bbold(g_\vartheta)$ is constant on each $C(p)$. 
    Write $d_A=d_{A_0}+b$. Since $g_\vartheta$ is constant on each $C(p)$,  we have
    $$
    0=d_A(g_\vartheta)=d_{A_0}(g_\vartheta)+[b,g_\vartheta]=[\sqrt{-1}\hat\alpha\,
    d\theta+b,g_\vartheta]
    $$
    But $b\in L^2_{\delta}$, so taking limits we conclude that
    $[\hat\alpha,g_\vartheta]=0$. 
        Hence, $A_0$ preserves the splitting given by $g_\vartheta$, and since $g_\vartheta$ is constant, the splitting is translation 
        invariant.  The result follows.
\end{proof}

\begin{Remark}
\begin{enumerate}
\item We note that the only $\CBbb^\ast$ fixed points in the \emph{framed}
moduli
space $\Mbold_{\Dol,\ast}^{par,s}(\alpha,\delta)$ lie in the locus where
$\Phi\equiv 0$.
Indeed, by the discussion above, $g_\vartheta$ has constant eigenvalues, and so
if $g_\vartheta$ were in the based gauge group $\Gcal_{\delta,\ast}$ it would
necessarily be the identity.
\item
It also follows that in the case of full flags, a Hodge pair is strongly
    parabolic. Indeed, by Proposition \ref{prop:parabolic-hodge}
    and the form \eqref{eqn:splitting}, it follows from Example
    \ref{ex:phi} that $\res(\Phi)$ is strictly upper triangular. 
    \end{enumerate}
\end{Remark}

As a consequence of Proposition \ref{prop:parabolic-hodge},
we may find a connection $\widetilde A_0$ that
preserves the splitting and agrees with $A_0$ on each $C(p)$, $p\in D$.
From now on we just assume $\widetilde A_0=A_0$ everywhere. Then
$A_0$ induces connections $A_0^{(a)}$ on the bundles $E_a$ that
are translation invariant. We may therefore define the weighted
Sobolev space of connections $\Acal_\delta^{(a)}$ for each
bundle $E_a$. We then have the inclusion:
$$
\Acal_\delta^{(1)}\times\cdots\times \Acal_\delta^{(\ell)}\lra
\Acal_{\delta}
$$
The Hodge pair structure induces
connections $A^{(a)}$ which are not a priori translation
invariant, but up to gauge they 
can be written as $d_{A^{(a)}}= d_{A_0^{(a)}} +
b^{(a)}$, with $b^{(a)}\in
L^2_{1,\delta}(\ufrak_{E_a}\otimes T^\ast X^\times)$.  
That is, $A^{(a)}\in \Acal_\delta^{(a)}$. 
Applying Lemma \ref{lem:smooth} for each $a$, we can assume that the
connection on the Hodge bundle agrees with $A_0$ on each $C(p)$. 
We shall say that the Hodge pair is \emph{in good gauge}. 

Now suppose $(A,\Phi)$ is a Hodge pair in good gauge. Following 
\cite{CollierWentworth:19} we consider bundles
$$
\nfrak^+_E=\bigoplus_{b>a}\Hom(E_b,E_a)\ ,\ \hfrak_E=\bigoplus_{a=1}^\ell
\End E_a \cap \slfrak_E
$$
with the induced connections from the $A^{(a)}$. 
Because the Hodge splitting is translation invariant on each $C(p)$, it makes sense to consider the subbundles $\nfrak_E^+\cap
\lfrak_p$ and $\hfrak_E\cap \lfrak_p$. Note that: 
\begin{equation} \label{eqn:direct-sum}
        \dim \lfrak_p=2\rank(\nfrak_E^+\cap\lfrak_p)+\rank(\hfrak_E\cap \lfrak_p)\ . 
\end{equation}

We then define the subcomplex of
\eqref{eqn:def-complex}:

\begin{align} 
    \begin{split}\label{eqn:bb-complex}
        \Ccal^{par,+}_\delta(A,\Phi) &: \\ 
        L^2_\delta(\nfrak^+_E)&\xrightarrow{\hspace{.1in}d_1\hspace{.05in}}
        L^2_{1,\delta}(\nfrak_E^+\otimes\overline K_X)\oplus
    \Dscr_\delta((\hfrak_E\oplus \nfrak_E^+)\otimes K_X)
\xrightarrow{\hspace{.1in}d_2\hspace{.05in}}
L^2_{\delta}(\hfrak_E\oplus\nfrak_E^+) 
    \end{split}
\end{align}
Similarly, we have the subcomplex of
\eqref{eqn:def-complex-strong}:

\begin{align} 
    \begin{split}\label{eqn:bb-complex-strong}
        \Ccal^{spar,+}_\delta(A,\Phi) &: \\ 
        L^2_\delta(\nfrak^+_E)&\xrightarrow{\hspace{.1in}d_1\hspace{.05in}}
        L^2_{1,\delta}(\nfrak_E^+\otimes\overline K_X)\oplus
    L^2_{1,\delta}((\hfrak_E\oplus \nfrak_E^+)\otimes K_X)
\xrightarrow{\hspace{.1in}d_2\hspace{.05in}}
L^2_{\delta}(\hfrak_E\oplus\nfrak_E^+) 
    \end{split}
\end{align}

\begin{Lemma} \label{lem:half-dimensional}
    Let $(A,\Phi)\in \Bcal^{spar,s}_\delta$ be a Hodge pair in good gauge.
    Then 
    \begin{align}
        \dim H^1(\Ccal^{par,+}_\delta(A,\Phi)) &=
            (g-1)\dim\sG+\sum_{p\in D} \dim(\sG/\sP_p) + \sum_{p\in D} \dim \sL_p\label{eqn:BB-par} \\
        \dim H^1(\Ccal^{spar,+}_\delta(A,\Phi)) &=
            (g-1)\dim\sG+\sum_{p\in D} \dim(\sG/\sP_p) + \sum_{p\in D}
        \rank(\nfrak_E^+\cap \lfrak_p)\label{eqn:BB-spar} 
    \end{align}
    Under the assumption of full flags, 
    \begin{equation} \label{eqn:BB-spar-prime}
        \dim H^1(\Ccal^{spar,+}_\delta(A,\Phi)) =
            (g-1)\dim\sG+\sum_{p\in D} \dim(\sG/\sP_p) 
        \tag{\theequation'} 
    \end{equation}
\end{Lemma}

\begin{proof}
    As in the proof of Proposition \ref{prop:index}, 
   elements of cohomology have representatives in  harmonic spaces
    $\Hbold^{par,+}_\delta(A,\Phi)$
    and $\Hbold^{spar,+}_\delta(A,\Phi)$.
    The dimension count then  reduces to
    computing the sum of the indices of the decoupled operators
\begin{align*}
    T_\beta^+&: L^2_{1,\delta}(\nfrak^+_E\otimes\overline K_X)\lra
    L^2_\delta(\nfrak_E) \\
    T_\varphi^+&: L^2_{1,\delta}((\hfrak_E\oplus \nfrak_E)\otimes K_X)\lra L^2_{\delta}(\hfrak_E\oplus \nfrak_E)
\end{align*}
 The operator $T_\beta^+$  is the adjoint of
$$
        \dbar_{A_0} : L^2_{1,\delta}(\nfrak^+_E)\lra
    L^2_\delta(\nfrak_E\otimes\overline K_X) 
$$
        As above, using \cite{APS:75} we have
        $$
        i(\dbar_{A_0})= \deg \nfrak_E^+ -
    \rank(\nfrak_E^+)(g-1+d/2)-\frac{1}{2}\sum_{p\in D}\rank(\nfrak_E^+\cap
    \lfrak_p)-\frac{\eta^+(0)}{2} 
        $$
        Here, $\eta^+(s)$ is the $\eta$-function for the boundary operator on $\nfrak_E^+$, 
        and by $\deg \nfrak_E^+$ we mean the integral of the $\alpha_0$ term in \cite[Thm.\ 3.10 (i)]{APS:75}.
    Similarly, 
    $$
    i(T^{+}_\varphi)=\deg \nfrak_E^++\deg\hfrak_E +(\rank
        \nfrak_E^++\rank\hfrak_E)(g-1+d/2)-\frac{1}{2}\sum_{p\in
    D}\rank((\nfrak_E^+\oplus\hfrak_E)\cap \lfrak_p)
    -\frac{\eta^+(0)}{2}
    $$
        Here, we have used the fact that the spectrum of the boundary operator on $\hfrak$ is symmetric about the origin (cf.\
        Remark  \ref{rem:spectrum}), so its
        contribution to the $\eta$-function vanishes.
    Now $\deg\hfrak_E=0$, since the curvature form is traceless. Noting that $ 2\rank
        \nfrak_E^++\rank\hfrak_E=\dim\sG=n^2-1$, and using \eqref{eqn:direct-sum}, we have
    \begin{align}
        \begin{split}\label{eqn:predimension}
            i(T_\beta^+)+i(T_\varphi^{+})
                &= (g-1+d/2)\dim\sG-\frac{1}{2}\sum_{p\in D}\rank(\hfrak_E\cap
            \lfrak_p) \\
                &= (g-1)\dim\sG+\frac{1}{2}\sum_{p\in D}(\dim \sG-\rank(\hfrak_E\cap
            \lfrak_p)) \\
                &=(g-1)\dim\sG+\frac{1}{2}\sum_{p\in D}(
            \dim(\sG/\sL_p)+2\rank(\nfrak_E^+\cap\lfrak_p)) \\
            &=
                (g-1)\dim\sG+\sum_{p\in D}(
            \dim(\sG/\sP_p)+\rank(\nfrak_E^+\cap\lfrak_p) )
        \end{split}
    \end{align} 
        This proves \eqref{eqn:BB-spar}. For the parabolic case, we again let $\widetilde T_\varphi^+$ be the
        operator $T_\varphi^+$, but with domain $\Dscr_\delta((\lfrak_E\oplus \nfrak_E)\otimes K_X)$. 
    As in the proof of Proposition \ref{prop:index}, 
        $$i(\widetilde T_\varphi^+)=i(T_{\varphi}^{+})+\sum_{p\in
    D}\rank((\nfrak_E^+\oplus\hfrak_E)\cap\lfrak_p)$$ 
        Now \eqref{eqn:BB-par} follows from \eqref{eqn:BB-spar} and \eqref{eqn:direct-sum}.
        However, note that since we assume full flags in the parabolic case, in
    fact $\nfrak^+_E\cap \lfrak_p=\{0\}$ and
        $\lfrak_p\subset\hfrak_E$, so the equality is automatic.
\end{proof}

\subsubsection{The Bia{\l}ynicki-Birula  slice} \label{sec:BB}
Following \cite[Def.\ 3.7]{CollierWentworth:19}, we define

\begin{Definition}[{\sc BB slice}] \label{def:bb-slice}
    Let $(A,\Phi)\in \Bcal^{spar,s}_\delta$ be a smooth Hodge pair in good
gauge.
    The $\BB$-slice at $(A,\Phi)$ in the strongly parabolic case is defined by
    \begin{align*}
    \Scal_\delta^{spar,+}(A,\Phi)=\bigl\{
        (\beta,\varphi)\in L^2_{1,\delta}(\nfrak_E^+\otimes\overline K_X)&\oplus
        L^2_{1,\delta}((\lfrak_E\oplus\nfrak_E^+)\otimes K_X) \mid \\
        d_2(\beta,\varphi)&+[\beta,\varphi]=0\ ,\ d_1^{\ast_\delta}(\beta,\varphi)=0
        \bigr\}
    \end{align*}
    In case the system of weights $\{\alpha(p)\}_{p\in D}$ corresponds to
    full flags, we define the $\BB$-slice in the parabolic case to be
    \begin{align*}
    \Scal_\delta^{par,+}(A,\Phi)=\bigl\{
        (\beta,\varphi)\in L^2_{1,\delta}(\nfrak_E^+\otimes\overline K_X)&\oplus
        \Dscr_\delta((\lfrak_E\oplus\nfrak_E^+)\otimes K_X) \mid \\
        d_2(\beta,\varphi)&+[\beta,\varphi]=0\ ,\ d_1^{\ast_\delta}(\beta,\varphi)=0
        \bigr\}
    \end{align*}
\end{Definition}

We have the analog of \cite[Prop.\ 4.2]{CollierWentworth:19}. 
\begin{Proposition} \label{prop:good-gauge}
    Let $(A,\Phi)\in \Bcal^{spar,s}_\delta$ be a smooth Hodge pair 
For each $$(\beta,\varphi)\in L^2_{1,\delta}(\nfrak_E^+\otimes\overline K_X)$$ satisfing
$d_2(\beta,\varphi)+[\beta,\varphi]=0$, there is a unique $f\in
    L^2_{2,\delta}(\nfrak_E^+)$ such that the complex gauge transformation
    $g=\Ibold+f\in \Gcal_{\delta,\ast}$
    takes $(\dbar_E+\beta,\Phi+\varphi)$ into the slice $\Scal^{par,+}_\delta(A,\Phi)$.
\end{Proposition}

The proof follows exactly as in the reference above. The key step is to use
the invertibility of $d_1^{\ast_\delta}d_1$ on $L^2_{2,\delta}$. This
follows from Lemma \ref{lem:simple}. We omit the details.  

\begin{Example} \label{ex:BB-residue}
    Consider a rank $2$ Hodge bundle $(A_0,\Phi_0)$ and an element
    $(\beta,\varphi)$ in the \BB-slice. 
    On $C(p)$ we write:

\begin{equation}
        \label{rank 2 slice}\dbar_A=
\left(\begin{matrix}
    \dbar_{L_1}& b\\ 0& \dbar_{L_2}
\end{matrix}\right)
\ ,\ \Phi=\left(\begin{matrix} \varphi_1&\varphi_2\\ \varphi_0&-\varphi_1
\end{matrix}\right)\otimes \frac{dz}{z}
\end{equation}

    Note that while $\varphi_0$ is holomorphic and possibly nonzero at
    $z=0$ (depending upon the relation between the Hodge splitting and the
    parabolic structure),
$\varphi_1$ is not holomorphic unless $b=0$ or $\varphi_0=0$.
Similarly, $\varphi_2$ is not holomorphic unless $b=0$ or $\varphi_1=0$.
Nevertheless,  we claim that
the limits $\displaystyle \lim_{z\to 0}\varphi_i$ exist for $i=1,2$.  Moreover,
identifying $\lfrak/W$ with $\CBbb/\pm1$: 
$$
\res_p([\dbar_A,\Phi])= \lim_{z\to 0}\varphi_1(z)
$$

To prove the claim,  we first bring $\dbar_A$ into the standard form
    $\dbar_{A_0}$ (modulo permuting the factors). See Example \ref{ex:phi}.
This is done by a based gauge transformation
$$
    g=\left(\begin{matrix} 1&-u\\ 0&1\end{matrix}\right)
$$
        where $g^{-1}\circ\dbar_A\circ g=\dbar_{A_0}$, and  $u\in
        L^2_{2,\delta}$ is such that $\dbar u=b$ on $C(p)$.
        Then
        $$
    g^{-1}\Phi g = \left(\begin{matrix} \varphi_1+u\varphi_0
        &\varphi_2-u(2\varphi_1+u\varphi_0)
\\
        \varphi_0&-(\varphi_1+u\varphi_0)
\end{matrix}\right)\otimes \frac{dz}{z}
        $$
        Notice that $\varphi_1+u\varphi_0$ is holomorphic. But we have a bound
        $|u(z)|\leq C|z|^\delta$, and $\varphi_0$ is holomorphic at $z=0$,
        hence
$$
    \text{res}_{\lfrak}([\dbar_A,\Phi])=\lim_{z\to
    0}(\varphi_1+u\varphi_0)(z)= \lim_{z\to 0}\varphi_1(z)
$$
Similarly, $\varphi_2-u(2\varphi_1+u\varphi_0)$ is holomorphic, and so
$$
    \lim_{z\to 0}\left(\varphi_2-u(2\varphi_1+u\varphi_0)\right)
    =\lim_{z\to 0} \varphi_2
$$
exists.
\end{Example}

Let $(A,\Phi)\in \Bcal^{spar,s}_\delta$ be a smooth Hodge pair  as above. 
We define the stable manifold of $(A,\Phi)$ for the action of $\CBbb^\ast$ on $\Mbold^{par,s}_{\Dol}$. 
\begin{equation} \label{eqn:W0}
    \Wcal_0(A,\Phi)=\left\{ [(\widetilde A,\widetilde\Phi)]\in
    \Mbold^{par,s}_{\Dol}\mid \lim_{t\to 0}[(\widetilde
    A,t\widetilde\Phi)]=[(A,\Phi)]\right\}
\end{equation}

Like Proposition \ref{prop:good-gauge} above, 
the proof of \cite[Cor.\ 4.3]{CollierWentworth:19}  can be adapted to give the
following result and proof of Theorem \ref{thm:coisotropic}.
\begin{Proposition}\label{prop:pH}
Let $(A,\Phi)\in \Bcal_{\delta}^{par,s}$ be a smooth, stable, 
Hodge pair. 
Then the  map
$$
    p_{\H} : 
\Scal_{\delta}^{par,+}(A,\Phi)\lra \Mbold_{\Dol}^{par,s}(\alpha,\delta)
: (\beta,\varphi)\mapsto [(\dbar_{A}+\beta,\Phi+\varphi)]
$$
is a biholomorphism onto $\Wcal_0(A,\Phi)$.  Moreover, $\Wcal_0(A,\Phi)$ is
coisotropic with respect to the Poisson structure on
$\Mbold_{\Dol}^{par,s}(\alpha,\delta)$, and
for any $\ell\in
    \bigoplus_{p\in D}\lfrak_p$, $\Wcal_0(A,\Phi)\cap \mathrm{Res}^{-1}(\ell)$
    is a holomorphic Lagrangian with respect to $\Omega_\ell$. 
\end{Proposition}

\begin{Remark} \label{rem:coisotropic}
Note that the image of
$ \Scal_{\delta}^{par,+}(A,\Phi) $ in the framed moduli space
$\Mbold_{\Dol,\ast}^{par,s}(\alpha,\delta)$ is a holomorphic Lagrangian; hence
the quotient in $\Mbold_{\Dol}^{par,s}(\alpha,\delta)$
is coisotropic.
\end{Remark}

\subsubsection{Relation with Simpson's partial stratification}

We define $\Wcal_1(A,\Phi)$ to be the set of all $[D]\in
\Mbold_{\DR,\ast}^{par,s}(\alpha,\delta)$ satisfying the following
condition. Let $(\Ecal(\alpha),\nabla)$ be a  holomorphic
bundle with logarithmic connection $\nabla$ associated to $[D]$ under the
identification from Theorem \ref{thm:analytic-DRmoduli}. 
By Proposition \ref{prop:simpson-griffiths} we know there is 
an associated Hodge bundle $(Gr_\Acal(\Ecal)(\alpha),\Phi_\Acal)$. Then we say
$[D]\in \Wcal_1(A,\Phi)$ if the isomorphism class of Higgs pairs associated
to $(Gr_\Acal(\Ecal)(\alpha),\Phi_\Acal)$
via Theorem \ref{thm:analytic-moduli} coincides with $[(A,\Phi)]$.

Recall the map $p_{\DR}$ from \eqref{eqn:p-dR}. 

\begin{Proposition}\label{prop:dR}
    Let $(A,\Phi)\in \Bcal_{\delta}^{par,s}$ be a smooth, stable, 
Hodge pair. Then the map 
    $$
    p_{\DR} : \Scal_\delta^{par,+}(A,\Phi)\lra \Mbold_{\DR,\ast}^{par,s}(\alpha,\delta)  
    $$
    is a biholomorphism onto $\Wcal_1(A,\Phi)$.
\end{Proposition}
This is \cite[Cor.\ 4.9]{CollierWentworth:19}. The proof there relies on
the following:

\begin{Lemma}
    Let $(A,\Phi)$ be as above, and let $D$ be the associated flat
    connection.
    Suppose that 
    $$(\beta,\varphi)\in L^2_{1,\delta}(\nfrak_E^+\otimes \overline K)\oplus
    \Dscr_\delta((\lfrak_E\oplus \nfrak_E^+)\otimes K_X)
    $$
    is such that $D+\beta+\varphi$ is flat.
    Then there exists a
    unique  smooth section $f\in L^2_{2,\delta}(\nfrak_E^+)$ such that if 
     $g=1+f\in \Gcal_{\delta,\ast}$, and
    $$
    g(D+\beta+\varphi) = D+\widetilde\beta+\widetilde\varphi
    $$
    then $(\widetilde\beta,\widetilde\varphi)\in
    \Scal_\delta^{par,+}(A,\Phi)$. 
\end{Lemma}

\begin{proof}
    The proof follows the recursive argument in 
    \cite[Prop.\ 3.11]{CollierWentworth:19}. The key step is to 
    invert the Laplacian $D'_\delta D''$ (see \cite[eq.\ (3.9)]{CollierWentworth:19}).
 By the
    assumption of stability of the Hodge bundle, the kernel, and hence also
    cokernel, of this operator vanishes (see \ref{lem:simple}).  Hence,
    the same proof as in that reference applies here as well. 
\end{proof}

\subsection{Existence of conformal limits}

In this section we prove the main result on the existence of a conformal
limit.  Fix a smooth stable Hodge pair $(A,\Phi)\in
\Bcal^{par}_{\delta}$ in good gauge. 
We assume, without loss of generality, that the hermitian metric $h$ on $E$ satisfies the
Hitchin equations. 
Let $D=\nabla^0_\delta\otimes D_\delta$ denote the corresponding flat
connection. 
To keep with the notation of the references, for $\ubold=(\beta,\varphi)\in
\Scal^{par,+}_\delta(A,\Phi)$, let  $\dbar_{\ubold}=\dbar_A+\beta$,
$\Phi_\ubold=\Phi+\varphi$.
Fix $R>0$. 
Let $h(\ubold,R)$ denote the harmonic metric
for the Higgs pair $(\dbar_\ubold,R\Phi_\ubold)$.
Consider the flat connection 
\begin{equation} \label{eqn:Gaiotto-connection}
    D_{(\ubold,R)}=\Phi_\ubold+\dbar_\ubold +\partial_\ubold^{h(\ubold,R)}+
    R^2 \Phi_\ubold^{\ast_{h(\ubold,R)}}\ .
\end{equation}

Notice that if $\ubold\in \Scal^{spar,+}_\delta(A,\Phi)$, then
$D_{(\ubold,R)}$ is a family in $\Fcal^{spar}_{\delta}$. 
However, if $\ubold\in \Scal^{par,+}_\delta(A,\Phi)$, it is not always true
that $D_{(\ubold,R)}$ lies in $\Fcal^{par}_{\delta}$. 
Nevertheless, we prove the following.

\begin{Proposition}[{\sc Conformal limit}]\label{prop:conformallimit}
    $\displaystyle \lim_{R\to 0} D_{(\ubold,R)} = \nabla_\delta^0\otimes
    (D_\delta+\beta+\varphi)$.
    Here, the convergence is $C^\infty$ on compact subsets of $X^\times$. 
\end{Proposition}

\begin{proof}
    Given the set-up of the previous sections, the proof is formally the same that of
    \cite{GaiottoLimitsOPERS} (see especially,  \cite[pp.\ 1223-4]{CollierWentworth:19}).
    As in these references, the key is to find an expansion in $R$ for the
    metric $h(\ubold, R)$. We first modify $h$ to get a metric $h'_R$ by an
    action of 
    $$
    g=\left(\begin{matrix} R^{m_1/2} &
    &0\\&\ddots&\\0&&R^{m_\ell/2}\end{matrix}\right)
    $$
    where $m_j-m_{j+1}=2$, $j=1,\ldots, \ell-1$, and $\sum_{j=1}^\ell
    (\rank E_j)m_j=0$. Notice that since $d_{A_0}g$ is compactly supported,
    in particular we have $g\in \Rcal^0_{\delta}$, and hence $g$ 
    is an element of the gauge group $\Gcal_{\delta}$.  
    Then as in \cite[eq.\ (5.4)]{CollierWentworth:19},
    $$
    \lim_{R\to 0}\left\{ \Phi_\ubold+\dbar_\ubold +
    e^{-\tau\delta}\partial_{\ubold}^{h'_R}e^{\tau\delta}+R^2\Phi_\ubold^{\ast_{h_R'}}\right\}=D_\delta+\beta+\varphi\
    .
    $$
    The goal now is express $h(\ubold,R)=g_R(h_R')$, $g_R= \exp(f_R)\in
    \Gcal_\delta$, where $f_R$ is a
    traceless $h_R'$-hermitian endomorphism. 
    The Hitchin equations 
    $$N_{(\ubold,R)}(f_R):=i\Lambda(F_{(\dbar_u,h(\ubold,R)}+[\Phi_\ubold,
    \Phi_\ubold^{\ast_{h(\ubold,R)}}])=0$$
    are now a function of $f_R$. As in \cite[p.\ 1224]{CollierWentworth:19}, we may view $N_{(\ubold,R)}$ as a map
    $$
    N_{(\ubold,R)} :  \Rcal_{\delta}(\hfrak_E\oplus \nfrak^+_E)\lra
L^2_{\delta}(\hfrak_E\oplus \nfrak^+_E)
    $$
    The linearization at $R=0$, $f_R=0$,  may be computed:
    $$
\frac{1}{2}dN_{(\ubold,0)}(0)\dt{f}=\dbar_{\ubold}\partial_E\dt{f}
+ [\Phi_\ubold, [\Phi^\ast,\dt{f}]]\ .
    $$
    As an operator on the $\dt f$  variable, 
    it follows as in the proof of Lemma \ref{lem:smooth}
    and Proposition \ref{prop:nude-index} that $dN_{(\ubold,0)}(0)$
    is Fredholm of index $0$; hence, surjectivity follows from injectivity. 
    With respect to the splitting of the bundle,
    decompose $\dt{f}=\dt{f}_{\hfrak}+\dt{f}_+$. 
    Following the proof of \cite[Lemma 5.2]{CollierWentworth:19}, an element of
     of $\ker dN_{(\ubold,0)}(0)$ would satisfy
     $
    (D'')^\ast D''\dt{f}_{\hfrak}=0
     $.
     Since we assume the Hodge pair is stable, by Lemma \ref{lem:simple} and
Proposition \ref{prop:nude-index}, the operator $(D'')^\ast D''$
     has no kernel, and so $\dt{f}_{\hfrak}=0$. Repeating this argument for the
     upper trianguler components $\dt{f}_+$ then shows that $\ker dN_{(\ubold,0)}(0)=\{0\}$. 
     The implicit function theorem can then be applied to $N_{(\ubold,R)}$
     to find the solution $f_R$ for small $R$. The rest of the proof
     follows as in the references cited above.
\end{proof}

\subsection{Conformal limit in Rank 2}\label{sec conf rk 2}
Recall from Example \ref{C* fixed points in rank 2}
 that there are two types of fixed points in rank two. Namely, $\cE(\alpha)$ is a stable parabolic bundle with $\Phi=0$ or 
\begin{equation}
        \label{eq general rk 2 fixed}(\cE(\alpha),\Phi)\cong(\cL_1(\beta_1)\oplus\cL_2(\beta_2),\begin{pmatrix}
        0&0\\\phi_0&0
\end{pmatrix})~,
\end{equation}
where $\cL_i(\beta_i)$ are parabolic line bundles and $\phi_0:\cL_1\to\cL_2\otimes K(D)$ is not zero and $\res_p(\phi_0)=0$ whenever $\beta_1(p)>\beta_2(p)$. 

In the case when $\cE(\alpha)$ is stable and $\Phi=0$, the 
the \BB-slice consists of all parabolic Higgs bundles $(\cE(\alpha),\varphi)$ with underlying bundle $\cE(\alpha).$ The $\hbar$-conformal limit in this case is 
\begin{equation}
        \label{eq stable bundle cl}\CL_\hbar(\bar\partial_E,0+\varphi)=(\hbar, \bar\partial_E,\hbar\partial_h+\varphi),
\end{equation}
 where $\bar\partial_E+\partial_h$ is the flat unitary logarithmic connection associated to the stable bundle $\cE(\alpha).$

For the fixed points with $\Phi\neq0$, let $h$ be the harmonic metric and $D=\bar\partial_E+\Phi^{*_h} +\partial_E^h+\Phi$  be the associated parabolic logarithmic connection. The splitting $\cL_1\oplus\cL_2$ is orthogonal with respect to $h$, so $h=h_1\oplus h_2.$
As in Example \ref{ex:BB-residue}, points in the \BB-slice through $(\bar\partial_E, \Phi)$ have the form $(\bar\partial_E+\beta,\Phi+ \varphi)$, where
\begin{equation}
         \label{eq sliced elements rank 2}\beta = \begin{pmatrix} 0 & b \\ 0 & 0 \end{pmatrix}, \quad \varphi=\begin{pmatrix}  \varphi_1 & \varphi_2 \\ 0 & - \varphi_1 \end{pmatrix}.
 \end{equation} 
Being in slice means 
$\bar\partial_E(\varphi)+[\beta,\Phi]+[\beta,\varphi]=0$ and
$e^{-\tau\delta}\partial_E^h(e^{\tau\delta}\beta)+[\Phi^{*_h},\varphi]=0.$
Specifically,
\begin{equation}
\label{eq slice eq}\xymatrix{\bar\partial_E\varphi_1+b\wedge\phi_0=0,&
\bar\partial_E
\varphi_2-2b\wedge\varphi_1=0&\text{and}&e^{-\tau\delta}\partial_E^h(
e^{\tau \delta}b) - 2 \phi_0 ^ { * _h } \wedge \varphi_1=0}.
\end{equation}
For such Higgs bundles, the $\hbar$-conformal limit is $(\hbar, \bar\partial_E+\hbar\Phi^{*_h}+\beta,\hbar\partial^h_E+\Phi+\varphi)$. More explicitly,
\begin{equation}
        \label{eq confomral limit rk 2}\CL_\hbar(\bar\partial_E+\beta,\Phi+\varphi)=\left(\hbar,\begin{pmatrix}
        \bar\partial_{1}&\phi_0^{*_h}+b\\0&\bar\partial_2
\end{pmatrix}, \begin{pmatrix}
        \hbar\partial^h_1+\varphi_1&\varphi_2\\\phi_0&\hbar\partial^h_2-\varphi_1
\end{pmatrix}\right).
\end{equation}

\begin{Remark}\label{Remark comple mass implies dol deformation}
As noted in Example \ref{ex:BB-residue}, the residue of the associated Higgs field at $p\in D$ is given by $\lim_{z\to p}\varphi_1$. In particular, if $\mathrm{Res}(\Phi)\neq 0$, then $\varphi_1\neq 0$, and the slice equations \eqref{eq slice eq} imply $b\neq0$. Hence, the loci of the slice with $b=0$ is in the strongly parabolic moduli space. Note also that $b=0$ if and only if $\varphi_2$ is holomorphic.
\end{Remark}


\section{A detailed study of the four-punctured sphere} \label{sec:example}
\label{sec: 4 punctured sphere}
In this section we discuss our main results in the special case of rank $2$
on the four-punctured sphere where things are relatively explicit. Many
aspects of the Higgs bundle moduli space in this special case were studied
by Fredrickson--Mazzeo--Swoboda--Wei{\ss} in \cite{FMSWfourpunctured}. 
We focus on the fixed points which are analogous to the uniformizing points in the unpunctured case. In particular, the harmonic metric of these fixed points is related to a metric on $\mathbb{CP}^1$ with constant negative curvature and conical singularities. Furthermore, in these cases the intersection of the \BB-slice with the hyperkahler moduli spaces with fixed complex masses parameterize sections of the Hitchin map, and the conformal limit of such objects share many features with the set of opers. 

\subsection{4-punctured sphere moduli space}\label{sec: fixeddata} Consider the
Riemann surface $X=\C\P^1$ and fix an effective divisor  $D=p_1+ p_2+p_3+p_4$.
Without loss of generality we assume $p_i\neq \infty$ for all $i.$ Throughout
this section, all local computations are done in the affine chart
$\C\P^1\setminus\{\infty\}$.

Fix a rank $2$ vector bundle $E \rightarrow \CP^1$ with degree $-4$. 
For each $p_i\in D$, fix a weighted filtration
\begin{align*}
 E_{p_i} &\supset \; \; \;F_{p_i}\; \; \;\supset 0\\ 
 0 < \alpha(p_i) &< 1-\alpha(p_i)<1,
\end{align*}
where $\alpha(p_i)\in(0,\frac{1}{2}).$ For notational convenience, we write $\alpha(p_i)=\alpha_i$ and $F_{p_i}=F_i$. The parabolic weights are determined by the vector $\alpha=(\alpha_1, \alpha_2, \alpha_3, \alpha_4) \in (0, \frac{1}{2})^{4}$.
This puts us in the case of full flags and trivial determinant since, 
given a parabolic bundle $\cE(\alpha)$ with this data, equation \eqref{eq:tensorofpar} gives an isomorphism $\det(\cE(\alpha)) \cong \cO(-4)(D) \cong \cO(0)$. 
We assume that the parabolic weights are generic, so that semistability implies stability.
With this fixed data, we define the moduli space $\cP_0(\alpha)$. It is a six dimensional complex manifold with Hitchin map $-\det:\cP_0(\alpha)\to \cB=H^0(K^2(2D))$. 

The residue map $\mathrm{Res}: \cP_0(\alpha) \to \oplus_{p \in D} \mathfrak{l}_p$ (see Remark \ref{rem:Resfullflag}) is determined by the vector $\mu=(\mu_1, \mu_2, \mu_3, \mu_4)$, where 
$\mu_i \in \C$ is the complex mass at $p_i \in D$.
Here, the subspace $F_{i}$ is the eigenspace of the residue $\mathrm{res}_{p_i}\Phi$ with eigenvalue $\mu_i$. 
Denote the fiber by
\[\cP_0(\alpha,\mu)=\mathrm{Res}^{-1}(\mu)\subset\cP_0(\alpha).\] For each $\mu$, $\cP_0(\alpha,\mu)$ is a smooth hyperk\"ahler manifold of complex dimension two. Let $\cB(\mu)$ denote the image of $\cP_0(\alpha,\mu)$ under the Hitchin map
\[-\det:\xymatrix@R=0em{\cP_0(\alpha,\mu)\ar[r]& \cB(\mu)\subset H^0(K^2(2D))\\[\cE(\alpha),\Phi]\ar@{|->}[r]&-\det(\Phi)}.\] 
The space $\cB(\mu)$ consists of all
elements of $H^0(K^2(2D))$ whose
residue\footnote{in the sense of \cite[eq.\ 1.2.19]{Tjurin:78}} at $p_i\in D$
is $\mu_i^2$. In particular, $\cB(\mu)$ is affine over $H^0(K^2(D))\cong \C$.

Since we are in the case of the full flags,  the strongly parabolic moduli space $\cS\cP_0(\alpha)$ corresponds to setting all complex masses $\mu_i=0$. Moreover, all $\C^*$-fixed points in $\cP_0(\alpha)$ lie in $\cS\cP_0(\alpha)$. In this case, the base $\cB(\alpha,0)$ is given by 
\[\cB(\alpha,0)\cong H^0(K^2(D))\cong H^0(\cO).\]
The nilpotent cone, i.e., the zero fiber of $\cS\cP_0(\alpha)\to \cB(\alpha,0)$, is the unique singular fiber. 

For generic parabolic weights
$\alpha$, the nilpotent cone consists of five spheres in an affine $D_4$
arrangement. There are five connected components of $\C^*$-fixed points, these are shown in red in Figure \ref{fig: nilpotentcone}. In particular, there are four isolated fixed points which determine the `exterior spheres' and one component of fixed points isomorphic to $\CP^1$ which is the `interior sphere.' 
\begin{figure}[h]
          \begin{tikzpicture}
  \draw[thick, red] (0,0) circle (1cm);
  \draw[red] plot[domain=pi:2*pi] ({cos(\x r)},{.2*sin(\x r)});
\draw[dashed,red] plot[domain=0:pi] ({cos(\x r)},{.2*sin(\x r)});
  \draw[thick,black] (1.06,1.06) circle (.5cm);
  \fill[fill=red] (1.42,1.42) circle (2pt);
  \draw[black] plot[domain=pi:2*pi] ({1.07+.3536*cos(\x r)+.0707*sin(\x r)},{1.07+.0707*sin(\x r)-.3536*cos(\x r)});
  \draw[black,dashed] plot[domain=0:pi] ({1.07+.3536*cos(\x r)+.0707*sin(\x r)},{1.07+.0707*sin(\x r)-.3536*cos(\x r)});
  \draw[thick,black] (-1.06,1.06) circle (.5cm);
  \fill[fill=red] (-1.42,1.42) circle (2pt);
  \draw[black] plot[domain=pi:2*pi] ({-1.07+.3536*cos(\x r)-.0707*sin(\x r)},{1.07+.0707*sin(\x r)+.3536*cos(\x r)});
  \draw[black,dashed] plot[domain=0:pi] ({-1.07+.3536*cos(\x r)-.0707*sin(\x r)},{1.07+.0707*sin(\x r)+.3536*cos(\x r)});
  \draw[thick,black] (-1.06,-1.06) circle (.5cm);
  \fill[fill=red] (-1.42,-1.42) circle (2pt);
  \draw[black] plot[domain=pi:2*pi] ({-1.07+.3536*cos(\x r)+.0707*sin(\x r)},{-1.07+.0707*sin(\x r)-.3536*cos(\x r)});
  \draw[black,dashed] plot[domain=0:pi] ({-1.07+.3536*cos(\x r)+.0707*sin(\x r)},{-1.07+.0707*sin(\x r)-.3536*cos(\x r)});
  \draw[thick,black] (1.06,-1.06) circle (.5cm);
  \fill[fill=red] (1.42,-1.42) circle (2pt);
  \draw[black] plot[domain=pi:2*pi] ({+1.07+.3536*cos(\x r)-.0707*sin(\x r)},{-1.07+.0707*sin(\x r)+.3536*cos(\x r)});
  \draw[black,dashed] plot[domain=0:pi] ({+1.07+.3536*cos(\x r)-.0707*sin(\x r)},{-1.07+.0707*sin(\x r)+.3536*cos(\x r)});
\end{tikzpicture}
\caption{\label{fig: nilpotentcone} The nilpotent cone of the strongly parabolic Hitchin moduli space on $(\CP^1, D)$. The $\C^*$-fixed points are shown in red.}
\end{figure}
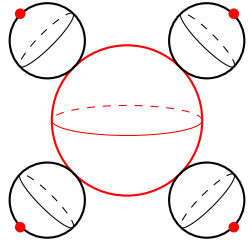

Consequently, the \BB-stratification of the moduli space $\cP(\alpha)$ from \eqref{eq bb strat} is 
 \begin{equation}
         \label{eq strata 4 punct}\cP(\alpha)=\coprod_{a\in
\pi_0(\cP_0(\alpha)^{\C^*})} \cW^a=\cW^{cent}\cup
\cW^{ext},
 \end{equation}
 where $\cW^{cent}$ is the stratum labeled by the central sphere in
Figure \ref{fig: nilpotentcone}, and $\cW^{ext}$ consists of the four
strata labeled by the exterior fixed points.
The volumes of the spheres with respect to the hyperk\"ahler metric on $\cS\cP_0(\alpha)$ vary with the parabolic weights $\alpha\in(0, \frac{1}{2})^4$, and degenerate to zero on certain 3-dimensional `walls' of $(0, \frac{1}{2})^4$ \cite{FMSWfourpunctured}. These walls divide the hypercube $(0, \frac{1}{2})^4$ into twenty-four chambers. As shown in \cite{Meneses,FMSWfourpunctured} and discussed below, the Higgs bundles corresponding to the $\C^*$-fixed points is chamber-dependent.

\subsection{$\C^*$-fixed points}\label{sec: fixedpoints4p} 

Consider a stable $\C^*$-fixed point of the form \eqref{eq general rk 2 fixed}.
Recall, for each $p\in D$, the subspace $F_p$ is either $\cL_1|_p$ or $\cL_2|_p$. Let $D_I$ and $D_{I^c}$ be the effective subdivisors of $D$ for which the subspace $F_p$ is $\cL_2|_p$ and $\cL_1|_p$, respectively. Thus, 
\[ D=D_I+D_{I_c}.\] 
Here $I$ denotes the subset of $\{1,2,3,4\}$ determined by the support of $D_I,$ and $I^c$ is its complement.   
The parabolic weights $\beta_1(p_i)$ and $\beta_2(p_i)$ are given by 
\begin{equation}
        \label{eq fixedpoint weights}\beta_1(p_i)=\begin{dcases}
        \alpha_i & i \in I \\ 1-\alpha_i & i \in I^c
\end{dcases}\ \ \ \ \ \ \ \text{and}\ \ \ \ \ \ \ \beta_2(p_i)=\begin{dcases}
       1-\alpha_i & i \in I\\
       \alpha_i&i\in I^c
\end{dcases}.
\end{equation}
The map $\phi_0$ in the Higgs field is a parabolic map $\cL_1(\beta_1)\to
\cL_2(\beta_2)\otimes K_X(D).$ By Definition \ref{def parabolic morph}, this
means that $\phi$ is a meromorphic section of $\cL_1^*\otimes\cL_2\otimes K_X$
with a worst simple poles at $p\in D$ and whose residue is zero whenever
$\beta_1(p_i)> \beta_2(p_i).$ Since $\beta_1(p_i)> \beta_2(p_i)$ implies $p_i\in
D_{I^c}$, we have
\[\phi_0\in H^0(\cL_1^* \cL_2 K(D_I)).\]

The $\C^*$-fixed point under consideration is stable whenever $\phi_0\in H^0(\cL_1^*\cL_2K(D_I))$ is nonzero and $\deg(\cL_2(\beta_2))<0$. Equivalently, we have 
\begin{equation}
        \label{eq stability of hodge}\vcenter{\xymatrix@=0em{\deg(\cL_2(\beta_2))=\deg(\cL_2)+\deg(D_I)-\sum_{i\in I}\alpha_i+\sum_{j\in I^c}\alpha_j<0\\ \deg(\cL_1^*\cL_2K(D_I))=2+2\deg(\cL_2)+\deg(D_I)\geq 0~.}}
\end{equation}
A straight forward computation shows that stability forces the degrees of $\cL_1$ and $\cL_2$ to be $-3,-2,-1$. Hence $\cL_2$ is isomorphic to $\cO(-1),\cO(-2)$ or $\cO(-3)$ and $\cL_1\cong\cO(-4)\cL_2^*.$ There are five cases determined by the degree of $D_I.$ The following table gives the conditions which are direct consequences of \eqref{eq stability of hodge}.

\begin{equation}
\label{eq table 4 puct}
        \begin{tabular}{|c|c|c|c|}\hline
                $\deg(D_I)$&$\cL_2$&condition on $(\alpha_1,\alpha_2,\alpha_3,\alpha_4)$& $\cL_1^* \cL_2 K(D_I)$\\\hline
                 $0$& $\cO(-1)$& $\sum_{i=1}^4\alpha_i<1$&$\cO$ \\\hline 
                 $1$ &$\cO(-1)$& $-\alpha_i+\sum_{j\in I^c}\alpha_j<0$, where $D_I=p_i$& \cO(1)\\\hline
                 $2$ &$\cO(-2)$&$-\sum_{i\in I}\alpha_i+\sum_{j\in I^c}\alpha_j<0$&$\cO$\\\hline
                 $3$&$\cO(-2)$& $\alpha_j-\sum_{i\in I}\alpha_i<-1$, where $D_{I^c}=p_j$&$\cO(1)$\\\hline
                 $4$&$\cO(-3)$& $\sum_{i=1}^4\alpha_i>1$&$\cO$\\\hline
        \end{tabular}
\end{equation}

The five cases are described succinctly in following proposition.
\begin{Proposition}\label{prop:generalsplit}
Consider a stable $\C^*$-fixed point $(\cE(\alpha),\Phi)=\left(\cL_1(\beta_1) \oplus \cL_2(\beta_2), \begin{pmatrix} 0 & 0 \\ \phi_0 & 0 \end{pmatrix}\right)$. Let $D_I\subset D$ be the effective subdivisor determined by the points $p\in D$ for which $F_p=\cL_2|_p$. Then, 
\begin{equation}
        \label{eq bundle for fixed}\cL_1\oplus\cL_2\cong \cO(d-3)\oplus \cO(-1 - d ),
\end{equation}
where $d=\left\lfloor \frac{\deg(D_I)}{2}\right\rfloor,$
and $\phi_0$ is a nonzero section of $\Hom(\cL_1,\cL_2)\otimes K(D_I)\cong\cO(\deg(D_I)-2 d)$.
\end{Proposition}

\subsubsection{The conformal limit for the central sphere}
The central sphere in Figure \ref{fig: nilpotentcone} either consists of stable parabolic bundles with zero Higgs field or corresponds to the case where $\deg(D_I)=1$ or $\deg(D_I)=3$ in Proposition \ref{prop:generalsplit}. Of the 24 chambers in the space of weights, 16 have stable parabolic bundles as their central sphere. From Table \eqref{eq table 4 puct}, these occur when $0 < - \alpha_k + \sum_{i \neq k} \alpha_i <1$ for all $k$.

In the remaining eight chambers, there are no stable parabolic bundles $\cE(\alpha)$. Instead, the central sphere consists of Higgs bundles described in Table \eqref{eq table 4 puct}; take $\deg(D_I)=1$ if there is a weight $\alpha_{k}$ with $-\alpha_{k}+\sum_{i\neq k}\alpha_i<0$ or $\deg(D_I)=3$ if there is a weight $\alpha_{k}$ with $\alpha_{k}-\sum_{i\neq k}\alpha_i>-1$.  
In both cases, the central sphere is parameterized by $\P(H^0(\cO(1))\setminus \{0\}).$
For $u \in \CP^1$, the bundle is $\cO(d-3)\oplus\cO(-1-d)$, where $d=\lfloor \frac{\deg(D_I)}{2}\rfloor$ and the Higgs field is determined by a nonzero section $\phi_u$ of $\cO(2-2d)K(D_I)\cong\cO(1)$ which vanishes at $u$ and is defined up to scale. 

In these cases, the \BB-slice is given by 
\[(\bar\partial_E,\Phi)=\left(\begin{pmatrix}
        \bar\partial_{d-3}&b\\0&\bar\partial_{-d-1}
\end{pmatrix},\begin{pmatrix}
        \varphi_1&\varphi_2\\\phi_0&-\varphi_{1}
\end{pmatrix}\right),\]
where $b$ and $\varphi_i$ satisfy \eqref{eq slice eq}. The $\hbar$-conformal limit of such a Higgs bundle is given by  \eqref{eq confomral limit rk 2}.
\begin{Remark}\label{remark b not zero in I odd}
        Note that for these fixed points, the term $b$ is always nonzero. Indeed, by Remark \ref{Remark comple mass implies dol deformation}, $b=0$ implies $\varphi_1=0$ and $\varphi_2$ is holomorphic. Hence, $\varphi_2$ is a holomorphic map 
        \[\varphi_2:\cO(-d-1)\to\cO(d-3)\otimes K(D_{I^c}).\] But $d+1+d-3+-2+\deg(D_{I^c})=-1$, so such a map must be zero. 
\end{Remark}

\subsection{Exterior fixed points: Hitchin sections and opers}  Each of the four exterior fixed points in Figure \ref{fig: nilpotentcone} is described by a case in Table \eqref{eq table 4 puct} where $\deg(D_I)$ is even. The condition on the parabolic weights  in \eqref{eq stability of hodge} implies that there is exactly one fixed point with $\deg(D_I)=0$ or $\deg(D_I)=4$, determined by the sign of $-1+\sum\alpha_i$, and three fixed points with $\deg(D_I)=2$, determined by the three effective subdivisors $D_I$ with $-\sum_{i\in I}\alpha_i+\sum_{j\in I^c}\alpha_j<0$.
In particular, each stratum in $\cW^{ext}$ from \eqref{eq strata 4 punct}
is labeled by the subdivisor $D_I$. Denote the corresponding stratum by
$\cW^{ext_I}.$ The $\C^*$-fixed point which labels the stratum
$\cW^{ext_I}$ is
\begin{equation}
        \label{eq fixedpoint}(\cE(\alpha),\Phi)=\left(\cO(d-3)\oplus\cO(-1-d), \begin{pmatrix}
        0&0\\\phi_0&0
\end{pmatrix}\right),
\end{equation}
where $d=\frac{\deg(D_I)}{2}$ and $\phi_0$ is a nowhere zero section of $\cO(2-2d)K(D_I)\cong\cO$, defined up to scale. 

In these cases, the \BB-slice parameterizes $\cW_0^{ext_I}$ and is given
by
\[(\bar\partial_E,\Phi)=\left(\begin{pmatrix}
        \bar\partial_{d-3}&b\\0&\bar\partial_{-d-1}
\end{pmatrix},\begin{pmatrix}
        \varphi_1&\varphi_2\\\phi_0&-\varphi_{1}
\end{pmatrix}\right),\]
where $b$ and $\varphi_i$ satisfy \eqref{eq slice eq}. 
Unlike the case when $\deg(D_I)$ is odd in Remark \ref{remark b not zero in I odd}, the $b=0$ locus of the slice is nontrivial when $\deg(D_I)$ is even. In fact, the $b=0$ locus of each slice defines a section 
\[s_I:H^0(K^2(D))\to \cS\cP_0(\alpha)\]of the Hitchin map for the strongly parabolic moduli space. Namely, for $q\in H^0(K^2(D))$ and $\phi_0:\cO(d-3)\to \cO(-d-1)\otimes K(D_I)$ nonzero, we have $q/\phi_0$ defines a holomorphic section of $\cO(2d-2)\otimes K(D_{I^c})$. For the exterior fixed point labeled by $D_I$, the Hitchin section $s_I$ is  
\[s_I(b)=(\cO(d-3)\oplus \cO(-d-1), \begin{pmatrix}
        0&\frac{q}{\phi_0}\\\phi_0&0
\end{pmatrix}).\] 
\begin{Remark}
        The difference between the $\deg(D_I)=0$ and $\deg(D_I)=4$ cases is which component of the Higgs field can vanish. 
        The bundle, flag structure and form of the Higgs field are the same.
\end{Remark}

The $\hbar$-conformal limit of such a Higgs bundle is given by 
\[\CL_\hbar\left(\begin{pmatrix}
        \bar\partial_{d-3}&b\\0&\bar\partial_{-d-1}
\end{pmatrix},\begin{pmatrix}
        \varphi_1&\varphi_2\\\phi_0&-\varphi_{1}
\end{pmatrix}\right)=(\hbar,\begin{pmatrix}
        \bar\partial_{1}&\phi_0^{*_h}+b\\0&\bar\partial_2
\end{pmatrix}, \begin{pmatrix}
        \hbar\partial^h_1+\varphi_1&\varphi_2\\\phi_0&\hbar\partial^h_2-\varphi_1
\end{pmatrix}).\]
Let $(\cV(\alpha),\nabla_\hbar)$ denote the corresponding parabolic bundle and parabolic logarithmic $\hbar$-connection of the conformal limit. From the form of the conformal limit, $\cV(\alpha)$ is an extension
\[\cO(d-3)\to \cV(\alpha)\to \cO(-1-d).\]
Moreover, the parabolic map 
\[\xymatrix{\cO(d-3)\ar[r]^{\nabla_\hbar}&\cV\otimes K(D)\ar[r]& \cO(-d-1)\otimes K(D)}\]
is given by $\phi_0$, and hence can only have simple poles at $D_I$. In particular, the parabolic logarithmic $\hbar$-connection induces an isomorphism $\cO(d-3)\cong\cO(-d-1)\otimes K(D_I)$.  
In the nonparabolic case, this is the property that defines the set of opers, and opers are exactly the the image of the Hitchin section under the conformal limit \cite{GaiottoLimitsOPERS}. 
As a result, such objects can be viewed as parabolic opers. 
\begin{Remark}
        Note that if we forget the parabolic structure and just consider the associated logarithmic $\hbar$-connection, then only the case $D_I=D$ satisfies the oper condition. Indeed, $\phi_0$ is an isomorphism between the subbundle and the quotient twisted by $K(D)$ only when $D=D_I.$ 

Similarly, forgetting the parabolic structure on the Higgs bundle defines a $K(D)$-twisted Higgs bundle on $\C\P^1$ with determinant $\cO(-4)$, i.e., $\Phi:\cE\to\cE\otimes K(D)$ is holomorphic. There is a natural notion of stability for such objects and a corresponding moduli space $\cM_{K(D)}$. Taking $-\det(\Phi)$ defines a Hitchin map $\cM_{K(D)}\to  H^0(K^2(2D))$, which has a natural section 
\begin{equation}\label{eq hitchin seciton K(D)twisted}
s_{K(D)}:\xymatrix{H^0(K^2(2D))\ar[r]& \Mm_{K(D)}}
\end{equation}
defined by $s_{K(D)}(q)=\left[\cO(-1)\oplus \cO(-3),\begin{pmatrix}
        0&q\\1&0
\end{pmatrix}\right]~.$
\end{Remark}

We prove that the intersection of each stratum $\cW_0^{ext_I}$ with the
fixed complex mass moduli spaces $\cP_0(\alpha,\mu)$ parameterizes a section of
the Hitchin map $-\det:\cP_0(\alpha,\mu)\to \cB(\mu)$.

\begin{Theorem}\label{Thm hitchin sections} 
Consider the stratum $\cW_0^{ext_I}$ of any exterior fixed point.  Then, for
each fixed complex mass vector $\mu=(\mu_1,\mu_2,\mu_3,\mu_4)$, there is a
section
\[s_{I}^\mu:\cB(\mu)\to\cP_0(\alpha,\mu)\] 
of the Hitchin map whose image is $\cW_0^{ext_I}\cap \cP_0(\alpha,\mu)$.
Furthermore,  the underlying $K(D)$-twisted Higgs bundle of $s_{I}^\mu(q)$ is
given by $s_{K(D)}(q)$ from \eqref{eq hitchin seciton K(D)twisted} when
$\deg(D_I)=4$.
\end{Theorem}

\begin{Remark}
The maps $s_I^\mu$ do not assemble to define section of the Hitchin map for the full moduli space $\cP_0(\alpha)$. This is because the base $\cB(\mu)$ only determines $\mu_i^2$ at each $p_i\in D$. Hence, fixing $\cB(\mu)$ and $D_I$ defines $2^k$ sections $s_I^\mu:\cB(\mu)\to \cP_0(\alpha)$, where $k$ is the number of $p_i\in D$ with $\mu_i\neq 0.$ 
Alternatively, let $\hat\cB$ be the image of the product map
\[(\mathrm{Res},-\det ):\cP_0(\alpha)\to\hat\cB\subset \C^4\times H^0(K^2(2D)),\] and $\pi:\hat\cB\to H^0(K^2(2D))$ be the natural projection map. Then, the sections $\pi^*s_I^\mu$ do assemble to define a section $S_I:\hat\cB\to\cP_0(\alpha)$. Namely, for $(\mu,q)\in\hat\cB\subset\C^4\times \cB$ we have 
\[S_I(\mu,q)=s_I^\mu(q).\]  
The space $\hat\cB$ is biholomorphic to $\C^5$ and $\pi:\hat\cB\to H^0(K^2(2D))$ is a $16$ to $1$ branched cover ramified at the points where $\mu_i=0$ for some $i.$ This perspective is discussed further in \cite{FMSWfourpunctured}. A similar subtlety motivates Bridgeland--Smith's moduli space of framed quadratic differentials \cite{BridgelandSmith}.
\end{Remark}

The parameterization of $s_I^\mu$ given in the proof is not by the \BB-slice.
However, since both objects parameterize $\cW_0^{ext_I}$ we have the following
corollary.
\begin{Corollary}
        For each exterior stratum $\cW_0^{ext_I}$ and
        each complex mass vector $\mu,$ the intersection of
        the \BB-slice with
$\cP_0(\alpha,\mu)$ parameterizes the section $s_I^\mu.$ Hence, the conformal
limit of the points in the image of the sections $s_I^\mu$ can be viewed as
parabolic opers.
\end{Corollary}
\begin{Remark} In \cite{GaiottoConj}, Gaiotto conjectures that in the
conformal limit, the image of ``a canonical complex Lagrangian submanifold which
is a section of the torus fibration''  is a complex Lagrangian submanifold of a
certain complex manifold. The above Corollary is in the spirit of Gaiotto's
conjecture, as indeed the image of the \BB-slice in $\cP_0(\alpha,\mu)$ is  a
holomorphic Lagrangian submanifold (Theorem \ref{thm:coisotropic}), and the
image of the section $s_I^\mu$ in the Corollary is a holomorphic Lagrangian
inside of $\cP_{\hbar}(\alpha, \hbar \alpha + \mu)$.  Note that
\cite{GaiottoConj} concerns $\PU(2)$-Hitchin moduli spaces rather than
$\SU(2)$-Hitchin moduli spaces; $\PU(2)$-Hitchin moduli spaces on the
four-punctured sphere have a single Hitchin section.
\end{Remark}

\begin{proof}[Proof of Theorem \ref{Thm hitchin sections}]
        There are three cases determined by $\deg(D_I)=0,2,4.$ First assume $\deg(D_I)=4$. Denote the restriction of the section $s$ from \eqref{eq hitchin seciton K(D)twisted} to $\cB(\mu)$ by $s_\mu^I(q)$.
       This $K(D)$-twisted Higgs bundle  admits a parabolic structure such that the parabolic Higgs bundle is stable and defines a point in $\cP_0(\alpha,\mu)$. 
        Indeed, define the flag $F_i\subset(\cO(-1)\oplus\cO(-3))|_{p_i}$ with weight $1-\alpha_i$  by 
        \[F_i=\text{the $\mu_i$-eigenspace of }
                \res_{p_i}\left(\begin{pmatrix}
                        0&q\\1&0
                \end{pmatrix}\right) .\]

       Now assume $\deg(D_I)=2$, without loss of generality assume $D_I=\{p_3,p_4\}$. 
       To define the section $s_\mu^I$, we first show that each $q\in \cB_\mu$ and $\phi\in H^0(K(D_I))$ determines unique sections $s_1\in H^0(K^2(D))$ and $s_2\in H^0(\cO(D_{I^C})\otimes \cO(p_4))$ such that  
       \begin{equation}\label{eq for quad diff}
               q=s_1^2+\phi s_2,
       \end{equation}       
    
       Since $s_1$ is uniquely determined by specifying its residue at three points, define $s_1$ by
       \[\xymatrix{\res_{p_1}(s_1)=\mu_1,&\res_{p_2}(s_1)=\mu_2&\text{and}&\res_{p_3}(s_1)=-\mu_3}.\]
       Hence, $\res_{p_4}(s_1)=\mu_3-\mu_1-\mu_2.$ For $s_1$ and $s_2$ to solve \eqref{eq for quad diff}, we must have
       \[\res_{p_4}(s_2)=-\frac{(\mu_3-\mu_2-\mu_1)^2+\mu_4^2}{\res_{p_4}(\phi)}~.\]
       For such an $s_2$ and any $q\in \cB(\mu),$ we have $q-(s_1^2+\phi s_2)\in H^0(K^2(D))$. Since $B(\mu)$ is affine over $H^0(K^2(D)),$ for each $q$ there is a unique $s_2$ with the above residue at $p_4$ which solves \eqref{eq for quad diff}. 

       For $q\in\cB(\mu)$ and $s_1,s_2$ as above, define the following $K(D)$-twisted Higgs bundle 
       \[s_I^\mu(q)=\left(\cO(-2)\oplus\cO(-2),\begin{pmatrix}
               s_1&s_2\\\phi&-s_1
       \end{pmatrix}\right).\]
       Denote the Higgs field by $\Phi,$ and note that  $-\det(\Phi)=q$ by \eqref{eq for quad diff}. The parabolic structure so that $s_I^\mu(q)\in \cP_0(\alpha,\mu)$ is defined as follows.
        By construction, we have
        \[\res_{p_1}(\Phi)=\begin{pmatrix}
                \mu_1&\res_{p_1}(s_2)\\0&-\mu_1
        \end{pmatrix} , \ \ \ \res_{p_2}(\Phi)=\begin{pmatrix}
                \mu_2&\res_{p_2}(s_2)\\0&-\mu_2
        \end{pmatrix},\ \ \ \res_{p_3}(\Phi)=\begin{pmatrix}
                -\mu_3&0\\\res_{p_3}(\phi)&\mu_3
        \end{pmatrix}.\]
        Thus, the flag $F_i\subset\cO(-2)\oplus \cO(-2)$ must be the first summand if $i=1,2$ and the second summand for $i=3.$ For $p_4$, we have 
        \[\res_{p_4}(\Phi)=\begin{pmatrix}
                \mu_3-\mu_1-\mu_2&\res_{p_4}(s_2)\\\res_{p_4}(\phi)&\mu_1+\mu_2-\mu_3
        \end{pmatrix},\]
        with eigenvalues $\pm\mu_4$. Hence, the subspace $F_4$ must be the $\mu_4$-eigenspace of $\res_{p_4}(\Phi)$. The parabolic weights on the subspaces $F_i$ are $1-\alpha_i$. 

        Finally, assume $\deg(D_I)=0$, a nonzero $\phi:\cO(-3)\to \cO(-1) K$ determines the associated fixed point. If $\sum\mu_i=0,$ then there is a unique section $t_1\in H^0(K(D))$ with $\res_{p_i}(t_1)=\mu_i$ for all $i$. Since $\phi$ has no poles, there is a unique $t_2\in H^0(K^2(D))$ with $q-t_1^2=\phi t_2.$ Define $s_I^\mu$ by
        \[s_I^\mu(q)=\left(\cO(-3)\oplus\cO(-1),\begin{pmatrix}
                t_1&t_2\\\phi&-t_1
        \end{pmatrix}\right).\]
        Since the residue at each $p_i$ is upper triangular, for all $i$, take $F_i=\cO(-3)|_{p_i}$ with weights $1-\alpha_i$.

        If $\sum\mu_i\neq0,$ then the section will be given by 
        \[s_I^\mu(q)=(\cO(-2)\oplus\cO(-2),\begin{pmatrix}
                t_1&t_2\\t_3&-t_1
        \end{pmatrix}),\]
        where $t_j\in H^0(K(D))$ for $j=1,2,3$ and $q=t_1^2+t_2t_3$. Each $t_j$ is determined by its residue at 3 points, and these residues are uniquely defined by the condition that all $F_i$ are contained in a fixed subbundle isomorphic to $\cO(-3)$ and $q\in \cB(\mu).$  Consider the holomorphic subbundle
        \[\cO(-3)\xrightarrow{(z-p_1,z-p_2)}\cO(-2)\oplus\cO(-2).\]
        For $i=1,2$, the condition that $\res_{p_i}(\Phi)$ has $\cO(-3)$ as its $\mu_i$-eigenspace implies
        \[\xymatrix{\res_{p_1}(t_1)=-\mu_1~,&\res_{p_1}(t_3)=0&\res_{p_2}(t_1)=\mu_2&\text{and}& \res_{p_2}(t_2)=0}.\]
        Moreover, the condition $q=t_1^2+t_2t_3$ fixes $\res_{p_1}(t_2)$ and $\res_{p_2}(t_3)$. 
        For $i=3$, $\res_{p_i}(\Phi)$ has the fixed $\cO(-3)$ as $\mu_3$-eigenspace implies 
        \[\xymatrix{\res_{p_3}(t_2)=(\mu_3-\res_{p_3}(t_1))\cdot\dfrac{p_3-p_1}{p_3-p_2}&\text{and}&\res_{p_3}(t_3)=(\mu_3+\res_{p_3}(t_1))\cdot\dfrac{p_3-p_2}{p_3-p_1}}~.\]
        Hence, specifying $\res_{p_3}(t_1)$ and $q=t_1^2+t_2t_3$ determines all $t_j.$ But the residue of $q$ and $t_1^2+t_2t_3$ are the same for any choice of $\res_{p_3}(t_1)$. Thus, $q=t_1^2+t_2t_3$ uniquely determines all $t_j.$

        It remains to show that the Higgs bundles described above are stable and
in the correct stratum $\cW_0^{ext_I}$. First suppose  $\deg(D_I)=4$. For
$\lambda\in\C^*$, consider the scaled Higgs field $\lambda\Phi$ of $s_I^\mu(q).$
The gauge transformation
$g_\lambda=\diag(\lambda^{\frac{1}{2}},\lambda^{-\frac{1}{2}})$ is holomorphic
and acts on $\lambda\Phi$ as
        \[g_\lambda(\lambda \Phi) g_\lambda^{-1}=\begin{pmatrix}0&\lambda^2 q\\1&0
        \end{pmatrix}~,\]
        Since $F_i\neq \cO(-1)|_{p_i}$ for all $p_i\in D$, we have $\lim_{\lambda\to 0}g_\lambda \cdot F_i=\cO(-3)|_{p_i}.$
        For $\lambda$ sufficiently small, this Higgs bundle is in an open
neighborhood of the $\C^*$ fixed point in $\cW_0^{ext_I}$. Since stability is
open, preserved by the $\C^*$ action and gauge transformation, $s_I(q)$ is
stable and in $\cW_0^{ext_I}.$

        The cases $\deg(D_I)=2$ and $\deg(D_I)=0$ with $\sum\mu_i=0$ are similar to the $\deg(D_I)=4$ case. 
        Finally, suppose $\deg(D_I)=0$ with $\sum\mu_i\neq0$. In this case, the holomorphic bundle is a nonsplit extension of $\cO(-1)$ by $\cO(-3).$ In a smooth splitting $\cO(-3)\oplus\cO(-1)$ the Higgs bundle is given by 
        \[(\bar\partial_E,\Phi)=\left(\begin{pmatrix}
                \bar\partial_{-3}&b\\0&\bar\partial_{-1}
        \end{pmatrix}, \begin{pmatrix}
                \varphi_1&\varphi_2\\\phi&-\varphi_1
        \end{pmatrix}\right),\]
        where $\phi:\cO(-3)\to\cO(-1)\otimes K(D)$ is nonzero and holomorphic. The gauge transformation $g_\lambda=\diag(\lambda^{\frac{1}{2}},\lambda^{-\frac{1}{2}})$ acts on $(\bar\partial_E,\lambda\Phi)$ as
        \[g_\lambda\cdot(\bar\partial_E,\lambda\Phi)=\left(\begin{pmatrix}
                \bar\partial_{-3}&\lambda b\\0&\bar\partial_{-1}
        \end{pmatrix}, \begin{pmatrix}
                \lambda \varphi_1&\lambda^2\varphi_2\\\phi&-\lambda\varphi_1
        \end{pmatrix}\right)~.\]
Moreover, $g_\lambda\cdot F_i=F_i$ since $F_i=\cO(-3)|_{p_i}$ for all $i.$ By
the same arguement as the $\deg(D_I)=4$ case, it follows that $s_I^\mu(q)$ is
stable and in $\cW_0^\mu$.
\end{proof}
\begin{Remark}
        Switching the roles of $p_3,p_4$ in the $\deg(D_I)=2$ case, or choosing a different $\cO(-3)$ subbundle in the $\deg(D_I)=0$ case with $\sum\mu_i\neq 0$ defines isomorphic parabolic Higgs bundles.
\end{Remark}

We end the paper by showing that, analogous to the Hitchin section in the nonparabolic case, the harmonic metric at an exterior fixed point comes from a constant negative curvature metric on $\C\P^1$ with conical singularities at each $p_i\in D$. 
As above, write the parabolic bundle associated to an exterior $\C^*$-fixed point \eqref{eq fixedpoint} as $\cL_1(\beta_1)\oplus\cL_2(\beta_2)$, where $\beta_1$ and $\beta_2$ are defined in \eqref{eq fixedpoint weights}. Then 
\[\cL_1(\beta_1)^2\cong \cO(-2)(\gamma)=K(\gamma),\ \ \ \ \text{where}  \ \ \ \ \gamma(p_i)=\begin{dcases}
        2\alpha_i&p_i\in D_I\\1-2\alpha_i&p_i\in D_{I^c} 
\end{dcases}.\]
Since $\deg(K(\gamma))=2\deg(\cL_1(\beta_1))>0$, the harmonic metric on $K^{-1}(-\gamma)$ is a constant negative curvature singular metric $g$ on the tangent bundle $K^{-1}$ which is smooth on $\CBbb\PBbb^1\setminus D$ and $|z|^{2\gamma}h$ extends as a smooth metric across $D$. The hermitian metric $g$ gives a Riemannian of constant negative curvature on $\CBbb\PBbb^1$ which has conical singularities at $D$ with 
\[\text{cone angle of $g$}=\begin{dcases}
        2\pi(1-2\alpha_i)& \text{if }p_i\in D_I\\
        4\pi\alpha_i& \text{if }p_i\in D_{I^c}
\end{dcases}.\]

Consider the square root $K(\gamma)^\frac{1}{2}=\cO(-1)(\frac{\gamma}{2})$, its dual is  given by 
\[K(\gamma)^{-\frac{1}{2}}=\cO(1)(-\frac{\gamma}{2})\cong(\cO(1)\otimes \cO_D^{-1})(1-\frac{\gamma}{2})\cong\cO(-3)(1-\frac{\gamma}{2}).\]
With this setup, the following lemma is immediate.
\begin{Lemma}
  Let $d=\frac{\deg(D_I)}{2}$ and $K(\gamma)^\frac{1}{2}$ be as above. Consider the parabolic line bundle $\cL_3(\beta_3)$, where $\cL_3=\cO(-2+d)$ and $\beta_3(p_i)=0$ if $p_i\in D_I$ and $\beta_3(p_i)=\frac{1}{2}$ if $p_i\in D_{I^c}.$ 
Then the parabolic bundle for the exterior fixed point from Proposition \ref{prop:generalsplit} is given by  
\[\cO(d-3)(\beta_1)\oplus\cO(-1-d)(\beta_2)\cong(\cL_3(\beta_3)\otimes K(\gamma)^\frac{1}{2})\oplus (\cL_3(\beta_3)^{-1}\otimes K(\gamma)^{-\frac{1}{2}}). \]  
\end{Lemma}

Denote the harmonic metric on $\cL_3(\beta_3)$ by $h_3$. Since $\deg(\cL_3(\beta_3))=0,$ $h_3$ is flat. Note that $g^\frac{1}{2}$ defines a metric on $K(\gamma)^{-\frac{1}{2}}(-\frac{\gamma}{2})$ and $g^{-\frac{1}{2}}$ defines a metric on $K(\gamma)^\frac{1}{2}$. The trivial metric on trivial parabolic bundle $\cO(0)$ defines a metric $h_{\det}$ on the weighted line bundle $\cO_D^{-1}(1)$, and $g^\frac{1}{2}h_{\det}$ is a compatible metric on the parabolic line bundle $\cO(-3)(1-\frac{\gamma}{2})$.

\begin{Proposition}With the notation above, the harmonic metric on the exterior $\C^*$-fixed point
\[\left(\cO(d-3)(\beta_1)\oplus\cO(-1-d)(\beta_2),\begin{pmatrix}
        0&0\\\phi&0
\end{pmatrix}\right)\]
is given by $h=(h_3\cdot g^{-\frac{1}{2}})\oplus (h_3^{-1}\cdot g^{\frac{1}{2}}\cdot h_{\det})$.
\end{Proposition}
\begin{proof}
         Since the metrics $h_{\det}$ and $h_3$ are flat, and the harmonic metric is diagonal at $\C^*$-fixed points, the metric $h$ solves the Hitchin equations $F_{A_h}+[\Phi,\Phi^{*_h}]=0$ if and only if 
        \[0=F_{A_g} + 2\phi \wedge \phi^*
        =F_{A_g} + 2gh_3^{-2}h_{\det} \phi \wedge \bar \phi.\]
The curvature of the Levi-Civita connection on the holomorphic tangent bundle is $-4 i(\frac{1}{4} K_{g} \omega_{g})$, where $\omega_{g}$ is the K\"ahler form and $K_g$ is the Gauss curvature of $g$ which is $-4$. To complete the proof, we show $2\phi\wedge\phi^*=-4i\omega_g$.  Note that if $g=\lambda^2 \; \de z \otimes \de \zbar$, then $\omega_g= \lambda^2 \; \frac{i}{2}  \de z \wedge \de \zbar$.
In the trivialization of $\cO(d-3)\oplus\cO(-1-d)$ given by $(\de z)^\frac{3-d}{2} \oplus (\de z)^\frac{1+d}{2}$, the Higgs field is given by $\phi = \de z^{d-1} \otimes \prod_{p_i\in D_I}(z-p_i)^{-1} \de z$ and its adjoint is 
 \[\phi^*=g h_3^{-2}h_{\det}\bar\phi=\underbrace{\lambda^2 \; \de z \otimes \de \bar z}_g \underbrace{\prod_{p_i\in D}|z-p_i|^{-2} |\de z|^{-2(d-2)}}_{h_3^{-2}( \de z^{d-2}, \de z^{d-2})} 
  \underbrace{\prod_{p_i\in D_{I^c}}|z-p_i|^{2}|\de z|^{-4}}_{h_{\mathrm{det}}(\de z^2, \de z^2)} \left(\de \bar z^{d-1}\otimes \prod_{p_i \in D_I}(\overline{z-p_i})^{-1} \de\bar z\right).\]
Hence $\phi\wedge\phi^*=\lambda^2 \; \de z \wedge \de \zbar =-2 i\omega_g.$
\end{proof}

\begin{appendix}
 \section{Simpson's stratification}\label{appendix Simpson}In this appendix we
show that Simpson's iterative process of \cite{SimpsonDeRhamStrata} generalizes
to parabolic logarithmic $\lambda$-connections. The main difference with
\cite{SimpsonDeRhamStrata} is that stable parabolic connections are not always
irreducible. When the parabolic logarithmic connections are irreducible, the
associated stratifications were studied in \cite{SimpsonDeRhamStrata}
and\cite{ClosureOfDeRhamStrata4P1}.
 We note that the discussion below also applies to Higgs bundles in both the parabolic and nonparabolic setting. For additional details on Simpson's construction we follow \cite[\S2.2.1]{PengfeiThesis}.

Let $(\cE(\alpha),\nabla)$ be a parabolic logarithmic $\lambda$-connection. A filtration 
\[\cE(\alpha)=\cA^0(\alpha)\supset \cA^1(\alpha)\supset\cdots\supset \cA^\ell(\alpha)\supset 0\]
is called Griffiths transverse (with respect to $\nabla$) if $\nabla(\cA^j(\alpha))\subset \cA^{j-1}(\alpha)\otimes K(D)$ for all $j.$ 
For example, if $\cA^1(\alpha)$ is the maximal destabilizing subbundle of $\cE(\alpha)$, then $\cA^1(\alpha)\subset\cA^0(\alpha)$ is a Griffiths transverse filtration.  Given such a filtration, the associated graded is 
\[Gr_\cA(\cE)=\bigoplus_{j=1}^\ell \cA_j(\alpha)~\ \ \ \text{where} \ \ \ \cA_j(\alpha)=\cA^j(\alpha)/\cA^{j+1}(\alpha).\]
The parabolic connection $\nabla$ induces a parabolic morphism $\phi_j:\cA_j(\alpha)\to\cA_{j-1}(\alpha))\otimes K(D),$ and hence defines a system of Hodge bundles 
\begin{equation}\label{eq ass Hodge bundle}
        (Gr_\cA(\cE)(\alpha),\Phi_\cA)=\Big(\cA_\ell(\alpha)\oplus\cdots\oplus\cA_0(\alpha)~ , ~ \begin{pmatrix}
    0&\\\phi_{\ell}&0\\&\ddots&\ddots\\&&\phi_{1}&0
\end{pmatrix}\Big).
\end{equation}
As in \cite[Lemma 4.1]{SimpsonDeRhamStrata}, we have the following proposition. 
\begin{Proposition} \label{prop:simpson-griffiths}
       If the system of Hodge bundles $(Gr_\cA(\cE)(\alpha),\Phi_\cA)$ from \eqref{eq ass Hodge bundle} is semistable, then 
       \[\lim_{\xi\to0}[\xi\lambda,\cE(\alpha),\xi\nabla)]=[Gr_\cA(\cE)(\alpha),\Phi_\cA]~.\] 
\end{Proposition}

If the parabolic system of Hodge bundles \eqref{eq ass Hodge bundle} is not semistable, let $(\bigoplus\hat \cA_j(\alpha),\Phi_{\hat\cA})$ be the maximal destabilizing system of Hodge bundles. 
Define the following two invariants of a Griffiths transverse filtration $\cA^\bullet$ whose associated system of Hodge bundles is not semistable: 
\[\xymatrix{\zeta(\cA^\bullet)= \mu(\bigoplus_j\hat\cA_j(\alpha))&\text{and}&\eta(\cA^\bullet)=\rk(\bigoplus_j\hat\cA_j(\alpha))~.}\]
Following \cite{SimpsonDeRhamStrata}, define a new filtration $\cB^\bullet$ of $\cE$ by \[\cB^j=\ker(\cE\to (\cE/\Aa^j)/\hat\cA_{j-1})\]
This new filtration is again Griffiths transverse, and the summands $\cB_j(\alpha)=\cB^j(\alpha)/\cB^{j+1}(\alpha)$ of the associated graded  fit in an exact sequence
\begin{equation}
        \label{eq exact sequence of gr}0\to \cA_j(\alpha)/\hat\cA_j(\alpha)\to \cB_j(\alpha)\to \hat\Aa_{j-1}(\alpha)\to 0.
\end{equation}
If the parabolic system of Hodge bundles $(Gr_\cB(\cE)(\alpha),\Phi_\cB)$ is semistable, then we have identified the limit $\xi\to 0$, if it is not, then the process can be repeated to obtain a new system of Hodge bundles. 
Simpson's key observation is that $(\zeta,\eta)$ decreases at each step of the above iterative process. 

\begin{Proposition}\label{Prop: decreasing inv}
        Suppose $(Gr_\cB(\cE)(\alpha),\Phi_\cB)$ is not semistable and let $\bigoplus_j\hat\cB_j$ be the maximal destabilizing system of Hodge bundles. Then
        \begin{enumerate}
                \item $\zeta(\cB^\bullet)\leq \zeta(\cA^\bullet)$,
                \item if $\zeta(\cB^\bullet)= \zeta(\cA^\bullet)$, then $\eta(\cB^\bullet)\leq \eta(\cA^\bullet)$, and
                \item if $\zeta(\cB^\bullet)\leq \zeta(\cA^\bullet)$ and $\eta(\cB^\bullet)\leq \eta(\cA^\bullet)$, then $\hat\cB_j(\alpha)\cong\hat\cA_{j-1}(\alpha)$ for all $j$.
        \end{enumerate}
\end{Proposition}
\begin{proof}
     For the nonparabolic setting, see \cite[\S2.2.1]{PengfeiThesis}. The only changes that need to be made are to consider all objects and destabilizing subobjects in the parabolic category.
\end{proof}

\begin{Proposition}\label{prop Griff trans filtration}
Let $(\cE(\alpha),\nabla)$ be a semistable parabolic $\lambda$-connection, then there exists a Griffiths transverse filtration such that the associated parabolic system of Hodge bundles is semistable. In particular, the limit $\lim_{\xi\to0}[\xi\lambda,\cE(\alpha),\xi\nabla)]$ exists. 
\end{Proposition}
\begin{proof}
We will show that the semistability assumption implies that the iterative process described above terminates. By Proposition \ref{Prop: decreasing inv}, the invariants $(\eta,\zeta)$ decrease in lexicographically at each step. We will show $(\eta,\zeta)$ can only take finitely many values and that semistability implies $(\eta,\zeta)$ can remain unchanged by only finitely many consecutive step of the process. 

The invariant $\eta$ can clearly only take finitely many values. To see that $\zeta$ can also only take finitely many values, note that $\zeta$ is bounded above by the slope of the maximal destabilizing subbundle of $\cE(\alpha)$ and bounded below by the slope of $\cE(\alpha)$. Since the parabolic weights are fixed, the parabolic degree of the maximal destabilizing subbundle can only take finitely many values, and hence $\zeta$ can take only finitely many values. 

Let $\cA^\bullet$ be a Griffiths transverse filtration of $\cE$, suppose the associated system of Hodge bundles $(\bigoplus_j \cA_j,\Phi_\cA)$ given by \eqref{eq ass Hodge bundle} is unstable. Consider the new system of Hodge bundles $(\bigoplus_j\cB_j,\Phi_\cB)$ given by \eqref{eq exact sequence of gr}. Let $\bigoplus_j\hat\cA_{j}(\alpha)$ and $\bigoplus_j\hat\cB_{j}(\alpha)$ be the maximal destabilizing subobjects. 
Suppose $\eta(\cB^\bullet)= \eta(\cA^\bullet)$ and $\zeta(\cB^\bullet)= \zeta(\cA^\bullet)$. Then $\hat\cB_j(\alpha)\cong\hat\Aa_{j-1}(\alpha)$ by part (3) of Proposition \ref{Prop: decreasing inv}. Hence, the exact sequence \eqref{eq exact sequence of gr} splits and 
\[\cB_j(\alpha)\cong\Aa_j(\alpha)/\hat\Aa_j(\alpha)\oplus \hat\Aa_{j-1}(\alpha)\]
with maximal destabilizing subobject $\hat\cA_{j-1}.$ In particular, the maximal destabilizing subobject is shifted to the left 1-step in the grading \eqref{eq ass Hodge bundle}. If the invariants $(\eta,\zeta)$ remain unchanged for infinitely many consecutive steps, the associated grade will eventually be of the form
\[\cE_k\oplus\cE_{k-1}\oplus\cdots \oplus\cE_{k-\ell}\oplus \cE_{k-\ell-2}\oplus\cdots\oplus\cE_0,\]
where $\bigoplus\hat\cA_j\cong\cE_k\oplus\cE_{k-1}\oplus\cdots \oplus\cE_{k-\ell}$. Since there is a gap in the grading, the $\phi_{k-\ell-1}$ term in the Higgs field vanishes. By Griffiths transversality, this means the maximal destabilizing subbundle is $\nabla$-invariant, contradicting $(\cE(\alpha),\nabla)$ being semistable. 
\end{proof}
\end{appendix}

 \bibliography{mybib}
 \bibliographystyle{plain}
 \end{document}